\documentclass[14pt, a4paper]{article}  % Usamos la clase básica article.
\usepackage[utf8]{inputenc}  % Codificación de caracteres en UTF-8.
\usepackage{amsmath}         % Paquete para matemáticas avanzadas.
\usepackage{amsfonts}        % Fuentes matemáticas.
\usepackage{amssymb}         % Símbolos matemáticos adicionales.
\usepackage{graphicx}        % Para incluir imágenes y gráficos. 
\usepackage{hyperref}        % Para los hipervínculos.
\usepackage{geometry}        % Para personalizar los márgenes.
\usepackage{fancyhdr}        % Para personalizar los encabezados.
\usepackage{times}           % Para usar la fuente Times New Roman.
\usepackage{amsthm} 
\usepackage[T1]{fontenc}
\usepackage{adjustbox}
\usepackage{mathrsfs}
\usepackage[all]{xy}

\usepackage{amsmath}
\newcounter{property}

\providecommand{\corchete}[1]{\left[#1\right]}

\providecommand{\llave}[1]{\left\lbrace #1\right\rbrace }
\providecommand{\pare}[1]{\left( #1\right) }
\providecommand{\norm}[1]{\left\lVert#1\right\rVert}

\providecommand{\re}[1]{\text{Re}\pare{#1}}
\providecommand{\im}[1]{\text{Im}\pare{#1}}

\providecommand{\abs}[1]{\left\lvert#1\right\rvert}
\providecommand{\Sum}[1]{\displaystyle\sum#1}

\newcommand{\var}{\varphi}

\newcommand{\tien}{\rightarrow}

\newcommand{\MW}{M_{B}\pare{V,W}}
\newcommand{\mw}{M}

\newcommand{\R}{\mathbb{R}}

\newcommand{\K}{\mathbb{K}}

\newcommand{\N}{\mathbb{N}}
\newcommand{\BK}{\text{BK}}

\newcommand{\BKn}{\text{BKN}}
\newcommand{\BKa}{\text{BKN}}

\newcommand{\BB}{\mathscr{B}}

\newcommand{\C}{\mathbb{C}}

\newcommand{\B}{\textup{B}}

\newcommand{\F}{\mathcal{F}}

\newcommand{\la}{\left\langle}
\newcommand{\ra}{\right\rangle}
\newcommand{\vs}{\vskip 1em}
\newcommand{\qs}{\setlength{\belowdisplayskip}{0pt}}

	\providecommand{\ral}[1]{\la#1\ra}

\newcommand{\V}{\mathcal{V}}
\newcommand{\W}{\mathcal{W}}

\newcommand{\J}{\mathrm{J}}

\newcommand{\G}{\textup{Gen}\hspace{-.1cm}}
\newcommand{\Mt}{M'}

\newtheorem{corollary}{Corollary}[section]
\newtheorem{proposition}{Proposition}[section]
\newtheorem{theorem}{Theorem}[section]
\newtheorem{lemma}{Lemma}[section]
\newtheorem{remark}{Remark}[section]

\newtheorem{example}{Example}[section]  % Numerar los ejemplos por sección
\newtheorem{definition}{Definition}[section]

\title{$B$-Multiplier Spaces}  % Aquí defines el título del artículo.
\author{
	Rafael Correa-Morales\\
	  \normalsize{Centro de Investigación en Matemáticas}\\
	\normalsize{\sf rafael.correa@cimat.mx}\\
	\quad \\
	{  Fernando Galaz-Fontes}\\
	\normalsize{Centro de Investigación en Matemáticas}\\
	\normalsize{\sf galaz@cimat.mx}\\
}

\date{}  % Deja la fecha vacía si no deseas incluirla o personalízala aquí.

\begin{document}
	
	\maketitle

\begin{abstract}
	We develop a general framework for $B$-multiplier spaces; these are vector spaces $M = M(V,W)$ obtained from a bilinear operator $B \colon M \times V \longrightarrow W$, where $V$ and $W$ are Banach spaces. We focus on their normability and completeness, mainly in the setting of spaces consisting of functions with values in a Banach space $X$, particularly sequences. Our approach relies on the underlying Banach spaces satisfying the $\BK$ property, that is, having continuous evaluations. Classical multiplier spaces arise when $B$ is a pointwise product of scalar functions and we make the point for the case when $V$ or $W$ consists of vector functions. Special attention is given to the sequence spaces $\Sigma \ell_{\infty}(X)$ (bounded partial sums), $\Sigma c(X)$ (summable), and $\ell_u(X)$ (unconditionally summable). Given a $\BK$ scalar sequence space $V$, we introduce the multiplier space $M_{\Sigma}(V,X)$ and establish conditions under which it determines a closed subspace of bounded linear operators from $V$ into $X$. The notion of associate space is precised for $\BK$-spaces, linking this construction with classical K\"othe duality. We consider what we named strong vectorialization $Y(X)$ and weak vectorialization $Y_w(X)$ of a Banach sequence ideal $Y$. The weak vectorialization is obtained as a multiplier space and employed to describe classical sequence spaces as $\ell_{p,w}(X)$, $1 \leq p \leq \infty$. We introduce the $b$-ideal part of a space and show that the $b$-ideal part of $\Sigma c(X)$ is $\ell_u(X)$ and that of $\Sigma \ell_{\infty}(X)$ is $\ell_{1,w}(X)$. We also study the sequence space $bv(X)$ (bounded variation), proving that $M(\Sigma c(X),\Sigma c(X)) = bv(\mathbb K)$.
\end{abstract}

\

	

	\section{Preliminaries and Introduction}\label{pop1}
	Let $D$ be a non-empty set and $F\pare{D,\K}$ the vector space of functions $f:D\tien \K$, where  $\K=\R$ is the field of real numbers, or $\K=\C$ is the field of complex numbers.  Given two vector subspaces  $V,W\subseteq  F\pare{D,\mathbb{K}}$,  let
	\begin{equation}
	M\pare{V,W}=\llave{f\in F\pare{D,\mathbb{K}}:f\cdot v\in W,\ \forall v\in V},
	\end{equation}
	where $f\cdot v : D \tien \K$ is  defined pointwise by  $f\cdot v\pare{d}: = f\pare{d}v\pare{d},\forall d\in D$.
	Then $M\pare{V,W}$ is a vector subspace of $F\pare{D,\mathbb{K}}$,  and   its  elements are known as the multipliers of $V$ into $W$.   In this paper, we consider a bilinear operator $B:M\times V\tien W$ instead of the pointwise product between functions and study the extended concept of multiplier that arises. This approach will enable us to give a unified treatment of various types of Banach spaces.

Let  $M$, $V$ and $W$ be vector spaces. We will say that $M$ is a \emph{multiplier space of $V$ into $W$}, if there exists a bilinear operator \begin{equation}\label{piz6}
B:M\times V\tien W.
\end{equation} In this case, to express that $f\in M$ we will say that $f$ is a \emph{$B$-multiplier.}  For each  $f\in M$, we define the function $B_{f}:V\tien W$ by \begin{equation}\label{piz5}
B_{f}v:=B\pare{f,v}.
\end{equation}
Clearly, $B_f$ is a linear operator. We will call an operator of this type a \emph{$B$-multiplication operator of $V$ into $W$}.

Additionally, suppose that $V$ and $W$ are normed spaces and let us consider the space of bounded linear operators $\mathcal{L}\pare{V,W}$. We are interested in determining whether $B_{f}\in \mathcal{L}\pare{V,W},\forall f\in M$.  This leads us to introduce the function $\norm{\cdot}_{M}:M\to [0,\infty]$
defined by
\begin{equation}\label{piz1}
\norm{f}_{M}:= 
\norm{B_{f}}=\sup\llave{\norm{B\pare{f,v}}_{W}:\norm{v}_{V}\leq 1}\in [0,\infty],\qquad \forall f\in M. 
\end{equation}
Therefore, for $f \in M$, we have   \begin{equation}\label{piz9}
\sup\llave{\norm{B\pare{f,v}}_{W}:\norm{v}_{V}\leq 1}<\infty\quad \text{ if and only if }\quad  B_{f}\in \mathcal{L}\pare{V,W}.  
\end{equation}

	The following inequality, which we will call the \emph{$M$-inequality},    is reminiscent of Hölder's inequality for $L^{p}$-spaces. Its proof is direct, as so is that of the lemma that follows.
\begin{proposition}\label{piz2_2} \label{q8_4} Let $V$ and $W$ be normed spaces, and  $B:M\times V \tien W$ be a bilinear operator. Then
	\begin{equation}\label{piz12}
	\norm{B\pare{f,v}}_{W}\leq \norm{f}_{M}\norm{v}_{V},\qquad \forall f\in M,\qquad  \forall v\in V.
	\end{equation} 
\end{proposition}

\begin{lemma}\label{piz7_0}
	Let $V$ and $W$ be normed spaces, and $B:M\times V \tien W$ be a bilinear operator. Then: 
	\begin{enumerate}
		\item[i.] $\norm{\alpha f}_{M}=\abs{\alpha}\norm{f}_{M},\qquad \forall f\in M,\quad \forall \alpha\in \K$.
		\item[ii.] $\norm{f+g}_{M}\leq \norm{f}_{M}+\norm{g}_{M},\qquad \forall f,g\in M$.
	\end{enumerate}
	Therefore,  $\norm{\cdot}_{M}$ is a seminorm on $M$ if, and only if $B_{f}\in \mathcal{L}\pare{V,W},  \forall f\in M.$
\end{lemma}

In the case that  $\norm{\cdot}_{M}$ is a norm, we will call it the \emph{operator norm}, and we will say that the multiplier space $M$ is \emph{normed (or $B$-normed, if convenient)}.  If, in addition, the operator norm is complete, we will say that  $M$ is \emph{complete (or $B$-complete, if convenient)}. \vs

 Let $M$ and $V$ be normed spaces and $B:M\times V \tien W$ be a bilinear operator. Recall that $B$ is \emph{bounded} if there exists $d \in [0,\infty)$ such that  
\begin{equation}\label{lola34}
	\sup\llave{\norm{B\pare{f,v}} : v \in \B_{V}}
	\leq d\norm{f}_{M} < \infty,
	\qquad \forall f \in M.
\end{equation}

In our development, we begin with a bilinear operator $B : M \times V \tien W$, and we usually assume that $V$ and $W$ are normed spaces. However, we do not initially assume that $M$ carries a norm.  In the case where $M$ is $B$-normed, the $M$-inequality indicates that the bilinear operator 
$B$ is bounded. 

 Suppose that for the bilinear operator $B:M\times V\tien W$ we have that  $B_{f}\in \mathcal{L}\pare{V,W},  \forall f\in M.$ This defines a linear operator $R:M\tien \mathcal{L}\pare{V,W}$ given by \begin{equation}
Rf:=B_{f},
\end{equation} which we will call the \emph{representation operator}. Notice that  $R$ is injective if, and only if

\begin{equation} 
	\text{for each $f\in M$ such that $B\pare{f,v}=0,\forall v\in V$, it holds that $f=0$.}  
	\tag{Property I}\label{lola76_0}
\end{equation}

\begin{proposition}\label{piz4_2}
	Let $V$ and $W$ be normed spaces, and $B:M\times V \tien W$ be a bilinear operator. Then  $M$ is  $B$-normed if, and only if
	\begin{enumerate}
		\item[i.] $B_{f}\in \mathcal{L}\pare{V,W},  \forall f\in M.$
		\item[ii.] The representation operator $R:M\tien \mathcal{L}\pare{V,W}$ is injective. 
	\end{enumerate}   In this case,   $R$ is a linear isometry, that is  $\norm{f}_{M}=\norm{B_{f}}_{\mathcal{L}\pare{V,W}},\forall f\in M$.
\end{proposition}
\begin{proof}
	Suppose that $\norm{\cdot}_{M}$ is a norm. Let $f\in M$. Then $\norm{f}_{M}<\infty$ and by   (\ref{piz9}), it follows that $B_{f}\in \mathcal{L}\pare{V,W}$.  On the other hand, let  $f\in M$ such that $B_{f}=0$. Since $B\pare{f,v}=B_{f}v=0$, $\forall v\in V$, it follows that $\norm{f}_{M}=0$. Using that   $\norm{\cdot}_{M}$ is a norm, we conclude that  $f=0$.  
	
	Now, suppose that conditions i and ii are satisfied. Since $B_{f}\in \mathcal{L}\pare{V,W},  \forall f\in M$, by lemma \ref{piz7_0} it follows that $\norm{\cdot}_{M}$ is a seminorm. Let $f\in M$ such that $\norm{f}_{M}=0$. Then $B\pare{f,v}=0,\forall v\in V$. Therefore, $Rf=B_{f}=0$, and since $R$ is injective, we conclude that  $f=0$.  
\end{proof}

In general, the discussion of the continuity of $B$-multiplication operators, the normability  of the space of $B$-multipliers and its completeness,   are the points which  interests us in this paper. The following example illustrates a situation that can occur. To present it we give first some definitions. \vs 

	 Throughout we will employ the following standard definitions. As usual,   the \emph{unit ball} of a normed space $V$ is denoted by $\B_{V}$.   Let  $\pare{V,\norm{\cdot}_{V}}$ and $\pare{W,\norm{\cdot}_{W}}$ be normed spaces such that  $W\subseteq V$. To express that there exists  $c\in \pare{0,\infty}$ such that $\norm{v}_{V}\leq c \norm{v}_{W},\forall v\in W$, we write $W \hookrightarrow_{c} V$. When  $c=1$,  we use  $W \hookrightarrow V$ and say that $W$ is \emph{continuously embedded} in $V$.  To indicate that  $V=W$ and $\norm{\cdot}_{V}=\norm{\cdot}_{W}$ we simply write $V\equiv W$.  \vs

\begin{example}
	Let  $M, V$ and $W$ be normed spaces, and let $B:M\times V\tien W$ be a bounded bilinear operator.
	In this case, all $B$-multiplication operators from $V$ to $W$ are bounded. Now suppose that $M$ is $B$-normed, and let $M_{B}$ denote the space $M$ equipped with its operator norm. From  (\ref{lola34}) it follows that $M \hookrightarrow_{d} M_{B}$. Hence, if $M$ is complete, then $M_{B}$ is complete if and only if the original norm in $M$ and the operator norm are equivalent. In terms of $B$, this means that there exists a constant $c>0$ such that
	\begin{equation}
			\|f\|_{M} \leq c \, \sup \{ \| B(f,v) \|_{W} : v \in B_{V} \}, \qquad \qquad \forall f \in M.
	\end{equation}
\end{example}

In important cases, from a single bilinear operator, we can  derive a wide variety of multiplier spaces. For example, let  $\F$, $\V$ and $\W$  be vector spaces, and suppose we are given a bilinear operator \begin{equation}\label{efren1}
	B:\F\times \V\tien \W.
\end{equation}  For subspaces $V\subseteq \V$ and  $W\subseteq \W$, we define
\begin{equation} 
	\MW:=\llave{f\in \F:B\pare{f,v}\in W,\ \forall v\in V}.
\end{equation}   
Then,  $\MW$ is a vector subspace of $\F$ and the restriction   
\begin{equation}
	\text{$B_{V, W}:=B:\MW\times V\tien W$}
\end{equation}
is a bilinear operator. Thus, $M_{B}\pare{V,W}$ is the vector space of  $B_{V,W}$-multipliers of $V$ in $W$. Additionally, suppose that $V$ and $W$ are normed spaces. If this is the case, when referring to  $M_{B}\pare{V,W}$ as a normed space, we will always mean that it has corresponding operator norm  $\norm{\cdot}_{M}$.  \vs

We next give two simple results that show how the associated multiplier spaces behave when the involved spaces are varied. 

\begin{lemma}\label{q10_0_b} \label{piz26}
		Let $V,V'\subseteq \V$ and $W,W'\subseteq \W$ be vector spaces.
	\begin{enumerate}
		\item[i.] If $W'\subseteq W$, then $M_{B}\pare{V,W'}\subseteq M_{B}\pare{V,W}$. 
		\item[ii.] If $V\subseteq V'$, then $M_{B}\pare{V',W}\subseteq M_{B}\pare{V,W}$. 
	\end{enumerate} 
\end{lemma}

Next, we distinguish a class of bilinear operators that take into account the multiplicative structure of bounded linear operators. To this end, we will assume that $\W=\V$ in (\ref{efren1}) and take $\F$ to be an algebra.

\begin{definition}
	Let  $\mathcal{F}$ be an algebra,   $\mathcal{V}$ a vector space, and $B:\F\times \mathcal{V}\tien   \mathcal{V}$ a bilinear operator. We say that   $B$ is \emph{multiplicative}  if \begin{equation}\label{lola7}
		B\pare{g  f,v}=B\pare{g,B\pare{f,v}},\qquad \forall f, g\in \F,\qquad \forall v\in V.
	\end{equation}
	Naturally, in the above expression, $gf$ denotes the product of $g$ and $f$ in $\F$.
\end{definition} 
Note that a bilinear operator $B$ is multiplicative if and only if
\begin{equation}\label{lola1}
	B_{gf}=B_{g}\circ B_{f},\qquad \forall f,g\in \F.
\end{equation}

Let  $\mathcal{F}$ be an algebra,  $\mathcal{V}$ a vector space, and $B:\F\times \mathcal{V}\tien   \mathcal{V}$ a multiplicative bilinear operator. Let  $V,W,Z\subseteq \mathcal{V}$ be vector spaces. Observe that if $f\in M_{B}\pare{V,W}$, $g\in M_{B}\pare{W,Z}$ and $v\in V$, then  $B\pare{f,v}\in W$, and hence $B\pare{gf,v}=B\pare{g, B\pare{f,v}}\in Z$. Since  $gf\in \F$, we conclude that  $ gf\in M\pare{V,Z}$.  Therefore, 
\begin{equation}\label{piz25_1}
	M_{B}\pare{W,Z} M_{B}\pare{V,W}  \subseteq M_{B}\pare{V,Z}.
\end{equation}

	Suppose now that $V, W, Z$ are normed spaces and $M_{B}\pare{V,W}$, $M_{B}\pare{W,Z}$, $M_{B}\pare{V,Z}$ are also normed. Next, we denote by $R_{V,W}$, $R_{W,Z}$ and $R_{V,Z}$  the representation operators of $M_{B}\pare{V,W}$, $M_{B}\pare{W,Z}$ and $M_{B}\pare{V,Z}$, respectively. Given   $f\in M_{B}\pare{V,W}$ and $g\in M_{B}\pare{W,Z}$, it follows that $R_{V,Z}\pare{gf}=\pare{B_{V,Z}}_{gf}=\pare{B_{W,Z}}_{g}\circ \pare{B_{V,W}}_{f}=R_{W,Z}\pare{g}\circ R_{V,W}\pare{f}.$ Thus,
	\begin{equation}\label{lola68}
		R_{V,Z}\pare{gf}=R_{W,Z}\pare{g}\circ R_{V,W}\pare{f}.
	\end{equation}
	
	Now, from  (\ref{lola7}) and the $M$-inequality, we obtain that \begin{equation}
		\norm{B\pare{gf,v}}_{Z}\leq \norm{g}_{M_{B}\pare{W,Z}}\norm{B\pare{f,v}}_{W} 
		\leq \norm{g}_{M_{B}\pare{W,Z}} \norm{f}_{M_{B}\pare{V,W}},\qquad \forall v\in \B_{V}.
	\end{equation}
	Therefore, \begin{equation}\label{d345}
		\norm{g  f}_{M_{B}\pare{V,Z}}\leq \norm{g}_{M_{B}\pare{W,Z}} \norm{f}_{M_{B}\pare{V,W}},\qquad \forall f\in M_{B}\pare{V, W},\quad g\in M_{B}\pare{W,Z}.
	\end{equation}

	Next, consider the case $V=W=Z$ and define $M_{B}\pare{V}:=M_{B}\pare{V,V}$. From  (\ref{piz25_1}) it follows that $M_{B}\pare{V}$ is a subalgebra of $\F$, and from (\ref{lola1}), the representation operator $R:M_{B}\pare{V}\tien \mathcal{L}\pare{V}$  satisfies \begin{equation}
		R\pare{gf}=Rg\circ Rf,\qquad   \forall f,g\in M_{B}\pare{V}.
	\end{equation} 
	That is, $R$ is an algebra homomorphism. From (\ref{piz25_1}) and (\ref{d345}) we finally obtain the following result.\vs

	\begin{proposition}\label{naty24}
		Let  $\mathcal{F}$ be an algebra,  $\mathcal{V}$ a vector space, $B:\F\times \mathcal{V}\tien   \mathcal{V}$ a bilinear operator, and $V\subseteq \mathcal{V}$ a Banach space. If $B$ is multiplicative and  $M_{B}\pare{V}$ is normed, then $M_{B}\pare{V}$ is a normed algebra, and the representation operator $R:M_{B}\pare{V}\tien \mathcal{L}\pare{V}$ is an algebra isomorphism onto its range. Moreover, if $\F$ has an identity element, then $M_{B}\pare{V}$ also has an identity element; and if $\F$ is commutative, then $M_{B}\pare{V}$ is commutative. 
	\end{proposition} 
	
 			\begin{remark}
 		If $\|\cdot\|_M$ is just a seminorm, note that  $M_B(V)$ is a subalgebra and $R\colon M_B(V) \tien \mathcal{L}(V)$ is an algebra homomorphism. 
 	\end{remark}

	Next, we will prove a type of reciprocal for Proposition \ref{naty24} : namely, if the range of $R$  is a subalgebra of $\mathcal{L}\pare{V}$, then it is possible to define an operation on  $M_{B}\pare{V}$ that makes it an algebra.
	
	Let us assume that the range of $R$ is a subalgebra of 
	$\mathcal{L}\pare{V}$. This is equivalent to saying that for each $f,g\in M_{B}\pare{V}$, there exists $h\in  M_{B}\pare{V}$ such that $Rf\circ Rg=Rh.$ This naturally takes us to define the operation $\#:M_{B}\pare{V}\times M_{B}\pare{V}\tien M_{B}\pare{V}$ by  
	\begin{equation}\label{naty38}
		f\# g:=R^{-1}\pare{Rf\circ Rg}. 
	\end{equation}
	Using  (\ref{naty38}) together with the linearity and injectivity of $R$, we can directly establish the   properties to conclude that $M_{B}\pare{V}$ is a subalgebra of $\mathcal{L}\pare{V}$. For example, let  $f,g,h\in M_{B}\pare{V}$, then \begin{align*}
		\pare{f\# g}\# h&=R^{-1}\pare{R\pare{f\#g}\circ Rh}=R^{-1}\pare{\pare{Rf\circ Rg}\circ Rh}\\[.3 cm]
		&=R^{-1}\pare{Rf\circ \pare{Rg\circ Rh}}=R^{-1}\pare{Rf\circ R\pare{g\# h}}=f\#\pare{g\# h}.
	\end{align*}
	Thus, $\#$ is associative.  Additionally,
	\begin{equation}
		\norm{f\# g}_{M_{B}\pare{V}}=\norm{R\pare{f\# g}}_{\mathcal{L}\pare{V}}=\norm{Rf\circ Rg}_{\mathcal{L}\pare{V}} \leq \norm{f}_{M_{B}\pare{V}}\norm{g}_{M_{B}\pare{V}}.
	\end{equation}
	From above, we immediately obtain the following result:
	
	\begin{proposition}\label{naty39} Let  $B:\F\times \mathcal{V}\tien   \mathcal{V}$ be a bilinear operator and   $V\subseteq \V$ a Banach space such that $M_{B}\pare{V}$ is normed. If the range of the representation operator  $R:M_{B}\pare{V}\tien \mathcal{L}\pare{V}$ is a subalgebra of $\mathcal{L}\pare{V}$, then    $M_{B}\pare{V}$, with the operation  $\#$ defined in equation (\ref{naty38}),  is a normed algebra, and     $R$ is an algebra isomorphism onto its range.  
	\end{proposition}
	
	Having presented the preliminaries, we next describe the remaining contents of this paper.
	
		In Section~\ref{pop2}  we study conditions under which a multiplier space $M$ is normed with
		respect to the natural function $\norm{\cdot}_{M}$. Our results  deal  mainly with BK-spaces, that is  Banach spaces consisting of functions and where the evaluations are continuous. The analysis is based on pointwise convergence in the underlying  spaces and on continuity properties of the associated multiplication operators; see, for example, Corollary~\ref{piz11_b_0} .  
The developments in Section~\ref{pop3},  examine the completeness of multiplier spaces with respect to the operator norm $\norm{\cdot}_{M}$. The discussion relies on the pointwise behavior of Cauchy sequences in $M$, on the stability of the underlying bilinear action under pointwise limits, and on the assumption that $M$ has continuous evaluations. See, for example, Theorem~\ref{d265_b_0} . Definition \ref{d1} introduces a property under which  $M$ is a normed space and has continuous evaluations.
	
	We consider a pointwise bilinear operator $B_{P}$ induced by a bounded bilinear
	mapping $P$ between Banach spaces. For suitable function spaces $V$ and $W$, we
	study the associated multiplier space $M_{B_{P}}(V,W)$. This setting, developed
	in Section~\ref{pop4}, extends the classical theory of scalar pointwise multipliers to a
	vector-valued framework and provides a unified approach for the developments of sections \ref{pop5} and \ref{pop6} .   Particularly, Theorem~\ref{lola28_3} states that if $V \subseteq F(D, X_{2})$ and $W \subseteq F(D, X_{3})$ are $\BK$-spaces and $P$ does not $\pare{V,W}$-vanish at any point, then $M_{B_{P}}(V, W)$ is a $\BK$-space.

	We next take $P$ as the scalar product $P_{X} \colon \K \times X \to X$ and use it to obtain a pointwise bilinear operator $\BB$ acting between spaces of functions. The associated multiplier spaces $M(V,W)$ arise from this induced action and describe the pointwise multiplication of scalar-valued functions on vector-valued functions. This perspective, considered in Section~\ref{pop5}, allows us to recover classical results  by taking $X = \K$. Assuming that $V \subset F(D,X)$ is a BKN-space,   Theorem~\ref{lola23} shows that the multiplier space $M(V,V)$ is a commutative Banach algebra with unity and  consists of bounded functions. Employing multiplier ideas, if   $U \subseteq F(D,X)$ is a BK-space under two norms, we conclude that these norms are equivalent. This result, with a different proof,  was already given by A.~Wilansky in \cite[p.~55]{Wil}.

   Let $B \colon \mathcal F \times \mathcal V \to \mathcal W$ be a   bilinear mapping and let
   $B' \colon \mathcal V \times \mathcal F \to \mathcal W$ be given by $B'(v,f) := B(f,v)$.
   The multiplier spaces associated with $B'$ are the main object of Section~\ref{pop6}.
   In the pointwise setting, this construction leads to multiplier spaces consisting of vector-valued functions,  denoted by $M'(V,W)$. From Theorem \ref{piz27_1_0} we can appreciate that for $M'(V,W)$  we have a situation similar to that of $M(V,W)$.

  Throughout sections~\ref{pop7} and \ref{pop8}, we work with b-ideals and v-ideals, which  are extensions of the notions of order ideal and Banach ideal (defined for spaces of scalar-valued functions) to spaces of vector-valued functions. These extensions were  studied by M.~Väth  in \cite{V}, where the name ideal* is used for a Banach b-ideal and   a Banach v-ideal is simply called ideal. We refer to \cite{V}  and the references therein for more interesting information  on these topics.  In Section~\ref{pop7}, for a given space $U \subseteq F(D,X)$, we define its $b$-ideal part as the multiplier space $M'(B(D,X),U)$ and show that it is the largest b-ideal contained in $U$. 
   Next, in Section~\ref{pop8} , we analyze what we call the strong vectorialization $Y(X)$ associated with an order ideal $Y \subseteq S(\K)$. As can be seen from Theorem \ref{piz24}, this construction yields $\BK$ v-ideals, including classical examples such as $\ell_p(X)$, $\ell_\infty(X)$, and $c_0(X)$. (See \cite[p.~32]{DJT}, where these spaces are denoted by $\ell_{p}^{strong}(X)$, $\ell_{\infty}^{strong}(X)$, and $c_0^{strong}(X)$, respectively.)
   
   We examine the sequences spaces $\Sigma \ell_{\infty}(X)$, $\Sigma c(X)$, and $\ell_{u}(X)$,
   which arise as particular instances of a general construction $\Sigma E$ we do for 
   Banach spaces $E \subseteq S(X)$, where $S(X)$ denotes the space of sequences in $X$. In Section~\ref{pop9}, we study some of their fundamental
   properties through the behavior of canonical projection functions. In particular,
   we show that the $b$-ideal part of $\Sigma c(X)$ is equal to the space of unconditionally summable sequences $\ell_{u}(X)$. This is a reformulation of   a  well known result called the Bounded Multiplier Test \cite[p. 5, 22]{DJT}.

We introduce the multiplier space $M_{\Sigma}(V,X)$ associated with a BKN-space
$V \subseteq S(\K)$ and a Banach space $X$. As a vector space, $M_{\Sigma}(V,X)$
coincides with the  initial multiplier space $\Mt(V,\Sigma c(X))$, but the original bilinear action is replaced by a
summation-based representation  taking values in $X$ rather than in $\Sigma c(X)$,
which induces a natural renorming of the original multiplier space. Section~\ref{pop10} presents
conditions under which $M_{\Sigma}(V,X)$ is a $\BK$-space, together with norm
equivalence results with $\Mt(V,\Sigma c(X))$. The representation operator
$R_{V,X}$ plays here a central role, and its   isomorphism property is
characterized in terms of $V$ having a Schauder basis. 
Finally, $M_{\Sigma}(V,\K)$ is used to define the analogue of the associate space for order ideals, thus connecting this construction with the classical theory of Köthe duality (see \cite{CDS} and the references therein).

Given a Banach space $U \subseteq F(D,X)$, in the last section we introduce the weak vectorialization $U_{w}(D,X)\subseteq F\pare{D,X}$, which arises naturally from the theory of $B$-multiplier spaces when the underlying bilinear operator is given by scalarization. This construction
provides a unified framework for several classical Banach spaces of
vector-valued functions, and is obtained via dual pairings. The general
theory is then specialized to sequence spaces, leading to the weak
vectorialization $Y_{w}\pare{X}$ of a Banach ideal $Y$ and recovering well-known spaces such as
$\ell_{p,w}(X)$ and $c_{0,w}(X)$. (See \cite[p.~32]{DJT}, where these spaces are denoted by $\ell_{p}^{weak}(X)$ and $c_0^{weak}(X)$, respectively.) We show that the $b$-ideal part
of $\Sigma \ell_{\infty}(X)$ coincides   with the space of weakly
absolutely summable sequences $\ell_{1,w}(X)$. This is a reformulation of a well-known result due to  Bessaga and Pełczyński \cite[Lemma~2]{BP}. We establish that the space $Y_w(X)$ is always a $b$-ideal and that the space $\ell_{1,w} (c_0)$ is a $b$-ideal that is not a $v$-ideal.  We also study the space $bv\pare{X}$ consisting of vector sequences of bounded variation and present the weak vectorialization of $bv\pare{\K}$. Finally, we show that $M\pare{\Sigma \ell_{\infty}\pare{X},\Sigma \ell_{\infty}\pare{X}} = bv\pare{\K}=M\pare{\Sigma c\pare{X},\Sigma c\pare{X}}.$

Summing up, we see that sections \ref{pop1} to \ref{pop3} give the general theory for B-multipliers. Thereafter,  sections~\ref{pop4} through \ref{pop8} mainly deal with pointwise multipliers, whereas sections~\ref{pop10} and \ref{pop11} provide examples of $B$-multiplier Banach spaces in which the bilinear operator $B$ is not induced by a pointwise bilinear operator. Other examples are presented in subsections \ref{lalo2} and \ref{lalo3} .

\section{Normed Multiplier Spaces}\label{pop2}

Motivated by pointwise multiplication, we now   consider a class of bilinear operators $B$ and provide conditions on $V$, $W$, and $B$ under which the multiplier space $M$ is $B$-normed. These conditions are based on the assumption that pointwise convergence can be considered in the involved vector spaces. Therefore, we start by additionally requiring    that
\begin{equation}  W\subseteq F\pare{D,X},
\end{equation} 
where  $D$ is a non-empty set and $X$  is a Banach space.

		Before going on, we  recall a property  which will be fundamental for  this work.   For each $a\in D$, we define the operator $\delta_{a}:F\pare{D,X}\tien X$ by
	\begin{equation}
	\delta_{a}\pare{f}:=f\pare{a}.
	\end{equation}
	Note that $\delta_{a}$ is a linear operator, which we will call the \emph{evaluation operator at $a$}. We say that a normed space  $V\subseteq F\pare{D,X}$ has  \emph{continuous evaluations} if for each $a\in D$, the operator    $\delta_{a}:V\tien X$ is continuous. If in addition $V$ is complete, we say that $V$ is a \emph{\BK-space}. 	Banach spaces with continuous evaluations have been extensively studied, focusing however on scalar sequence spaces that are locally convex and Fréchet. For example, see \cite[Ch. 4]{Wil}, where additionally we can find  points close to our work.

\begin{theorem}\label{bobo}\label{piz21}
	Let $V$ and $W$ be Banach spaces, and let $B:M\times V\tien W$ be a bilinear operator. If $W\subseteq F\pare{D,X}$ is a $\BK$-space, then every $B$-multiplication operator is continuous if and only if 
	% Importamos el paquete adjustbox
	\begin{align} 
		&\text{for each sequence $\llave{v_{k}}_{k}\subseteq V$ and $v\in V$ such that $v_{k}\tien v$, we have} \nonumber\\[.3 cm]
		&\qquad \qquad \qquad \text{$B\pare{f,v_{k}} \tien B\pare{f,v}$ pointwise, $\forall f\in M$}.
		\tag{Property II} \label{lola86_0}
	\end{align}  
\end{theorem}
\begin{proof}
	First, suppose that $B_{f}\in \mathcal{L}\pare{V,W},\forall f\in V$.  Fix  $f\in M$, and let $\llave{v_{k}}_{k}\subseteq V$ and $v\in V$ be such that $v_{k}\tien v$. Then  $B\pare{f,v_{k}}=B_{f}v_{k}\tien B_{f}v=B\pare{f,v}.$ Now, since $W$ is a $\BK$-space, it follows that $B\pare{f,v_{k}} \tien  B\pare{f,v}$ pointwise.
	
	Now, suppose that II holds. Fix $f\in M$, and consider the multiplication operator    $B_{f}:V\tien W$. We will show that the graph of  $B_{f}$, denoted by $G$, is closed in $V\times W$. Thus,  by the closed graph theorem, we will have proven that $B_{f}$ is continuous. Let $\llave{v_{k}}_{k}\subseteq V$, $v\in V$ and $w\in W$ be such that $v_{k}\tien v$ and $B_{f}v_{k}\tien w$. Then, from (\ref{lola86_0}),   $B\pare{f,v_{k}}\pare{a}\tien B\pare{f,v}\pare{a}$, $\forall a\in D$. On the other hand, since $B_{f}v_{k}\tien w$ and $W$ has continuous evaluations, it follows that $B\pare{f,v_{k}}\pare{a}\tien w\pare{a},\forall a\in D$. Therefore, $B\pare{f,v}=w\in W$.  Hence, $G$ is closed.	  
\end{proof}

Under the conditions of Theorem \ref{bobo} , note that \ref{lola86_0} holds if and only if  $\|\cdot\|_M$ is a seminorm on $M$. The following results are clear.

\begin{corollary}\label{piz11_b_0}   Let $V$ and $W$ be Banach spaces with $W \subseteq  F(D,X)$, and let $B:M\times V\tien   W$ be a bilinear operator. If $W$ is a $\BK$-space, then $M$ is  $B$-normed if and only if $B$ satisfies properties I and II.
\end{corollary} 

\begin{lemma}\label{final11}
	Let $W \subseteq F(D,X)$ be a normed space. Then $W \subseteq B(D,X)$ with continuous inclusion if and only if there exists some real number $C >0$ such that 
	$\| f(a) \|_X \leq C \|f\|_W, \forall a \in D, \forall f \in W.$ 
	Clearly, in this case $W$ has continuous evaluations.
\end{lemma}

	The next example presents a wide class of BK-spaces. More examples will appear in sections \ref{lalo1} , \ref{sec9} and \ref{sec10} .

	  \begin{example}\label{lup1}
	Let $D$ be a non-empty set and $X$ a Banach space. We denote by   $B\pare{D,X}$ the space of bounded functions $f:D\tien X$. Then   $B\pare{D,X}$ with the \emph{supremum norm}  is a $\BK$-space.  When  $X=\K$, the function $\mathit{1}:D\tien \K$   defined by $\mathit{1}\pare{a}:=1,\forall a\in D$,   is bounded. In fact, $B\pare{D,\K}$ is a Banach algebra with $\mathit{1}$ as its unit.
	
	 In the case $D=\N$, we will use $S(X) := F\pare{\N, X}$ to denote the space of sequences in $X$, and $\ell_{\infty}\pare{X} := B\pare{\N, X}$.   We also introduce the space of sequences in $X$ that converge
	\begin{equation}
		c(X) := \Big\{ \llave{x_{n}}_{n} \subseteq X : \exists\, x \in X \text{ such that } \lim_{n \to \infty} \|x_n - x\| = 0 \Big\}.
	\end{equation}
	
	Similarly, we consider $c_0(X)$ as the space of all null sequences in $X$, namely,
	\begin{equation}
			c_0(X) := \llave{x = \llave{x_{n}}_{n} \in S(X) : \lim_{n \to \infty} x_n = 0}.
	\end{equation}
	
	Note that $c(X)$ and $c_0(X)$ are Banach subspaces of $\ell_{\infty}\pare{X}$, and therefore they are $\BK$-spaces.
	
\end{example} 

  For $1 \leq p < \infty$, we take
\begin{equation}
	\ell_p(X) := \Big\{ \llave{x_{n}}_{n} \subseteq X : \sum_{n=1}^{\infty} \|x_n\|^p < \infty \Big\}, \qquad
	\|\llave{x_{n}}_{n}\|_{\ell_p(X)} = \pare{ \sum_{n=1}^{\infty} \|x_n\|^p }^{1/p}.
\end{equation}
The map $\|\cdot\|_{\ell_p(X)}$ defines a norm on $\ell_p(X)$, and equipped with this norm, $\ell_p(X)$ forms a Banach space.  Note that  $\ell_{p}\pare{X}$, for $1 \leq p \leq \infty$, is a $\BK$-space. In fact, we have
\begin{equation}
	\|x_m \|_{X} \leq \norm{\llave{x_{n}}_{n}}_{\ell_p(X)},\qquad  \forall m \in \N.
\end{equation}

\section[Completion of multiplier spaces]{Completeness of Multiplier Spaces}\label{q21}\label{pop3}

Let $B: \F \times \V \to \W$ be a bilinear operator, and let $V \subseteq \V$, $W \subseteq \W$ be Banach spaces. Having established in Theorem \ref{bobo} a criterion for the multiplier space $M_{B}\pare{V,W}$ to be normed, we now assume that this is the case and study its completeness. In addition to being vector spaces, we will consider $\F$ and $\W$ as function spaces such that
\begin{equation}\label{piz13} 
\text{$\F\subseteq F\pare{D_{1}, X_{1}}$\qquad \text{and} \qquad  $\W\subseteq F\pare{D_{3},X_{3}}$},
\end{equation} where $D_{1},D_{3}$ are non-empty sets and $X_{1},X_{3}$ are Banach spaces. 

Our analysis of the completeness of $M_{B}\pare{V,W} \subseteq F\pare{D_{1},X_{1}}$ will rely on $M_B(V,W)$ having continuous evaluations  (in particular, it is a normed space). The following condition will then prove to be important.

\begin{definition}\label{d1} Suppose that $D_{1} = D_{3} = D$ in (\ref{piz13}). Let $V \subseteq \V$ and $W \subseteq \W$ be normed spaces. We say that $B$ \emph{does not  $\pare{V,W}$-vanish at any point}  if, for each $a\in D$, there exists $c_{a}\in (0,\infty)$ such that for each $f\in M_{B}\pare{V,W}$, there exists $v_{f}\in \B_{V}$  such that
	\begin{equation}\label{piz17} 
		\norm{f\pare{a}}_{X_{1}}\leq c_{a} \norm{B\pare{f,v_{f}}\pare{a}}_{X_{3}}.
	\end{equation}
\end{definition}

\begin{lemma}\label{d339_0}
	Suppose that $D_{1} = D_{3} = D$ in (\ref{piz13}). Let $V \subseteq \V$ and $W \subseteq \W$ be normed spaces such that $W$ has continuous evaluation  and  \ref{lola86_0} is satisfied. If $B$ does not  $\pare{V,W}$-vanish at any point, 
	then    $M_{B}\pare{V,W}$ is normed and  has continuous evaluations. 
\end{lemma}
\begin{proof} 
	Let  $f\in M_{B}\pare{V,W}$ such that $B\pare{f,v}=0,\forall v\in V$. Applying (\ref{piz17}), we deduce that $f=0$, which proves that property I holds. Since all multiplication operators are bounded, by Corollary \ref{piz11_b_0}  we conclude that    $M_{B}\pare{V,W}$ is normed.
	
	Next, let $a\in D$. Since we are assuming $W$  has continuous evaluations, there exists $d\in \pare{0,\infty}$ such that $\norm{w\pare{a}}_{X_{3}}\leq d\norm{w}_{W}, \forall w\in W. $
	Now, using  (\ref{piz17}) and the $M$-inequality, we obtain \setlength{\belowdisplayskip}{-8pt} 
	\begin{align*}
	\norm{f\pare{a}}_{X_{1}}& \leq c_{a} \norm{B\pare{f,v_{f}}\pare{a}}_{X_{3}}\\[.3 cm]
	& \leq c_{a}d\sup\llave{\norm{B\pare{f,v}}_{W}:v\in \B_{V}}= c_{a}d\norm{f}_{M},\qquad  \forall f\in M_{B}\pare{V,W}.
	\end{align*} 	 
\end{proof}

\begin{theorem} \label{d265_b_0} \label{d248_0_b}   \label{q15_b} Let  $\F\subseteq F\pare{D_{1}, X_{1}}$, $\W\subseteq F\pare{D_{3},X_{3}}$ and  $\V$ be vector spaces. Let $B:\F\times \V\tien   \W$ be a bilinear operator and let $V\subseteq \V$, $W\subseteq \W$ be Banach spaces such that $W$ and   $M_{B}\pare{V,W}$ have continuous evaluations. Then,  $M_{B}\pare{V,W}$ is a Banach space if, and only if 
		\begin{align} 
	&\text{for every Cauchy sequence $\llave{f_{k}}_{k}\subseteq M_{B}\pare{V,W}$ and $f\in F\pare{D_{1}, X_{1}}$ such that $f_{k}\tien f$ pointwise,} \nonumber \\[.3 cm]
	&\qquad \qquad \qquad \text{then $f\in \F$ and $B\pare{f_{k},v} \tien B\pare{f,v}$ pointwise,\quad $\forall v\in V$}. 
	\tag{Property III} \label{lola85_0}
	\end{align} 
\end{theorem}
\begin{proof} 
	First, assume that the multiplier space $M_{B}\pare{V,W}$ is complete. Let $\llave{f_{k}}_{k}\subseteq M_{B}\pare{V,W}$ be a Cauchy sequence and $f\in F\pare{D_{1},X_{1}}$ such that $f_{k}\tien f$ pointwise. Then, there exists $g\in M_{B}\pare{V,W}$ such that  $f_{k}\tien g$ in $M_{B}\pare{V,W}$. Since   $M_{B}\pare{V,W}$ has continuous evaluations and $f_{k}\tien f$ pointwise, it follows that $f=g\in M_{B}\pare{V,W}\subseteq \F$.
 	 
 	  Fix $v\in V$. Applying the $M$-inequality, we conclude that the linear operator $B\pare{\cdot,v}:M\tien W$ is continuous, and hence $B\pare{f_{k},v}\tien B\pare{f,v}$ in $W$. Since $W$ has continuous evaluations, this implies that $B\pare{f_{k},v} \tien B\pare{f,v}$ pointwise.
	
	Now, assume that \ref{lola85_0}  holds.	Let  $\llave{f_{k}}_{k}\subseteq  M_{B}\pare{V,W}$ be a Cauchy sequence and let $a\in D_{1}$. Since    $M_{B}\pare{V,W}$ has continuous evaluations,   $\llave{f_{k}\pare{a}}_{k}$ is a Cauchy sequence in $X_{1}$, and thus converges to some $x_{a}\in X_{1}$. 	Define the function $f:D_{1}\tien X_{1}$ by  $f\pare{a}:=x_{a}.$ 
	Since  $\llave{f_{k}}_{k}\subseteq M_{B}\pare{V,W}$ is a Cauchy sequence and $f_{k}\tien f$ pointwise, by property  III, we conclude that $f\in \F$.
	
		Let  $v\in V$. Since  $M_{B}\pare{V,W}$ is normed, by the $M$-inequality, the operator $B\pare{\cdot,v}$ is linear and continuous. Therefore, $\llave{B\pare{f_{k},v}}_{k}\subseteq W$ is a Cauchy  sequence, and hence it converges to some $w\in W$. Since $W$ has continuous evaluations, it follows that $B\pare{f_{k},v} \tien w$ pointwise. On the other hand, since \ref{lola85_0} holds, we have that $B\pare{f_{k},v}\tien B\pare{f,v}$ pointwise. Thus,   $B\pare{f,v}=w\in W$, and therefore $f\in M_{B}\pare{V,W}$ and $B\pare{f_{k},v}\tien B\pare{f,v}$ in $W$. 	Next, we show that $f_{k}\tien f$ in $M_{B}\pare{V,W}$. Given  $\epsilon>0$, there exists $K\in \N$ such that
	\begin{equation}
	\norm{B\pare{f_{i},v}-B\pare{f_{k},v}}_{W}\leq\norm{f_{i}-f_{k}}_{\mw}<\epsilon,\qquad \forall i,k\geq K,\quad \forall v\in \B_{V}.
	\end{equation}
	Fixing  $k\geq K$, and taking $i\tien \infty$, we obtain 
	\begin{equation}
	\norm{B\pare{f-f_{k},v}}_{W}=\norm{B\pare{f,v}-B\pare{f_{k},v}}_{W}\leq\epsilon,\qquad   \forall v\in \B_{V}.
	\end{equation}
	It follows that $\norm{f-f_{k}}_{\mw}\leq \epsilon,\forall k\geq K$.  
\end{proof}

When   $\F=F\pare{D_{1},X_{1}}$  in (\ref{piz13}), \ref{lola85_0} simplifies  and from the previous theorem  we obtain the following.

\begin{corollary}\label{d189_4} Let   $B:F\pare{D_{1},X_{1}}\times \V\tien \W$ be a bilinear operator, and let  $V\subseteq \V$, $W\subseteq \W$  be Banach spaces such that $W$ and $M_{B}\pare{V,W}$ have continuous evaluations. Then $M_{B}\pare{V,W}$ is a Banach space if, and only if, for every Cauchy sequence $\llave{f_{k}}_{k} \subseteq M_{B}\pare{V,W}$ and $f\in F\pare{D_{1},X_{1}}$ the following holds:  \begin{equation}\label{lola9}
	\text{if $f_{k}\tien f$ pointwise, then  $B\pare{f_{k},v} \tien B\pare{f,v}$  pointwise}, \quad \forall v\in V.
	\end{equation} 
\end{corollary}

In what follows, it will be convenient to have  the next  definitions.	Let  $V$ and $W$ be normed spaces. A linear operator $T:V\tien W$ is a \emph{normed isomorphism}  if it is bijective and both $T, T^{-1}$ are continuous. In this case, we say that $V$ and $W$ are \emph{norm-isomorphic}. Moreover, if additionally $\norm{Tv}=\norm{v}$, $\forall v\in V,$ then  $V$ and $W$ are said to be \emph{isometrically isomorphic}, or equivalently, $T$ is an \emph{isometric isomorphism}.
 
 \subsection{A  Multiplier Space not Necessarily Consisting of Functions}\label{lalo2}

Let $D$   be a non-empty set,  $U$, $\V$ and $\W$  vector spaces and  $X_{3}$ a  Banach space. We will now analyze the situation in which we have a bilinear operator \begin{equation}\label{lola57}
B_{1}:U\times \V\tien \W\subseteq F\pare{D,X_{3}},
\end{equation}  where we do not assume that $U$ consists of functions, as we did in (\ref{piz13}). We now start by  introducing the set
\begin{equation}
\F:=\llave{f\in F\pare{D,U}:f\pare{a}=u,\forall a\in D,\text{ for some }u\in U},		\end{equation}
which is a vector subspace of $F\pare{D,U}$. Let us denote the elements of $\F$ by  $f_{u}$, where $f_{u}\pare{a}=u\in U,\forall a\in D$.  We then have an isomorphism   $J:U\tien \F$ given by
\begin{equation}
Ju:=f_{u}.
\end{equation} 
Next, we define the function $B_{D}:\F\times \V\tien \W$ by
\begin{equation}
B_{D}\pare{f_{u},v}:=B_{1}\pare{u,v},\qquad \forall f_{u}\in \F,\qquad v\in \V.
\end{equation}
Note that $B_{D}$ is a bilinear operator, and it satisfies the hypotheses of Lemma \ref{d339_0}. Let  $V\subseteq \V$ and $W\subseteq \W$ be normed spaces. First, observe that \begin{equation}
J\pare{M_{B_{1}}\pare{V,W}}=M_{B_{D}}\pare{V,W}.
\end{equation}
Thus, the correspondence $J:M_{B_{1}}\pare{V,W}\tien M_{B_{D}}\pare{V,W}$ is an isomorphism. Moreover, given $u\in M_{B_{1}}\pare{V,W}$ we have
\begin{align}\label{lola55}
\norm{u}_{M_{B_{1}}\pare{V,W}}&=\sup\llave{\norm{B_{1}\pare{u,v}}_{W}:v\in \B_{V}}\nonumber\\[.3 cm]
&=\sup\llave{\norm{B_{D}\pare{f_{u},v}}_{W}:v\in \B_{V}}=\norm{f_{u}}_{M_{B_{D}}\pare{V,W}}.
\end{align}
Hence, the $B_{1}$-multiplication operators are bounded if and only if the $B_{D}$-multiplication operators are bounded.  

Considering the above isomorphism $J$  and using Lemma \ref{d339_0} with Theorem \ref{d265_b_0}, we next obtain a condition for the completeness of the multiplier space $M_{B_{1}}\pare{V,W}$.

\begin{corollary}\label{lola56_2}
	Suppose that $\W\subseteq F\pare{D,X_{3}}$. Let  $V\subseteq \V$ and  $W\subseteq \W$ be Banach spaces such that $W$ has continuous evaluations and the   $\pare{B_{1}}_{V,W}$-multiplication operators are bounded. Suppose that the following properties hold:
	\begin{enumerate}
		\item[i.]  Given  $a\in D$, there exists $c_{a}\in (0,\infty)$ such that for each $u\in M_{B_{1}}\pare{V,W}$, there exists $v_{u}\in \B_{V}$  such that $\norm{u}_{U}\leq c_{a} \norm{B_{1}\pare{u,v_{u}}\pare{a}}_{X_{3}}.$
		\item[ii.] If  $\llave{u_{k}}_{k}\subseteq M_{B_{1}}\pare{V,W}$ is a Cauchy sequence, $u\in U$ and  $u_{k}\tien u$ in $U$, then  $B_{1}\pare{u_{k},v} \tien B_{1}\pare{u,v}$ pointwise , $\forall v\in V$.
	\end{enumerate}  
	Then    $M_{B_{1}}\pare{V,W}$ is a Banach space. 
\end{corollary}

  \subsection{A Multiplier Space Consisting of Bounded Linear Operators}\label{lalo3}
  
	 We begin this section by introducing some notation that will be used throughout this and the subsequent sections. Given $g \in F\pare{D, \K}$ and $x \in X$, the \emph{pointwise product of $g$ and $x$} is the function $gx: D \tien X$ defined by
 	
 	\begin{equation}
 			gx\pare{a} := g\pare{a} x, \quad \forall a \in D.
 	\end{equation} More generally, given  $f \in F\pare{D, X}$, the \emph{pointwise product of $g$ and $f$} is the function $g \cdot f: D \tien X$ defined by
 	
 \begin{equation}
 		g \cdot f\pare{a} := g\pare{a} f\pare{a}, \quad \forall a \in D.
 \end{equation}

	For each $a \in D$, we define $e_a : D \to \R$ by
	\begin{equation}
	e_{a}\pare{d}: = \left \{ \begin{matrix} 1 & \mbox{if }d=a,
	\\ 0 & \mbox{if }d\neq a\end{matrix}\right., \qquad \forall d\in D.
	\end{equation}
	Let $c_{00}(D, X)$ denote the vector space generated by $\llave{e_{a}x: a\in D,\ x\in X}\subseteq F\pare{D,X}$. Note that $f \in c_{00}(D, X)$ if and only if there exists a finite set $A \subseteq D$ such that $f(d) = 0$ for all $d \in D \setminus A$. Thus, $c_{00}(D, X) \subseteq B(D, X)$. From now on, we will use $c_{00}(X) := c_{00}(\N, X)$.   Given $n \in \mathbb{N}$, observe that $e_{n} \in S\pare{\K}$ is the \emph{$n$-th canonical sequence}, which takes the value 1 in the $n$-th entry and 0 in all other entries.
	
	\begin{example}
	We have $c_{00}(X) \subseteq c_0(X) \subseteq c(X) \subseteq \ell_\infty(X)$ and $c_{00}(X) \subseteq \ell_p(X)$ for $1 \leq p < \infty$. 
	\end{example}

  Let $X$ and $Y$ be Banach spaces with $X,Y\neq \llave{0}$, and let $T\in \mathcal{L}\pare{X,Y}$. For each sequence $s\in S\pare{X}$, we define    
	\begin{equation}\label{lola162}
	\widetilde{T}s:=\llave{T\pare{s\pare{n}}}_{n}\in S\pare{Y}.
	\end{equation}  This defines a operator $\widetilde{T}: S(X) \to S(Y)$, which we will refer to as the \emph{sequences operator induced} by $T$.

 Note that if $T$ is an isomorphism, then $\widetilde{T}$ is an isomorphism and $\widetilde{T}^{-1} = \widetilde{T^{-1}}$.

%Let  $X$ and $Y$  be Banach spaces, and consider the corresponding spaces of vector sequences $S\pare{X}$ and $S\pare{Y}$.

Let us now define the function $	B_{\mathcal{L}}:\mathcal{L}\pare{X,Y}\times S\pare{X}\tien S\pare{Y}$ by
\begin{equation}
B_{\mathcal{L}}  \pare{T,s}:=\widetilde{T}s,
\end{equation}
where  $\widetilde{T}$ is the sequences operator induced by $T$. Note that $B_{\mathcal{L}}$ is a bilinear operator as in (\ref{lola57}), with  $U=\mathcal{L}\pare{V,W}$, $\V=S\pare{X}$ and $\W=S\pare{Y}$. \vs

	To show that $M_{\mathcal{L}} (V,W)$ is a Banach space  we need to establish before the following lemma. For each $n \in \mathbb{N}$, we define the function $Q_{n}: X \tien c_{00}\pare{X}$ by
\begin{equation}
	Q_{n}x := e_{n}x.
\end{equation} 
Clearly, $Q_{n}$ is an injective linear operator.

\begin{lemma}\label{d278_1}
	Let  $V\subseteq S\pare{X}$ be a $\BK$-space. If  $c_{00}\pare{X}\subseteq V$, then $Q_{n}\pare{X}$ is closed in $V$, and the  linear operator $Q_{n}:X\tien Q_{n}\pare{X}\subseteq V$  is a normed isomorphism,  $\forall n\in \N$.
\end{lemma}
\begin{proof}
	Fix  $n\in \N$.  
	First, let us show that $Q_{n}\pare{X}$ is closed in $V$. Let  $\llave{x_{k}}_{k}\subseteq X$ and $v\in V$ such that $Q_{n}x_{k}\tien v$. Since  $V$ has continuous evaluations, evaluating at $n$  yields that the sequence converges to $v\pare{n}$. It then follows that
	\begin{equation}\label{d277}
		v\pare{m} = \lim_{k\tien \infty}Q_{n}x_{k}\pare{m}= \left \{ \begin{matrix} \ \lim_{k\tien \infty}x_{k} & \mbox{if }m=n,
			\\ \lim_{k\tien \infty}0 & \mbox{if }m\neq n\end{matrix}\right.= e_{n}x\pare{m}=Q_{n}x\pare{m}.
	\end{equation}
	Thus,  $Q_{n}\pare{X}\subseteq V$ is closed, and hence is a Banach subspace.
	
	To show that $Q_{n}$ is continuous, by the closed graph theorem, it suffices to verify that the graph of $Q_{n}$ is closed. Let $\llave{x_{k}}_{k}\subseteq X$, $x\in X$ and $v\in V$ such that $x_{k}\tien x$ and $Q_{n}x_{k}\tien v$. From   (\ref{d277}), it directly follows that $v=Q_{n}x$. Hence, the graph of $Q_{n}$ is closed.
	
	Finally, since  $Q_{n}$ is bounded and bijective, the bounded inverse operator theorem implies that $Q_{n}$ is a normed isomorphism.
\end{proof}

The following result shows that linear operators appear naturally as multipliers (see, for instance, \cite{BCLS} or \cite{CPSW}).
 
\begin{theorem}\label{efren10} Let  $V\subseteq S\pare{X}$ and  $W\subseteq S\pare{Y}$ be BK-spaces. If $c_{00}\pare{X}\subseteq V$, then  the (multiplier) space  
	\begin{equation}
		M_{\mathcal{L}}\pare{V,W}=\llave{T\in \mathcal{L}\pare{X,Y}:\widetilde{T}v\in W, \ \forall v\in V}.
	\end{equation}  is a Banach space with the (operator) norm  
	\begin{equation}
	\norm{T}_{M_{\mathcal{L}}\pare{V,W}} = \sup\llave{\norm{\llave{T\pare{s\pare{n}}}_{n}}_{W}:s\in \B_{V}},\qquad \forall T\in M_{\mathcal{L}}\pare{V,W}.
	\end{equation}  In particular, all induced sequences  operators are
	continuous. 
\end{theorem}
\begin{proof}To obtain the conclusion, we will verify the conditions of Corollary \ref{lola56_2}.
	
	Let  $T\in 	M_{\mathcal{L}}\pare{V,W}$, $\llave{v_{k}}\subseteq V$ and $v\in V$ such that $v_{k}\tien v$. Given  $n\in \N$, since  $V$ has continuous evaluations, it follows that $v_{k}\pare{n}\tien v\pare{n}$, and thus  $B_{\mathcal{L}}  \pare{T,v_{k}}\pare{n}=\widetilde{T}v_{k}\pare{n}=T\pare{v_{k}\pare{n}}\tien T\pare{v\pare{n}}=\widetilde{T}v\pare{n}=	B_{\mathcal{L}}  \pare{T,v}\pare{n}$. 
	Applying Theorem \ref{bobo}, we conclude that the multiplication operators of $\pare{B_{\mathcal{L}}}_{V,W}$ are bounded.
	
	Fix  $n\in \N$.  Take $c_{n}=2\norm{Q_{n}}<\infty$, by Lemma \ref{d278_1}.  For  $T\in 	M_{\mathcal{L}}\pare{V,W}$, choose  $x\in \B_{X}$ such that $\norm{T}\leq 2\norm{Tx}_{Y}$. Taking  $v_{T}:=\frac{xe_{n}}{\norm{xe_{n}}}$, we have $v_{T}\in \B_{V}$ and $\norm{xe_{n}}B_{\mathcal{L}}\pare{T,v_{T}}\pare{n}=B_{\mathcal{L}}\pare{T,xe_{n}}\pare{n}=\widetilde{T}\pare{xe_{n}}\pare{n}=Tx$. Thus, \begin{equation}
		\norm{T}\leq 2\norm{Tx}_{Y} = 2\norm{xe_{n}} \norm{B_{\mathcal{L}}\pare{T,v_{T}}\pare{n}}_{Y}\leq c_{n} \norm{B_{\mathcal{L}}\pare{T,v_{T}}\pare{n}}_{Y}.
	\end{equation} This proves condition i of Corollary \ref{lola56_2}.
	
	Now, note that if $\llave{T_{k}}_{k}\subseteq 	\mathcal{L}\pare{X,Y}$ and there exists $T\in \mathcal{L}\pare{X,Y}$ such that $T_{k}\tien T$, then  $T_{k}x\tien Tx$, $\forall x\in X$. Therefore,  $\widetilde{T}_{k}v\tien \widetilde{T}v$ pointwise, $\forall v\in V$. In this case, property ii of Corollary \ref{lola56_2} is always satisfied.   
\end{proof}

		Consider $X=Y$ and recall that the space $\mathcal{L}(X): = \mathcal{L}(X,X)$ is an algebra under composition. 
	\begin{proposition}
		The bilinear operator $B_{\mathcal{L}}:\mathcal{L}(X)\times S(X)\to S(X)$ is multiplicative.
	\end{proposition}
	\begin{proof} In this case, $B= B_{\mathcal{L}}$ is given by $B(T,s):=\widetilde{T}s$.
		For $H, T\in \mathcal{L}(X)$, we have that
		\begin{equation}
			B_{TH}(s) = \widetilde{T\circ H}(s) = \{T\circ H(s(n))\}_{n}
			=B_{T}\big(\{H(s(n))\}_{n}\big) = B_{T}\circ B_{H}(s)=B_{T} B_{H}(s), \qquad \forall s\in S(X).
		\end{equation}
		Hence, $B$ is multiplicative.
	\end{proof}

	\begin{example}
		Let $1 \leq p,q \leq \infty$. Since $c_{00}(X)\subseteq \ell_{p}(X)$ and the spaces $\ell_{p}(X)$, $\ell_{q}(Y)$ are $BK$-spaces, it follows that the multiplier space $M_{\mathcal{L}}(\ell_{p}(X),\ell_{q}(Y))$ is a Banach space and $M_\mathcal{L} \pare{\ell_p (X)}$ is a Banach algebra.
	\end{example}

	\section{Pointwise Bilinear Operator Induced by $P$}\label{holadenuevo}\label{pop4}
	
	In this section, we present a general framework to obtain different multiplier spaces of functions corresponding to  a bounded bilinear mapping $P$ between Banach spaces. This construction extends the classical notion of scalar-valued pointwise multipliers to a broader vector-valued setting and allows us to  treat several   classes of multiplier spaces in a unified manner, as well as to establish their normability and completeness under mild assumptions.\\

	Let $X_{1}, X_{2}, X_{3}$ be Banach spaces, and suppose we have a bounded bilinear operator 
	\begin{equation}
	P: X_{1} \times X_{2} \to X_{3}.
	\end{equation}
	Given a non-empty set $D$, we define $B_{P}: F(D, X_{1}) \times F(D, X_{2}) \to F(D, X_{3})$ pointwisely, that is
	\begin{equation}
	B_{P}(f, v)(a) := P(f(a), v(a)), \qquad \forall a \in D.
	\end{equation}
	Then, $B_{P}$ is a bilinear operator, which we will call the \emph{pointwise bilinear operator induced by $P$}. From Section \ref{q21}, if $V \subseteq F(D, X_{2})$ and $W \subseteq F(D, X_{3})$ are vector spaces, we have that 
	\begin{equation}
	M_{B_{P}}(V, W) = \{f \in F(D, X_{1}) : B_{P}(f, v) \in W, \forall v \in V \}.
	\end{equation}
	This type of multiplier function $f$ will be referred to as a \emph{pointwise multiplier}. In the case where $X_{1} = X_{2} = X_{3} = \K$ and $D=\N$ or $D=\C$, they have been extensively studied (see \cite{MP}, \cite[$\S$ 14]{B}, \cite{CDS}, \cite[Ch. I, $\S$ 5]{S}), \cite{JM}, \cite{KLM1}, \cite{KLM2}. \vs
	
	The following result indicates a fundamental property of pointwise multipliers.
	
	\begin{proposition}\label{lola11_1} 
		If $V \subseteq F(D, X_{2})$ and $W \subseteq F(D, X_{3})$ are $\BK$-spaces, then all $B_{P}$-multiplication operators from $V$ into $W$ are continuous.
	\end{proposition}
	\begin{proof} 
		By Theorem \ref{piz21}, it suffices to verify  \ref{lola86_0}. Let $f \in M_{B_{P}}(V, W)$, $v \in V$, and $\{v_{k}\}_{k} \subseteq V$ be such that $v_{k} \to v$, and fix $a \in D$. Using the fact that $V$ has continuous evaluations and that $P$ is bounded, it follows that $B_{P}(f, v_{k})(a) = P(f(a), v_{k}(a)) \to P(f(a), v(a)) = B_{P}(f, v)(a)$.
	\end{proof}

	Next, we will analyze the normability and completeness of pointwise multiplier spaces.

\begin{corollary}\label{lola22_2_0} \label{lola13_5_0}
	Let $V \subseteq F(D, X_{2})$ and $W \subseteq F(D, X_{3})$ be $\BK$-spaces. The following statements are equivalent:
	\begin{enumerate}
		\item[i.] $M_{B_{P}}(V, W)$ is normed.
		\item[ii.] If $f \in M_{B_{P}}(V, W)$ and $P(f(a), v(a)) = 0$,  $\forall v \in V$,  $ \forall a \in D$, then  $f = 0$.
	\end{enumerate}  
\end{corollary}

\begin{proof}
	Suppose $M_{B_{P}}(V, W)$ is $B_{P}$-normed. Take $f \in M_{B_{P}}(V, W)$ such that $P(f(a), v(a)) = 0$ for all $v \in V$ and all $a \in D$. Then $B_{P}(f, v) = 0$ for all $v \in V$, and applying \ref{lola76_0}, we deduce that $f = 0$.
	
	Now assume ii. By Proposition \ref{lola11_1}, it remains to verify \ref{lola76_0}. Let $f \in M_{B_{P}}(V, W)$ such that $B_{P}(f, v) = 0$ for all $v \in V$. Then $P(f(a), v(a)) = 0$ for all $v \in V$ and all $a \in D$. Finally, applying ii, we conclude that $f = 0$.
\end{proof} 

%piz17

\begin{definition}
	Let $V \subseteq F(D, X_{2})$ and $W \subseteq F(D, X_{3})$ be $\BK$-spaces. We say that $P$ does not  $\pare{V,W}$-vanish at any point  if $B_{P}$ not  $\pare{V,W}$-vanish at any point. That is,  if for each $a\in D$, there exists $c_{a}\in (0,\infty)$ such that for each $f\in M_{B_{P}}\pare{V,W}$, there exists $v_{f}\in \B_{V}$  such that
	\begin{equation} 
		\norm{f\pare{a}}_{X_{1}}\leq c_{a} \norm{P\pare{f\pare{a},v_{f}\pare{a}}}_{X_{3}}.
	\end{equation}
\end{definition}

\begin{remark}
	Let $\{f_{k}\}_{k} \subseteq F(D, X_{1})$, $f \in F(D, X_{1})$, and $v \in V$ such that $f_{k} \to f$ pointwise. Since $P$ is bounded, it follows that 
	\begin{equation}\label{lola47}
		B_{P}(f_{k}, v)(a) = P(f_{k}(a), v(a)) \to P(f(a), v(a)) = B_{P}(f, v)(a), \quad \forall a \in D.
	\end{equation}
	Thus, the pointwise bilinear operator $B_{P}$ always satisfies the property indicated in (\ref{lola9}). 
\end{remark}

The following result is proven using Lemma \ref{d339_0}, (\ref{lola47}), and Theorem \ref{d265_b_0}.
\begin{theorem}\label{lola28_3}\label{lola35_2}
	Let $V \subseteq F(D, X_{2})$ and $W \subseteq F(D, X_{3})$ be $\BK$-spaces. If $P$ does not $\pare{V,W}$-vanish at any point, then $M_{B_{P}}(V, W)$ is a $\BK$-space.
\end{theorem}

\section{Spaces of Classical Scalar Pointwise Multipliers}\label{d170}\label{pop5}  
We now consider the fundamental example of pointwise multipliers. Let $D$ be a non-empty set, $X$ a Banach space, and consider the \emph{scalar multiplication} $P_{X}:\K\times X\to X$, defined by  
\begin{equation}\label{lola17}  
P_{X}(b,x):=bx.  
\end{equation}  
Then $P_{X}$ is a bounded bilinear operator and therefore we are  within the framework of Section \ref{holadenuevo}, with $X_{1}=\K$ and $X_{2}=X_{3}=X$. Now consider the induced pointwise bilinear operator $\BB:=B_{P_{X}}$. That is,  
\begin{equation*}  
\BB:F(D,\K)\times F(D,X)\to F(D,X),   
\end{equation*}  
is defined by
\begin{equation}\label{d302_0}  
\BB(f,v)=f\cdot v, \qquad\forall f\in F(D,\K),\qquad v\in F(D,X).  
\end{equation}  

The pointwise bilinear operator $\BB$, gives us now a space of scalar multipliers, which we will denote by $M(V, W)$, instead of the previously used notation $M_{\BB}(V, W)$. \\

The following result is a direct consequence of Proposition \ref{lola11_1}.  

\begin{corollary}\label{d327_1} \label{d329_2}  
	If $V,W\subseteq F(D,X)$ are $\BK$-spaces, then all $\mathscr{B}_{V,W}$-multiplication operators from $V$ to $W$ are continuous.  
\end{corollary}  

To analyze the completeness of $M(V, W)$, we first recall the following definition and introduce a convenient notation. Let $V \subseteq F(D, X)$ be a  vector space. The space $V$ \emph{does not vanish at any point} if for every $a \in D$, there exists $v \in V$ such that $v(a) \neq 0$. If, in addition, $V$ is a \BK-space, we will say that $V$ is a \emph{\BKn-space}. 

\begin{example} Let $V \subseteq S(\K)$ be a vector space. If $\{e_n\}_{n} \subseteq V$, then $V$ does not vanish at any point. On the other side, if $\llave{1}_{n}\in V$, then $V$ does not vanish at any point.
\end{example}

\begin{lemma}\label{efren6}  
	If $V \subseteq F(D,X)$ is a $\BKn$-space, then $P_X$ does not $(V,W)$-vanish at any point, for every $\BK$-space $W\subseteq F\pare{D,X}$.  
\end{lemma}  
\begin{proof}  
	Let $a\in D$. Since $V$ does not vanish at some point, there exists $v\in V$ such that $v(a)\neq 0$. Taking $c_{a}:=\frac{1}{\norm{v(a)}}\in (0,\infty)$, it follows that $\abs{f(a)} =c_{a}\abs{f(a)}\norm{v(a)}_{X}=c_{a}\norm{P_{X}(f(a),v(a))}_{X}, \forall f\in M(V,W).$  
	Thus, $V$ does not $P_{X}$-vanish at some point.  
\end{proof}

The following result is obtained by applying the previous lemma and Theorem \ref{lola28_3}.  The case $X = \K$ is well known. For example, see \cite[Thm 4.3.15]{Wil}.

\begin{theorem}\label{q31_3_0_0}\label{d318_1_0_0} \label{d328_0_0_0}  
	Let $V,W\subseteq F(D,X)$. If $V$ is a $\BKn$-space   and $W$ is a $\BK$-space, then $M(V,W)$ is a $\BK$-space. Furthermore, if $c_{00}(D,X) \subseteq W$, then $c_{00}(D,\K) \subseteq M(V,W)$. 
\end{theorem}

In the following proposition, we present a situation in which the condition of $V$ not vanishing at some point is both necessary and sufficient for $M(V,W)$ to be a $\BK$-space.  

\begin{proposition}\label{d343_0_new_0}  
	Let $V,W\subseteq F(D,X)$ be $\BK$-spaces such that $\{e_{a}\}_{a\in D}\subseteq M(V,W)$. Then $M(V,W)$ is a $\BK$-space if and only if $V$ does not vanish at some point.  
\end{proposition}  
\begin{proof}  
	In the previous corollary, we saw that if $V$ does not vanish at some point, then $M(V,W)$ is a $\BK$-space.  
	
	Now, assume that $M(V,W)$ is a $\BK$-space. Fix $a\in D$ and suppose that $v(a)=0$ for all $v\in V$. Then $\BB_{V,W}(e_{a},v)=0$ for all $v\in V$. Using Proposition \ref{piz4_2}, it follows that $e_{a}=0$, which is not possible.  
\end{proof}

 The following property of $\BB$ allows to algebraically operate with the scalar multiplier spaces $M\pare{V,W}$, $M\pare{W,Z}$, and $M\pare{V,Z}$  (For example, see \cite{B}).

\begin{proposition}
	The bilinear operator $\BB:F\pare{D,\K}\times F\pare{D,X}\to F\pare{D,X}$ is multiplicative.
\end{proposition}
\begin{proof}
	First, note that $F\pare{D,\K}$ is an algebra under the usual operations for scalar-valued functions. Let $f,g\in F\pare{D,\K}$ and $v\in F\pare{D,X}$. Then, $\BB\pare{g\cdot f,v} = \pare{g\cdot f}\cdot v = g\cdot\pare{f\cdot v} = \BB\pare{g,\BB\pare{f,v}}$.
\end{proof}

  Let $V,W,Z\subseteq F\pare{D,X}$ be $\BKn$-spaces. Since $\BB$ is a multiplicative bilinear operator, by (\ref{piz25_1}) it follows that $M\pare{V,W}\cdot M\pare{W,Z}\subseteq M\pare{V,Z}$. Denoting by  $R_{V,W}$, $R_{W,Z}$, and $R_{V,Z}$ the corresponding representation operators. For $f\in M\pare{V,W}$ and $g\in M\pare{W,Z}$, by (\ref{lola68}) we have
\begin{equation}\label{lola69}
R_{V,Z}\pare{g\cdot f} = R_{W,Z}\pare{g}\circ R_{V,W}\pare{f}.
\end{equation}
 Now, using (\ref{d345}), for each $f\in M\pare{V,W}$ and $g\in M\pare{W,Z}$, we obtain that 
\begin{equation}\label{lola32}
\norm{g\cdot f}_{M\pare{V,Z}} \leq \norm{f}_{M\pare{V,W}}\norm{g}_{M\pare{W,Z}}.
\end{equation}

The result we present below is well known in the context of certain spaces of analytic functions, in which case $D=X=\K$ (see \cite[Prop. 7.1]{CMR}).

\begin{theorem}\label{lola23} \label{piz25_2} \label{lola98_0} 
	If $V\subseteq F\pare{D,X}$ is a $\BK$-space that does not vanish at any point, then:
	\begin{enumerate}
		\item[i.]   $M\pare{V}$ is a commutative Banach algebra with $\mathit{1}$ as the identity and $\norm{\mathit{1}}_{M\pare{V}}=1$. 
		
		\item[ii.] $M\pare{V}\hookrightarrow B\pare{D,\K}$.
	\end{enumerate} 
\end{theorem}  
\begin{proof}
	i. From (\ref{lola32}), it follows that $M\pare{V}$ is an algebra and clearly its product is commutative. Moreover, from Theorem \ref{q31_3_0_0}, we have that the multiplier space $M\pare{V}$ with its operator norm is a $\BK$-space, and $\mathit{1} \in M\pare{V}$. Additionally, $\norm{\mathit{1}}_{M\pare{V}} = \sup \llave{\norm{\mathit{1}\cdot v}_{V}:v \in \B_{V}} = \sup \llave{\norm{ v}_{V}:v \in \B_{V}} = 1$.
	
	ii. Let $f \in M\pare{V}$, $a \in D$, and $n \in \N$. Since $M\pare{V}$ is a $\BK$-space and $f^n \in M\pare{V}$, there exists $c_a \in \pare{0,\infty}$ such that  
\begin{equation}
	\abs{f\pare{a}}^n = \abs{f^n\pare{a}} \leq c_a \norm{f^n}_{M\pare{V}} \leq c_a \norm{f}_{M\pare{V}}^n.
\end{equation}
	Taking the $n$-th root in the inequality above, and then taking $n \to \infty$, it follows that $\abs{f\pare{a}} \leq \norm{f}_{M\pare{V}}$. Thus, $\norm{f}_{B\pare{D,\K}} \leq \norm{f}_{M\pare{V}}$.
\end{proof}

In the case when $W \neq V$ we have the following.

\begin{proposition}
	Let $V, W \subseteq F\pare{D,X}$ be BK-spaces such that $c_{00} (D,X) \subseteq  V \cap W$. If there exists $x\in X\setminus\llave{0}$ such that
\begin{equation}\label{roro}
	c:=\sup\llave{\frac{\norm{e_{a}x}_{V}}{\norm{e_{a}x}_{W}}:a\in D}<\infty,
\end{equation}
	then   $M(V,W)\hookrightarrow_{c}B\pare{D,\K}$.
\end{proposition}

\begin{proof} 
	Let $f\in M\pare{V,W}$ and $a\in D$. Then
	$
	\abs{f\pare{a}}\norm{e_{a}x}_{W}
	=
	\norm{f\cdot\pare{e_{a}x}}_{W}
	=
	\norm{\BB\pare{f,e_{a}x}}_{W}
	\leq
	\norm{f}_{M}\norm{e_{a}x}_{V},
	$
	and therefore
	$
	\abs{f\pare{a}}\leq c\norm{f}_{M}.
	$
\end{proof}
\begin{remark}
When $X = \K$ and $D = \N$, note that (\ref{roro}) takes the form
$
c=\sup\llave{\frac{\norm{e_{n}}_{V}}{\norm{e_{n}}_{W}}:n\in \N}.
$
\end{remark}

\begin{corollary}
	Let $V, W \subseteq F\pare{D,X}$ be BK-spaces such that $c_{00} (D,X) \subseteq  V \cap W$. If $W\hookrightarrow_{c}V$, then  $M(V,W)\hookrightarrow_{c}B\pare{D,\K}$.
\end{corollary}

\subsection{Equivalence of Complete Norms in $\BK$  Spaces of Functions}

We will now use multipliers to establish that the property of having continuous evaluations allows us to conclude the equivalence of two complete norms defined in the same function space. To do this, we will first show that if $V, W \subseteq F\pare{D, X}$ have continuous evaluations and $V \subseteq W$, then the inclusion map $i: V \to W$ is continuous. This was already proven in a different way by A. Wilansky \cite[p. 55]{Wil}.

\begin{proposition}\label{d326_0}   
	Let $V, W \subseteq F\pare{D, X}$ be Banach spaces. Then 
	\begin{equation}\label{lola211}
	\text{$V \subseteq W$ if and only if $\mathit{1} \in M\pare{V, W}$.}
	\end{equation} 
	In this case, if $V$ and $W$ have continuous evaluations, then the inclusion map $i: V \to W$ is continuous. In fact, $V \hookrightarrow_{\norm{\mathit{1}}_{M\pare{V, W}}} W$.  	Moreover, if $V \neq \{0\}$, we have $\norm{\mathit{1}}_{M\pare{V, W}} > 0$.
\end{proposition}
\begin{proof} 
	Clearly, (\ref{lola211}) is satisfied. Let $v \in V$. Applying the $M$-inequality (Proposition \ref{piz2_2}), it follows that $\norm{v}_{W} = \norm{\mathit{1} \cdot v}_{W} \leq \norm{\mathit{1}}_{M\pare{V, W}} \norm{v}_{V}$.
\end{proof}

\begin{example}\label{lola189_0}
	If $V \subseteq F\pare{D, \K}$ is a $\BK$-space such that $V \subseteq B\pare{D, \K}$, then the inclusion map $i: V \to B\pare{D, \K}$ is continuous.
\end{example}

\begin{example} 
	Let $A\subseteq F\pare{D,\K}$ be a $\BK$-space that does not vanish at any point. If $A$ is a Banach algebra with respect to the pointwise product, then $A \hookrightarrow_{\norm{\mathit{1}}_{M\pare{A, B\pare{D,\K}}}} B\pare{D,\K}$.
\end{example}
\begin{proof}
	Since $A$ is a $\BKa$-space, by Theorem \ref{lola98_0}, it follows that $M\pare{A} \subseteq B\pare{D,\K}$. Now, observing that $A \subseteq M\pare{A}$ because $A$ is an algebra,  and applying Proposition \ref{d326_0}, we obtain the conclusion.
\end{proof}

% In the case where $V$ and $W$ are $\BK$-spaces with $V \subseteq W$, it will always be convenient to keep in mind that the inclusion will be continuous.

\begin{theorem}\label{piz22_1_0} 
	Let $V \subseteq F\pare{D, X}$ be a vector space. If the normed spaces  $V_{1} = (V, \norm{\cdot}_{1})$ and $V_{2} = (V, \norm{\cdot}_{2})$ are complete and have continuous evaluations, then the norms $\norm{\cdot}_{1}$ and $\norm{\cdot}_{2}$ are equivalent. 
\end{theorem}

In \cite[p. 414]{AN} it is proven that every infinite-dimensional Banach space has a complete norm that is not equivalent to the original one. Hence, if $D$ is infinite and $X\neq\llave{0}$, there always exists in $B\pare{D, X}$ a complete norm that lacks continuous evaluations.

\begin{example} 	
	Let $E$ be a normed space. Then its dual normed space $E^{*}$ is a $\BK$-space, and so any norm on $E^{*}$ under which it is a $\BK$-space is equivalent to its dual norm.
\end{example}

\section{Vector Pointwise Multipliers}\label{pop6}

We now will associate with each bilinear operator $B$  another  bilinear operator  $B'$.

\begin{definition} 	Let $\F$, $\V$ and $\W$ be vector spaces, and let $B:\F\times \V\tien \W$ be a bilinear operator. The  \emph{reflected operator of $B$} is the function    $B':\V\times \F\tien \W$ defined by
	\begin{equation}
		B'\pare{v,f}:=B\pare{f,v}.
	\end{equation}  
\end{definition}
Note that $B'$ is also a bilinear operator, and moreover, $B^{''}:=\pare{B'}'=B$.

Let  $B:\F\times \V\tien \W$ be a bilinear operator, and consider its reflected operator $B':\V\times \F\tien \W$. Given vector subspaces $F\subseteq \F$ and  $W\subseteq \W$, we have
\begin{equation}
	M_{B'}\pare{F,W}=\llave{v\in \V:B\pare{f,v}\in W,\forall f\in F}.
\end{equation}

	Let us continue the discussion of Section~\ref{holadenuevo} and consider    the pointwise bilinear operator $B_{P}: F(D, X_{1}) \times F(D, X_{2}) \to F(D, X_{3})$ induced by $P$. Observe that the reflected operator $P' : X_{2} \times X_{1} \to X_{3}$ is also bounded, and $\pare{B_{P}}'$ is precisely the bilinear operator induced by $P'$, i.e., $\pare{B_{P}}' = B_{P'}$. Thus, the reflected bilinear operator of $B_{P}$ is also pointwise, and consequently, the results we established in Section  \ref{holadenuevo} for $B_{P}$   also hold for $B_{P'}$.   \vs

	Now, let us return to the scalar multiplication $P_{X}:\K \times X \to X$ and its induced pointwise linear operator $\BB: F(D,\K) \times F(D,X) \tien F(D,X)$. Working with its reflected operator, we will obtain  multipliers that are vector-valued  and derive results analogous to those established for $\BB$, where the multipliers were scalar-valued. So, let us consider its reflected operator $\BB'$, that is,
	\begin{equation} \BB':F(D,X)\times F(D,\K)\to F(D,X),  
	\end{equation}  
	and  
	\begin{equation}\label{d302_0}  
		\BB'(v,f)=f\cdot v, \qquad\forall f\in F(D,\K),\qquad v\in F(D,X).  
	\end{equation}  
	
	From the pointwise bilinear operator $\BB'$, we  now obtain a space of  multipliers, which we will denote by $M'(V, W)$, instead of the previously used notation $M_{\BB'}(V, W)$.

\begin{lemma}\label{fin1}
	If $V \subseteq F\pare{D, \K}$ is a $\BKn$-space, then $P_X'$ does not $(V,W)$-vanish at any point, for every $\BK$-space $W\subseteq F\pare{D, X}$.   
\end{lemma}
\begin{proof}  
	Let $a \in D$. Since $V$ does not vanish at some point, there exists $v \in V$ such that $v(a) \neq 0$. Let $f \in M'(V,W)$ be such that $f(a) \neq 0$. It follows that
	$P_{X}'(f(a),v(a)) = v(a)f(a) \neq 0$.
	Thus, $V$ does not $(V,W)$-vanish at any point.
\end{proof}

It follows from Theorem \ref{lola28_3} and Lemma \ref{fin1} that  spaces of vector multipliers have similar  properties  to that of spaces of scalar multipliers.  In particular, the following results hold. We omit the first two proofs.

\begin{theorem}\label{piz27_1_0} Let $V \subseteq F\pare{D, \K}$ and $W \subseteq F\pare{D, X}$. If $V$ is a $\BKa$-space  and $W$ is a $\BK$-space, then $\Mt\pare{V, W}$ is a $\BK$-space.
\end{theorem}

\begin{proposition} \label{d343_nu_0} 
	Let $V \subseteq F\pare{D, \K}$ and $W \subseteq F\pare{D, X}$ be $\BK$-spaces such that for each $a \in D$, there exists $x \in X$, $x \neq 0$, with $e_{a}x \in \Mt\pare{V, W}$. Then, $\Mt\pare{V, W}$ is a $\BK$-space if and only if $V$ does not vanish at some point.
\end{proposition}

\begin{proposition}\label{final2}
	Let $V \subseteq F\pare{D,\K}$ and $W \subseteq F\pare{D,X}$ be BK-spaces such that $\llave{e_{a}}_{a\in D}\subseteq V$ and $c_{00} (D,X) \subseteq  W$. If
	\begin{equation}
		c:=\sup\llave{\frac{\norm{x}_{X}\norm{e_{a}}_{V}}{\norm{e_{a}x}_{W}}:  a\in D,\ x\in X\setminus\llave{0}}<\infty,
	\end{equation}
	then $M'(V,W)\subseteq B\pare{D,X}$ and $M'(V,W)\hookrightarrow_{c}B\pare{D,X}$.
\end{proposition} 
\begin{proof}  Let $f\in M'(V,W)$ and $a\in D$. Then
	\begin{equation}
		\norm{e_{a}f\pare{a}}_{W}
		=
		\norm{e_{a}\cdot f}_{W}
		=
		\norm{\BB'\pare{f,e_{a}}}_{W}
		\leq
		\norm{f}_{M'}\norm{e_{a}}_{V}.
	\end{equation}
	Using the hypothesis, it follows that
	$
	\norm{f\pare{a}}_{X}\leq c\norm{f}_{M'}.
	$
\end{proof}

\subsection{Schur Multiplier Space}

We now apply our results to the space of pointwise multipliers on (infinite) matrix spaces.

An \emph{(infinite) matrix} is a function $A:\N\times \N\tien \K$.
We use the standard notation $A=\pare{a_{ij}}$ to refer to such a matrix, and denote by $Mat$ the set of all matrices of this type. Hence, $Mat=F\pare{\N\times \N,\K}$.

\begin{definition}
	Let $A=\pare{a_{ij}}$ and $C=\pare{c_{ij}}\in Mat $. The \emph{Schur product} of the matrices $A$ and $C$ is the matrix $A \cdot C\in Mat $ defined by
	\begin{equation}
		A\cdot C:=\pare{a_{ij}c_{ij}}.
	\end{equation}
\end{definition}

Within the framework of Theorem \ref{d318_1_0_0}, let us take $D=\N\times \N$ and $X=\K$. Next, we consider the operator $B \ast: Mat  \times Mat \to Mat $ defined by  
\begin{equation}
	\text{$B \ast$}\pare{A,C}:=A\cdot C.
\end{equation}
Clearly, $\text{$B \ast$}$ is bilinear. By Theorem \ref{d318_1_0_0}, the following result holds.

\begin{theorem}
	Let $M_{1}\subseteq Mat $ be a $\BK$-space that does not vanish at some point. If $M_{2}\subseteq Mat $ is also a $\BK$-space, then the Schur multiplier space $M_{\text{$B \ast$}}\pare{M_{1},M_{2}}=\llave{A\in Mat :A\cdot  C\in M_{2},\forall C\in M_{1}}$
	endowed with its operator norm, is a $\BK$-space. Moreover, if $Mat _{00}\subseteq M_{2}$, then
	\begin{equation}
		Mat _{00}\subseteq M_{\text{$B \ast$}}\pare{M_{1},M_{2}}.
	\end{equation}
\end{theorem}

\section{The b-ideal Part}\label{pop7}

 	Let $D$ be a non-empty set. Given $f \in F(D, \K)$ we use the standard notation  $|f|(a) := |f(a)|,  a \in D$. Here, $|f(a)|$  denote the modulus  of $f(a)$.   Recall a vector subspace $I \subseteq F(D,\K)$ is called an \emph{order ideal} if for  $f \in F(D,\K)$ and $g \in I$ such that $0 \le \abs{f} \le \abs{g}$, we have $f \in I$.\

%\begin{definition}\label{lola115}
%	
%\end{definition}

The following result is well known and motivates a condition for the definition of a $b$-ideal, which generalizes the notion of an order ideal to the vector-valued setting.  

	\begin{lemma}\label{d231_1*_1}
		Let $I \subseteq F(D, \K)$ be a vector space. Then $I$ is an order ideal if, and only if $B(D,\K) \subseteq M(I)$, that is \begin{equation}\label{tu}
			f \cdot g \in I,\qquad  \forall f \in B(D, \K),\qquad  \forall g \in I.
		\end{equation}  
	\end{lemma}

	\begin{definition}\label{lola108}
		Let $U$ be a subspace of $F(D, X)$. Then $U$ is a \emph{$b$-ideal} if $B(D,\K) \subseteq M(U)$, that is \begin{equation}
			f \cdot g \in U,\qquad  \forall f \in B(D, \K),\qquad  \forall g \in U.
		\end{equation} 
	\end{definition}

	The following propositions extend to the vector-valued case the well-known result on pointwise multipliers for spaces $V$ and $W$ consisting of scalar functions (For example, see \cite[Prop. 2]{MP} and the references therein.) 
	
\begin{proposition}\label{d274_2}\label{lola100_0} \label{d273_3}\label{lola83_0} 
	Let $V,W\subseteq F\pare{D,X}$ be vector spaces.  If one of the spaces $V$ or $W$ is a  $b$-ideal, then $M\pare{V,W}$ is an order ideal.  
\end{proposition}
\begin{proof}
	   Let $g\in B\pare{D,\K}$ and $f\in M\pare{V,W}$. Since $W$ is a $b$-ideal, it follows that $(g\cdot f)\cdot v = g\cdot \pare{f\cdot v} \in W$, $\forall v\in V.$ 
	Thus, $g\cdot f \in M\pare{V,W}, \forall g\in B\pare{D,\K}$. The case where $V$ is a  $b$-ideal    can be proven similarly. 
\end{proof}

In the case of vector multipliers, the above result takes the following form. Its proof is similar.

\begin{proposition}\label{lola99_0}   Let $V \subseteq F\pare{D, \K}$ and $W \subseteq F\pare{D, X}$ be vector spaces. If one of the spaces $V$ or $W$ is a $b$-ideal, then $M'\pare{V, W}$ is a $b$-ideal.  
\end{proposition}
%\begin{proof} 	i. Let $g \in B\pare{D, \K}$, $f \in \Mt\pare{V, W}$, and $v \in V$. Since $W$ is b-invariant, it follows that $v \cdot (g \cdot f) = g \cdot (v \cdot f) \in W, \forall v \in V$. Hence, $g \cdot f \in \Mt\pare{V, W}, \forall g \in B\pare{D, \K}$. The case where $V$ is $b$-invariant can be proved in a similar manner.
%	
%	ii. This follows directly from  Theorem \ref{piz27_1_0} and i. 
%\end{proof}
	
	Given a vector space $U \subseteq F\pare{D, X}$, we define the set
	\begin{equation}\label{d331_0}
		\text{Ip}\hspace{-2pt}\pare{U}:=M'\pare{B\pare{D,\K}, U}.
	\end{equation}
	Clearly, $\text{Ip}\hspace{-2pt}\pare{U}$ is a vector space. Since $\mathit{1} \in B\pare{D, \K}$, it follows that $\text{Ip}\hspace{-2pt}\pare{U} \subseteq U$. By Proposition \ref{lola99_0}, $\text{Ip}\hspace{-2pt}\pare{U}$ is  a $b$-ideal and, furthermore, it is the largest $b$-ideal  space contained in $U$. For this reason, we will refer to this space as the \emph{$b$-ideal  part of $U$}.
	
		\begin{theorem}\label{d334} \label{d335_0} \label{d310_0} 	If $U \subseteq F\pare{D,X}$ is a $\BK$-space, then  $\text{Ip}\hspace{-2pt}\pare{U}$ is a $b$-ideal $\BK$-space. Moreover, \begin{equation}\label{efren35} \norm{g\cdot f}_{U}\leq \norm{g}_{B\pare{D,\K}}\norm{f}_{\text{Ip}\hspace{-0pt}\pare{U}},\qquad g\in B\pare{D,\K}, \qquad \forall f\in \textnormal{Ip}\hspace{-2pt}\pare{U}. \end{equation} In particular, $\textnormal{Ip}\hspace{-2pt}\pare{U} \hookrightarrow U$. Furthermore, if $c_{00}\pare{D,X} \subseteq U$, then $c_{00}\pare{D,X} \subseteq \textnormal{Ip}\hspace{-2pt}\pare{U}$.  
	\end{theorem}
	\begin{proof}  
		By Theorem \ref{piz27_1_0}, it follows that $\text{Ip}\hspace{-2pt}\pare{U}$ is a $\BK$-space. Using the $M'$-inequality, we obtain (\ref{efren35}). Furthermore, if $c_{00}\pare{D,X} \subseteq U$, it is clear that $c_{00}\pare{D,X} \subseteq \text{Ip}\hspace{-2pt}\pare{U}$.
	\end{proof}
	
	\begin{example} $\textnormal{Ip}\hspace{-2pt}\pare{ c\pare{\K}}\equiv c_{0}\pare{\K}$. 
	\end{example}

	\begin{theorem}\label{lola166}
		If $V, W \subseteq F\pare{D, X}$ are $\BK$-spaces, where $V$ does not vanish at some point, then  
		\begin{equation}
			\textnormal{Ip}\hspace{-2pt}\pare{M\pare{V, W}} \equiv M\pare{V, \textnormal{Ip}\hspace{-2pt}\pare{W}}.
		\end{equation}
	\end{theorem}
	\begin{proof}  Since $M\pare{V, \text{Ip}\hspace{-2pt}\pare{W}} \subseteq M\pare{V, W}$, we obtain $M\pare{V, \text{Ip}\hspace{-2pt}\pare{W}} \subseteq \text{Ip}\hspace{-2pt}\pare{M\pare{V, W}}$.   Let $f \in \text{Ip}\hspace{-2pt}\pare{M\pare{V, W}}$ and $v \in V$. Then,  
		$ 
		g \cdot \pare{f \cdot v} = \pare{g \cdot f} \cdot v \in W,  \forall g \in B\pare{D, \K}.
		$  
		Then $f \cdot v \in \text{Ip}\hspace{-2pt}\pare{W}$, and therefore $f \in M\pare{V, \text{Ip}\hspace{-2pt}\pare{W}}$. This proves that $\text{Ip}\hspace{-2pt}\pare{M\pare{V, W}} \subseteq M\pare{V, \text{Ip}\hspace{-2pt}\pare{W}}$.  
		
		For $f \in \text{Ip}\hspace{-2pt}\pare{M\pare{V, W}}$, using (\ref{efren35}), we have  
		\begin{align*}
			\norm{f}_{\text{Ip}\hspace{-0pt}\pare{M\pare{V, W}}} &= \sup \llave{\norm{g \cdot f}_{M\pare{V, W}} : g \in \B_{B\pare{D, \K}}} = \sup \llave{\norm{g \cdot f \cdot v}_{W} : g \in \B_{B\pare{D, \K}}, \, v \in \B_{V}} \\[.3 cm]
			&= \sup \llave{\norm{f \cdot v}_{\text{Ip}\hspace{-0pt}\pare{W}} : v \in \B_{V}} = \norm{f}_{M\pare{V, \text{Ip}\hspace{-0pt}\pare{W}}}.\qedhere
		\end{align*}
	\end{proof}

	\subsection{Banach $b$-ideal Space}
		Recall a normed space $Y \subseteq F(D,\K)$ is called a \emph{normed ideal} if $Y$ is an order ideal  and for each $f ,g \in Y$  satisfying $0 \le |g| \le |f|$, we have  $\|g\|_Y \le \|f\|_Y$. If, in addition, $Y$ is a Banach space, then $Y$ is called a Banach ideal.
	
	The following result is well known.  
	\begin{lemma}\label{lol2_1}
		Let $Y \subseteq F(D,\K)$ be a normed space which is a  $b$-ideal. Then $Y$ is a normed ideal if and only if
		\begin{equation}\label{lol}
			\|b \cdot f\|_{Y} \leq \|b\|_{B\pare{D,\K}} \norm{f}_{Y},\qquad \forall f\in Y,\qquad \forall b\in B\pare{D,\K}.
		\end{equation}
	\end{lemma}

 This motivates the following definition.
	
	\begin{definition}\label{lola108_new}
		Let $U\subseteq F(D, X)$ be a  normed space. Then $U$ is a normed $b$-ideal, if $U$ is a $b$-ideal and \begin{equation}
			\|b \cdot f\|_{U} \leq \|b\|_{B\pare{D,\K}} \norm{f}_{U},\qquad \forall f\in U,\qquad \forall b\in B\pare{D,\K}.
		\end{equation} 
	\end{definition} 
	
	To indicate that $U \subseteq F(D,X)$ is a complete normed b-ideal, we will simply say that $U$ is a \emph{Banach $b$-ideal}. If additionally  $U$ is a $BK$-space, we will say that $U$ is a \emph{BK b-ideal}.  (In \cite[p. 814]{V}   a Banach b-ideal is called an ideal* space.)

The  following result is clear. 

\begin{proposition}\label{sam1}  
	Let $V \subseteq F\pare{D, \K}$ and $W \subseteq F\pare{D, X}$ be normed spaces. If $W$ is a normed $b$-ideal, then $M'\pare{V,W}$ is a normed $b$-ideal.
\end{proposition}

	\begin{corollary}\label{yaya2}	If $U \subseteq F\pare{D,X}$ is a $\BK$-space, then  $\text{Ip}\hspace{-2pt}\pare{U}$ is a   $\BK $ b-space. Hence, \begin{equation}\label{yaya1}  \norm{g\cdot f}_{\text{Ip}\hspace{-0pt}\pare{U}}\leq \norm{g}_{B\pare{D,\K}}\norm{f}_{\text{Ip}\hspace{-0pt}\pare{U}},\qquad g\in B\pare{D,\K}, \qquad \forall f\in \textnormal{Ip}\hspace{-2pt}\pare{U}. \end{equation}  
\end{corollary}
\begin{proof}
By Theorem~(\ref{d334}), it only remains to verify (\ref{yaya1}). Let $g,h\in B\pare{D,\K}$ and $f\in \textnormal{Ip}\hspace{-2pt}\pare{U}$. Using (\ref{efren35}), we obtain
\begin{equation}
	\norm{h\cdot (g\cdot f)}_{U}
	=
	\norm{(h\cdot g)\cdot f}_{U}
	\leq
	\norm{h\cdot g}_{B\pare{D,\K}}\norm{f}_{\text{Ip}\hspace{-0pt}\pare{U}}
	\leq
	\norm{ g}_{B\pare{D,\K}}\norm{f}_{\text{Ip}\hspace{-0pt}\pare{U}}.
\end{equation}
Taking the supremum over $h\in B\pare{D,\K}$, it follows that
$
\norm{g\cdot f}_{\text{Ip}\hspace{-0pt}\pare{U}}
\leq
\norm{g}_{B\pare{D,\K}}\norm{f}_{\text{Ip}\hspace{-0pt}\pare{U}}.
$
\end{proof}

The following result is obtained directly by applying Theorem~\ref{piz22_1_0} and Corollary \ref{yaya2} .

\begin{proposition}\label{final8} 
If $U \subseteq F\pare{D, \K}$ is a BK-space which is a $b$-ideal, then the norm on $U$ is equivalent to that of $\text{Ip}\hspace{-0pt}\pare{U}$, which is a BK $b$-ideal. Therefore, there exists $c\in(0,\infty)$ such that
	\begin{equation}\label{yaya3}
		\| b \cdot u \|_U \leq c \|b\|_{B(D,\K)} \|u\|_U,
		\qquad
		b\in \ell_{\infty}\pare{\K},
		\qquad
		\forall u\in U.
	\end{equation}
\end{proposition}

	  \begin{corollary}\label{lola84_0}
	  	Let $V,W\subseteq F\pare{D,X}$ be $\BK$-spaces such that $V$ does not vanish at some point. If $V$ is a Banach $b$-ideal, then
	  	\begin{equation}
	  		M\pare{V,W}\equiv M\pare{V,\text{Ip}\hspace{-0pt}\pare{W}}.
	  	\end{equation} 
	  	In particular, if $V\subseteq F\pare{D,X}$ is a $BKN$-space, then
	  	$M(\textnormal{Ip}\hspace{-2pt}\pare{V}, W) = M(\textnormal{Ip}\hspace{-2pt}\pare{V}, \textnormal{Ip}\hspace{-2pt}\pare{W})$ and therefore \begin{equation}
	  		M(V,W) \subseteq  M\pare{\textnormal{Ip}\hspace{-2pt}\pare{V}, \textnormal{Ip}\hspace{-2pt}\pare{W}}.
	  	\end{equation}
	  \end{corollary}
	  \begin{proof}
	  	By Proposition \ref{d274_2} and Theorem \ref{lola166}, it follows that 
	  	$M\pare{V,W}=\text{Ip}\hspace{-0pt}\pare{M\pare{V,W}}= M\pare{V,\text{Ip}\hspace{-0pt}\pare{W}}$.
	  	Since $V$ is a Banach $b$-ideal, we obtain that \setlength{\belowdisplayskip}{-4pt}
	  	\begin{align*} 
	  		\norm{f}_{M\pare{V,\text{Ip}\hspace{-0pt}\pare{W}}}
	  		&=\sup\llave{\norm{f\cdot v}_{\text{Ip}\hspace{-0pt}\pare{W}}:v\in \B_{V}}
	  		=\sup\llave{\norm{b\cdot\pare{f\cdot v}}_{W}:v\in \B_{V}, b\in \B_{B\pare{D,\K}}}\\[.3 cm] 
	  		&=\sup\llave{\norm{f\cdot h}_{W}:h\in \B_{V}}
	  		=\norm{f}_{M\pare{V,W}},\qquad 
	  		\forall f\in M\pare{V,\text{Ip}\hspace{-0pt}\pare{W}}.\qedhere
	  	\end{align*}
	  \end{proof} 
	  
	  \begin{remark}
	  	Let $V,W\subseteq F\pare{D,X}$ be $\BK$-spaces, where in addition $V$ is a $b$-ideal  that does not vanish at some point. Let $f\in M\pare{V,W}$ and let $B_{f}:V\tien W$ be the corresponding multiplication operator. Consider the inclusion $i:\text{Ip}\hspace{-0pt}\pare{W}\tien W$ and denote by $T$ the operator $B_{f}$, but regarded as taking values in $\text{Ip}\hspace{-0pt}\pare{W}$. Corollary \ref{lola84_0} then indicates that $B_{f}=i\circ T$. In other words, any multiplication operator from $V$ into $W$ factors through $\text{Ip}\hspace{-0pt}\pare{W}$, or equivalently, the following diagram is commutative:
	  	\begin{equation}\label{lola121}
	  		\xymatrix{
	  			V \ar[r]^{\ B_{f}\quad} \ar@{-->}[rd]_{T\ } &   \ W    \\ 
	  			\quad & \ \text{Ip}\hspace{-0pt}\pare{W} \ar@{-->}[u]_{\ i} .
	  		}
	  	\end{equation}
	  \end{remark}

The following result is obtained by directly applying Corollary \ref{lola84_0} and Theorems \ref{lola166} and \ref{lola23} .
\begin{corollary}
	Let $U\subseteq F\pare{D,X}$ be a $BK$-space that does not vanish at some point. Then $U$ is a $b$-ideal if and only if $M(U) = B\pare{D,\K}$. Moreover, $U$ is a Banach $b$-ideal if and only if $M(U) \equiv B\pare{D,\K}$.
\end{corollary} 

Corresponding to  Theorem \ref{lola166} and Corollary \ref{lola84_0} for scalar multipliers, the following results also hold for vector-valued multipliers.

\begin{theorem}\label{lola216} 
	If $V\subseteq F\pare{D,\K}$ and $W\subseteq F\pare{D,X}$ are $\BK$-spaces, and moreover $V$ does not vanish at some point, then 
	\begin{equation}
		\textnormal{Ip}\hspace{-2pt}\pare{\Mt\pare{V,W}}\equiv \Mt\pare{V,\textnormal{Ip}\hspace{-2pt}\pare{W}}.
	\end{equation}
\end{theorem}	

%\begin{corollary}
%	Let $V\subseteq F\pare{D,\K}$ and $W\subseteq F\pare{D,X}$ be $\BK$-spaces, and assume moreover that $V$ does not vanish at some point. If $V$ or $W$ is $b$-invariant, then 
%	\begin{equation}
%		\textnormal{Ip}\hspace{-2pt}\pare{\Mt\pare{V,W}}= \Mt\pare{V,W}.
%	\end{equation}
%\end{corollary} 

\begin{corollary}\label{d332_3}
	Let $Y \subseteq F\pare{D,\K}$ and $W \subseteq F\pare{D,X}$ be $\BK$-spaces, such that $Y$ does not vanish at some point. If $Y$ is a Banach ideal, then  
	\begin{equation}
		\Mt\pare{Y,W}\equiv \Mt\pare{Y,\textnormal{Ip}\hspace{-2pt}\pare{W}}.
	\end{equation} 
	 In particular, if $V\subseteq F\pare{D,\K}$ is a $BKN$-space, then
	$M'(\textnormal{Ip}\hspace{-2pt}\pare{V}, W) = M'(\textnormal{Ip}\hspace{-2pt}\pare{V}, \textnormal{Ip}\hspace{-2pt}\pare{W})$ and therefore \begin{equation}
		M'(V,W) \subseteq  M'\pare{\textnormal{Ip}\hspace{-2pt}\pare{V}, \textnormal{Ip}\hspace{-2pt}\pare{W}}.
	\end{equation}
	In this case, the corresponding diagram (\ref{lola121}) is also commutative.
\end{corollary} 

\begin{example}\label{d282_1} \label{d216_1}
	If $Y\subseteq S\pare{\K}$ is a Banach $b$-ideal that does not vanish at some point, then  
	$\Mt\pare{Y, c\pare{X}}\equiv \Mt\pare{Y,c_{0}\pare{X}}.$
\end{example}

 	\section{v-Ideal}\label{lalo1}\label{pop8}
 	
 		Let $D$ be a non-empty set and $X$ a Banach space. We define the \emph{norm scalarization operator}
 	$\J:F(D,X)\tien F(D,\K)$  by
 	\begin{equation}\label{q20_0}
 		\J g(a):=\norm{g(a)},\qquad a\in D.
 	\end{equation} 
 	Let us now take $D = \N$. Then the norm scalarization of $s \in S(X)$ is
 	\begin{equation}
 		\J s=\llave{\norm{s(n)}}_{n}\in  S(\K).
 	\end{equation}  
 	
 	The concept of order ideal can also be extended to  the vector-valued case as follows. (See \cite[p.~800]{V}.)
 
 \begin{definition} 
 	Let $U$ be a vector subspace of $F\pare{D,X}$. Then $U$ is a \emph{v-ideal} if for each $f\in F\pare{D,X}$ and $g\in U$ such that $0\leq \J f\leq \J g$, it follows that $f\in U$. 
 \end{definition}
 
The proof of the following result is straightforward.
 
 \begin{lemma}\label{d230_1}
 	If $U\subseteq F\pare{D,X}$ is a v-ideal, then it is a   $b$-ideal.
 \end{lemma} 
 
 Instead of Proposition~\ref{lola99_0} for $b$-ideals, we now have the following result.

 \begin{proposition}   Let $V \subseteq F\pare{D, \K}$ and $W \subseteq F\pare{D, X}$ be vector spaces. If  $W$ is a v-ideal, then $M'\pare{V, W}$ is a v-ideal.  
 \end{proposition}
 
 \begin{definition} 
 	Let $U\subseteq F\pare{D,X}$ be a normed space which is a v-ideal. Then  $U$ is a \emph{normed v-ideal} if for each $f,g\in U$ such that $\J f\leq \J g$, it follows that $\norm{f}\leq \norm{g}.$ If, in addition, $U$ is complete, we say that $U$ is a \emph{Banach v-ideal}.  
 \end{definition}

Note that for $Y \subseteq F(D,\K)$, the notions of order ideal, $b$-ideal, and v-ideal are equivalent. The same holds for normed ideals, normed $b$-ideals, and normed v-ideals. 

The proof of the following result can be obtained along the same lines as that of Lemma \ref{d230_1} .
 
 \begin{proposition}  If $U\subseteq F\pare{D,X}$ is a  normed  v-ideal, then  $U$ is a normed $b$-ideal.
 \end{proposition}
 
 In Corollary~\ref{yaya4}, we provide an example of a Banach $b$-ideal that is not a $v$-ideal.\vs

The following properties of $v$-ideals correspond to properties of $b$-ideals. 
We only prove Proposition \ref{final7} , which is related to Proposition \ref{final8}.
 
 \begin{corollary}\label{final6}
 	Let $V,W\subseteq F\pare{D,X}$ be a normed spaces.  If    $V$ or $W$ is a normed v-ideal, then  $M\pare{V,W}$ is a normed ideal.  
 \end{corollary}   
 
 \begin{proposition}   Let $V \subseteq F\pare{D, \K}$ and $W \subseteq F\pare{D, X}$ be a normed spaces. If  $W$ is a normed v-ideal, then  $M'\pare{V,W}$ is a normed v-ideal. 
 \end{proposition}
 
 \begin{proposition}\label{final7}
 	Let $U \subseteq F(D,X)$ be a Banach space that is also a v-ideal. Then $U$ has an equivalent norm under which $U$ is a Banach v-ideal if, and only if,
 	$k:=\sup \llave{ \frac{ \|f\|}{\|g\|}:   f, g  \in U, J_f \leq Jg } < \infty$.
 \end{proposition}
 \begin{proof}
First, suppose that $U$ has a an equivalent norm   $\norm{\cdot}_{v}$ with respect to which $U$ is a Banach v-ideal. Then there exist $c,d\in(0,\infty)$ such that
$
c\norm{u}\leq \norm{u}_{v}\leq d\norm{u},  \forall u\in U.
$
Let $f,g\in U$ be such that $\J f\leq \J g$. Then
$
c\norm{f}\leq \norm{f}_{v}\leq \norm{g}_{v}\leq d\norm{g},
$
and therefore
$
k <\infty.
$ 

Now suppose that $k<\infty$. Consider the function $\norm{\cdot}_{v}:U\tien[0,\infty]$ defined by
\begin{equation}
	\norm{u}_{v}:=\sup\llave{\norm{h}:Jh\leq Ju}.
\end{equation}
To see that $\norm{\cdot}_{v}$ is a norm on $U$, we only prove the triangle inequality. Let $f,g,h\in U$ be such that $Jh\leq J\pare{f+g}$.
Given $d\in D$, it follows that
$
\norm{h\pare{d}}
\leq
\norm{f\pare{d}+g\pare{d}}
\leq
\norm{f\pare{d}}+\norm{g\pare{d}}.
$
Now define
\begin{equation}
	h_{1}\pare{d}:=
	\frac{\norm{f\pare{d}}h\pare{d}}{\norm{f\pare{d}}+\norm{g\pare{d}}},
	\qquad
	h_{2}\pare{d}:=
	\frac{\norm{g\pare{d}}h\pare{d}}{\norm{f\pare{d}}+\norm{g\pare{d}}},
	\qquad
	\text{if }\norm{f\pare{d}}+\norm{g\pare{d}}>0,
\end{equation}
and let $h_{1}(d)=h_{2}(d)=0$ whenever
$
\norm{f\pare{d}}+\norm{g\pare{d}}=0.
$
Consider the functions $h_{1}:D\tien X$ and $h_{2}:D\tien X$. Clearly, $Jh_{1}\leq Jf$, $Jh_{2}\leq Jg$, and $h=h_{1}+h_{2}$. Since $U$ is a $v$-ideal, it follows that $h_{1}, h_{2}\in U$. Therefore,
\begin{equation}
	\norm{h}=\norm{h_{1}+h_{2}}\leq \norm{h_{1}}+\norm{h_{2}}\leq \norm{f}_{v}+\norm{g}_{v},
\end{equation}
and it follows that 
$
\norm{f+g}_{v}\leq \norm{f}_{v}+\norm{g}_{v}.
$
Clearly
\begin{equation}\label{yaya6}
	\norm{u}\leq \norm{u}_{v},\qquad \forall u\in U.
\end{equation}
Moreover, if $f,g\in U$ satisfy $Jf\leq Jg$, then $\norm{f}_{v}\leq \norm{g}_{v}$. To prove that $\norm{\cdot}_{v}$ is complete, it is enough to show that it is equivalent to the original norm on $U$. 
Let $f,g\in U$ be such that $Jg\leq Jf$. Therefore
$
\norm{g}\leq k\norm{f}.
$
Taking the supremum over such  $g$, we obtain
$
\norm{f}_{v}\leq k\norm{f}.
$
Together with (\ref{yaya6}), this shows that the norm on $U$ and $\norm{\cdot}_{v}$ are equivalent.
 \end{proof}

	\subsection{The Strong Vectorialization $Y\pare{X}$}\label{sec8}

	Let $Y\subseteq S(\K)$ be an order ideal  and define
	\begin{equation}\label{d136}
		Y(X):=\llave{s\in S(X):\J s\in Y}.
	\end{equation}   In this situation $Y(X)$ is also a vector space over $\K$, which we will call  the \emph{strong vectorialization of $Y$ (or $X$-strong vectorialization)}.

	The main part of the following result can be obtained as a direct  consequence of Corollary 9 in \cite[p. 811]{V}.  
	
	\begin{theorem}\label{piz24}
		If $Y\subseteq S(\K)$ is a Banach ideal, then $Y(X)$ endowed with the norm
		\begin{equation}
			\norm{s}_{Y(X)}:=\norm{\J s}_Y
		\end{equation}
		is a  Banach v-ideal. Moreover, if $Y$ does not vanish at any point and $X\neq \llave{0}$, then neither does $Y\pare{X}$. 
	\end{theorem}
	
	Next we verify that the usual definitions of the norm spaces  $c_{00}(X)$, $c_{0}(X)$ and $\ell_{p}(X)$ $(1\le p\le \infty)$ (For example, see \cite[p. 32]{DJT}). coincide with the vector-valued sequence spaces $Y(X)$ obtained by applying the previous construction to the normed ideals $Y = c_{00} (\K)$, $Y=c_{0}(\K)$ and $Y=\ell_{p}(\K)$, respectively.

	\subsection{Examples: Spaces $c_{00}(X)$, $\ell_p(X)$ and $c_0(X)$}
	
	Taking the order ideal $Y=c_{00}(\K)$ we obtain
	\[
	Y(X)=\llave{\llave{x_n}_n\in S(X):\llave{\norm{x_n}}_n\in c_{00}(\K)}=c_{00}(X).
	\]
	
	Let $1\le p<\infty$ and consider the Banach ideal $Y=\ell_p(\K)$. Then 
	\begin{equation}
		Y(X) =\llave{\llave{x_n}\in S(X):\llave{\norm{x_n}}_n\in\ell_p(\K)} =\llave{\llave{x_n}\in S(X):\sum_{n=1}^{\infty}\norm{x_n}^p<\infty}
		=\ell_p(X).
	\end{equation}
	Moreover,
	\begin{equation}
		\norm{\llave{x_n}_n}_{Y(X)}
		=\norm{\llave{\norm{x_n}}_n}_{Y}
		=\left(\sum_{n=1}^{\infty}\norm{x_n}^p\right)^{1/p}
		=\norm{\llave{x_n}_n}_{\ell_p(X)} .
	\end{equation}
	
	In particular, $\ell_1(X)$ consists of absolutely summable sequences and
	\begin{equation}
		\norm{\llave{x_n}_n}_{\ell_1(X)}=\sum_{n=1}^{\infty}\norm{x_n}.
	\end{equation}
	
	Next, let us take  $Y=\ell_\infty(\K)$, then
	\begin{equation}
		Y(X)=\llave{\llave{x_n}_n\in S(X):\sup_{n\in\N}\norm{x_n}<\infty}
		=\ell_\infty(X),
	\end{equation}
	and
	\begin{equation}
		\norm{\llave{x_n}_n}_{Y(X)}
		=\norm{\llave{\norm{x_n}}_n}_{Y}
		=\sup_{n\in\N}\norm{x_n}
		=\norm{\llave{x_n}_n}_{\ell_\infty(X)} .
	\end{equation}
	
	Finally, if $Y=c_0(\K)$, then
	\begin{equation}
		Y(X)=\llave{\llave{x_n}\in S(X):\llave{\norm{x_n}}_n\in c_0(\K)}
		=c_0(X),
	\end{equation}
	and
	\begin{equation}
		\norm{\llave{x_n}_n}_{Y(X)}
		=\norm{\llave{\norm{x_n}}_n}_{\ell_\infty}
		=\sup_{n\in\N}\norm{x_n}
		=\norm{\llave{x_n}_n}_{\ell_\infty(X)} .
	\end{equation}

	\section{The Spaces $\Sigma \ell_{\infty}(X)$ and $\Sigma c(X)$}\label{sec9}\label{pop9}
	
	Let $X$ be a Banach space and consider a Banach space $E \subseteq S(X)$. Our aim is now to construct a Banach space $\Sigma E$ that will play a relevant role in the definition of the multiplier space $M_{\Sigma}\pare{V,X}$, which will be carried out in the next section.    We start  by considering some general definitions and properties that will be needed. Thereafter, we study the properties of $\Sigma E$ in the cases $E=\ell_{\infty}\pare{X}$ and $E=c\pare{X}$.

	\subsection{Projections and  the Space $V_{0}$}
	 Let $s=\{x_{n}\}_{n}\subseteq X$ and $n\in \N$. Define
	\begin{equation}\label{d126}
		P_{n}(s): = \Sum_{k=1}^{n}e_{k}x_{k}= (x_{1}, \dots, x_{n}, 0, 0, 0, \dots).
	\end{equation}  Then  $P_{n}: S(X) \to S(X)$ is a linear operator such that $P_{n}^{2} = P_{n}$ and $R(P_{n}) \subseteq c_{00}(X)$, where $R(P_{n})$ denotes the range of $P_{n}$. In the case $X = \K$, we will use $\pi_{n}$ instead of $P_{n}$.
	
	Let $V \subseteq S(X)$ be a normed space. We are interested in the case where $P_{n}(V) \subseteq V$ and the operator $P_{n}: V \to V$ is continuous. In this situation, $P_{n}$ is a projection, and we will refer to it as the \emph{$n$-th canonical projection of $V$}. 
	
	The following result is well known.
	\begin{example}\label{lola129}
		If $X \neq \{0\}$ and $c_{00}(X) \subseteq V$, then $\|P_n\| \geq 1, \forall n \in \mathbb{N}$.
	\end{example}  
	
	\begin{example}\label{d293}
		Take $x\in X$ such that $x\neq 0$, and let $s=\llave{x}_{n}$ be the constant sequence. Consider the generated space $V=\G\pare{\llave{s}}$, and observe that $P_{n}s\notin V$, $\forall n\in \N$.
	\end{example}
	
	Let $V\subseteq S\pare{X}$ be a vector space and fix $n\in \N$. Given $s\in V$, as Example~\ref{d293} shows, it may happen that $P_{n}s\notin V$. A property that prevents this situation is $c_{00}\pare{X}\subseteq V$. In this case, $P_{n}:V\tien V$, and we shall see below that $P_{n}$ is bounded.

	\begin{proposition}\label{d268_0_a_0} \label{d285_a_0}
		Let $V\subseteq S\pare{X}$ be a $\BK$-space such that $c_{00}\pare{X}\subseteq V$. Then each canonical projection $P_{n}:V\tien V$ is bounded. 
	\end{proposition} 
	\begin{proof}
		i. Fix $n\in \N$. By Lemma~\ref{d278_1}, there exists $r'>0$ such that
		\begin{equation}\label{d284}
			\norm{e_{k}x}\leq r'\norm{x},
			\qquad
			\forall k\in \llave{1,\dots,n},
			\qquad
			\forall x\in X.
		\end{equation}
		Since $V$ has continuous evaluations, choose $r\geq r'$ such that
		$
		\norm{s\pare{k}}\leq r\norm{s},
		$
		$\forall k\in \llave{1,\dots,n}$ and $\forall s\in V$. Using this together with (\ref{d284}), we obtain
		$
		\norm{P_{n}\pare{s}}
		=
		\norm{\sum_{k=1}^{n}e_{k}s\pare{k}}
		\leq
		\sum_{k=1}^{n}\norm{e_{k}s\pare{k}}
		\leq
		\sum_{k=1}^{n}r\norm{s\pare{k}}
		\leq
		nr^{2}\norm{s}, 
		\forall s\in V.
		$ 
	\end{proof}

	\begin{definition}\label{naty22}
		Let $V\subseteq S\pare{X}$ be a normed space such that $c_{00}\pare{X}\subseteq V$.
		\begin{enumerate}
			\item The \emph{projection coefficient of $V$} is the value $\nu_{V}:=\sup_{n\in \N}\norm{P_{n}}\in[1,\infty]$, where $P_{n}:V\tien V$ is the $n$-th canonical projection.
			\item The space $V_{0}$ is the closure of $c_{00}(X)$ in $V$.
		\end{enumerate}
	\end{definition}

	\begin{lemma}\label{efren27_2}
		Let $V \subseteq S\pare{X}$ be a $\BK$-space such that $c_{00}\pare{X} \subseteq V$. Then:
		\begin{enumerate}
			\item[i.]  $\nu_V < \infty$ if and only if, for every $v \in V$, there exists $C_v > 0$ such that $\| P_n v \| \leq C_v$, for all $n \in \N$.
			\item[ii.] If $s \in S\pare{X}$ and $\llave{P_{n}s}_{n}$ is Cauchy in $V$, then $s \in V_{0}$ and $P_{n}s \to s$.
			\item[iii.] $P_{n}v \to v$, $\forall v \in V_{0}$ if and only if $\nu_{V_{0}} < \infty$.
		\end{enumerate}  
	\end{lemma}
	\begin{proof}i. First, suppose that $\nu_V < \infty$. Given $v\in V$, we have
	$
		\norm{P_{n}v}_{V}\leq \nu_{V}\norm{v}_{V}$,  $\forall n\in \N.
	$  The converse implication follows from the Uniform Boundedness Theorem, which yields
	$
	\sup_{n \in \mathbb{N}} \norm{P_{n}} < \infty.
	$
		
		ii. Let $s \in S\pare{X}$ be such that $\llave{P_{n}s}_{n}$ is Cauchy in $V$. Then there exists $z \in V$ such that $P_{n} \to z$. Fix $m \in \N$. Then, using  that $V$ is a $\BK$-space, we have that $s\pare{m} = \lim_{n \to \infty} P_{n}s\pare{m} = z\pare{m}.$
		Thus, $s = z$.
		
		iii. Let us assume that $P_{n} v \tien v, \forall v \in V_{0}$. Then, we have $\sup_{n \in \mathbb{N}} \norm{P_{n}v} < \infty, \forall v \in V_{0}$. Now, applying the Uniform Boundedness Theorem, it follows that $\sup_{n \in \mathbb{N}} \norm{P_{n}} < \infty$.
		
		Now, suppose that $\nu_{V_{0}} < \infty$. Take $v \in V_{0}$. Given $\epsilon > 0$, there exists $c \in c_{00}\pare{X}$ such that
		\begin{equation}
			\norm{v - c} < \frac{\epsilon}{2\pare{1 + \nu_{0}}}.
		\end{equation}
		Since $c \in c_{00}\pare{X}$, choose $N \in \N$ such that $P_{n}c = c$ for all $n \geq N$. Then,
		$$
		\norm{P_{n}v - v} \leq \norm{P_{n}v - P_{n}c} + \norm{c - v} \leq \frac{\epsilon}{2} + \frac{\epsilon}{2} = \epsilon, \quad \forall n \geq N.\qedhere
		$$\qedhere
	\end{proof}
	
		Let $V$ be a Banach space. Recall  a sequence $\llave{v_{n}}_{n}\subseteq V$ is a \emph{Schauder basis} of $V$ if for every $v\in V$, there exists a unique sequence $\llave{a_{n}}_{n}\subseteq \K$ such that $v=\sum_{n=1}^{\infty}a_{n}v_{n}$, that is, $\lim_{n\tien \infty}\norm{v-\sum_{i=1}^{n}a_{i}v_{i}}=0.$ In this case,
	for each $n\in \N$ we define the $n$-th \emph{coefficient functional} $v_{n}^{*}:V\tien \K$ by $\la v,v_{n}^{*}\ra:=a_{n}, \forall n\in \N.$ Recall also that  $\{v_{n}\}_{n}$ is a \emph{basic sequence} of $V$ if $\llave{v_{n}}_{n}$ is a basis of the closure of the vector space generated by the sequence $\llave{v_{n}}_{n}$  in $V$. 
	
	\begin{proposition}\label{d286_0}
		Let $V \subseteq S(\K)$ be a Banach space such that $\{e_{n}\}_{n} \subseteq V$. Then, $\{e_{n}\}_{n}$ is a basis of $V$ if and only if $V$ has continuous evaluations and
		\begin{equation}\label{d279}
			\pi_{n}v \tien v, \qquad \forall v \in V.
		\end{equation}
	\end{proposition}
	
	\begin{proof}
		First, suppose that $\{e_{n}\}_{n} \subseteq V$ is a basis. Clearly, condition (\ref{d279}) holds. Fix $n \in \N$. Since the coefficient functional $\var_{n}$ is bounded, we have that $	|v(n)| = |\var_{n}v| \leq \|\var_{n}\|\|v|,   \forall v \in V.$
		Thus, $V$ has continuous evaluations. Now, assume that $V$ has continuous evaluations and that condition (\ref{d279}) holds. Let $v \in V$. From (\ref{d279}), it follows that $v = \sum_{n=1}^{\infty} v(n) e_{n}$. Now, let us show that the sequence $\{v(n)\}_{n} \subseteq \K$ is the unique sequence satisfying the above equality. Suppose that $v = \sum_{n=1}^{\infty} a_{n} e_{n}$, where $\{a_{k}\}_{k} \subseteq \K$. Then, $\sum_{n=1}^{\infty} (a_{n} - v(n)) e_{n} = 0,$
		and since $V$ has continuous evaluations, it follows that
		$
		0 = \sum_{n=1}^{\infty} (a_{n} - v(n)) e_{n}(m) = a_{m} - v(m),   \forall m \in \N.
		$
		Therefore, $a_{m} = v(m),  \forall m \in \N$.
	\end{proof} 
	
		The following corollary is an immediate consequence of Lemma~\ref{efren27_2} and Proposition \ref{d286_0}.
	
	\begin{corollary}\label{lola199_0} \label{lola209_0}
		Let $V \subseteq S(\K)$ be a $\BK$-space such that $\{e_{n}\}_{n} \subseteq V$. Then, $\{e_{n}\}_{n}$ is a basic sequence in $V$ if and only if $\nu_{V_{0}} < \infty$. In particular, this holds if $\nu_{V} < \infty$.
	\end{corollary}
	
		\begin{example}\label{again}
		If $V \subseteq S(X)$, $V\neq \llave{0}$, is a Banach b-ideal, then $\nu_{V}=1$.
	\end{example} 
	
The following results deal with spaces of multipliers involving $V_0$ and $W_0$ spaces. We only prove the first one.
		\begin{proposition}\label{chuy18_0_1_0}
		Let  $V \subseteq S\pare{X}$ and $W \subseteq S\pare{X}$ be $\BK$-spaces. 
		If $c_{00}\pare{X} \subseteq V$ and $c_{00}\pare{X} \subseteq W$, then 
		\begin{enumerate}
			\item[i.] $M\pare{V_{0},W_{0}} = M\pare{V_{0},W}$. 
			\item[ii.] $M\pare{V,W}_{0} \subseteq  M\pare{V,W_{0}}$. 
		\end{enumerate} 
	\end{proposition} 
	\begin{proof}
		i. Clearly, $M\pare{V_{0},W_{0}} \subseteq M\pare{V_{0},W}$. 
		Take $s \in M\pare{V_{0},W}$ and $v \in V_{0}$. 
		Then there exists a sequence $\llave{v_{k}}_{k} \subseteq c_{00}\pare{X}$ such that $v_{k} \to v$. 
		Hence, $s \cdot v_{k} \in c_{00}\pare{X}$ for all $k \in \N$, and $s \cdot v_{k} \to s \cdot v$. 
		It follows that $s \cdot v \in W_{0}$, and therefore $s \in M\pare{V_{0},W_{0}}$.
		
		ii. Let $s\in M\pare{V,W}_{0}$. Then there exists $\llave{s_{n}}_{n}\subseteq c_{00}\pare{X}$ such that $s_{n}\to s$. Taking $v\in V$, we have $s_{n}\cdot v\to s\cdot v$, where $s_{n}\cdot v\in c_{00}\pare{X}, \forall n\in \N$. Hence, $s\in M\pare{V,W_{0}}$.
	\end{proof} 
		\begin{proposition}\label{chuy18_0_1_t_0}
		Let $V \subseteq S\pare{\K}$ and $W \subseteq S\pare{X}$ be $\BK$-spaces. 
		If $c_{00}\pare{\K} \subseteq V$ and $c_{00}\pare{X} \subseteq W$, then 
		\begin{enumerate}
			\item[i.] $	M'\pare{V_{0},W_{0}} = M'\pare{V_{0},W}.$
			\item[ii.]  $M'\pare{V,W}_{0} \subseteq  M'\pare{V,W_{0}}$. 
		\end{enumerate} 
	\end{proposition}
	
	\begin{lemma}
		Let  $U \subseteq S\pare{X}$ be a BK-space such that $c_{00} (X) \subseteq U$. If $U$ is a b-ideal, then
		$U_0$ is also a $b$-ideal.
	\end{lemma}
	\begin{proof}
Let $s\in U_{0}$. Then there exists $\llave{s_{n}}_{n}\subseteq c_{00}\pare{X}$ such that $s_{n}\to s$. Given $b\in \ell_{\infty}\pare{\K}$, Proposition~\ref{final8} implies that $b\cdot s_{n}\to b\cdot s$. Hence, $b\cdot s \in U_{0}$.
	\end{proof}
	
	\subsection{$\Sigma E$ Spaces}
	
	Let $X$ be a Banach space and $E \subseteq S(X)$ a normed space. The concept of the sequence spaces $\Sigma E$ is often used implicitly; however, stating their definition explicitly helps to emphasize their significance. As we will see, these kind of spaces arise from classical spaces through a natural and fairly general construction. \vs
	
	Let $s := \{x_k\}_k \in S(X)$. Then the \emph{$n$-th partial sum} of $\sum_k x_k$ is defined as
	\begin{equation}
		S_n(s) := \sum_{k=1}^n x_k \in X, \quad \forall n \in \mathbb{N}.
	\end{equation}
	Next, we define the sequence of partial sums
 
	\begin{equation}
		\sigma_s := \{S_n(s)\}_n \subseteq X, 
	\end{equation}
	and the function $T_{\sigma}: S(X) \to S(X)$ by
	\vspace{-.2 cm}
	\begin{equation}
		T_{\sigma} s := \sigma_s. 
	\end{equation}
	The following result  can be proved directly.
	\begin{lemma}
		$T_{\sigma}: S(X) \to S(X)$ is an isomorphism.
	\end{lemma}
	
	Let $E \subseteq S(X)$ be a normed space, and let us consider the collection
	\begin{eqnarray}\label{f54}
		\Sigma E := T_{\sigma}^{-1} E = \{s \in S(X) : \sigma_s \in E\}.
	\end{eqnarray}
	Thus, $\Sigma E$ consists of the sequences in $X$ whose corresponding sequence of partial sums belongs to $E$. Notice now that $T_E := T_{\sigma} : \Sigma E \to E$ is an isomorphism, and we consider in $\Sigma E$ the norm $\| \cdot \|_{\Sigma E}$ induced by $T_E$, that is,
	\begin{equation}
		\|s\|_{\Sigma E} := \|T_E s\|.
	\end{equation}
	Then $T_E: \Sigma E \to E$ is an isometric isomorphism.\vs
	
	The following results show that $\Sigma E$ inherits several properties from $E$.
	
	\begin{proposition}\label{g1_0_new} Let $E\subseteq S\pare{X}$ be a normed space.
		\begin{enumerate}
			\item[i.]  If $E$ is a Banach space, then $\Sigma E$ is a Banach space.
			\item[ii.]   If $E$ has continuous evaluations, then $\Sigma E$ has continuous evaluations.
			\item[iii.] If $c_{00}\pare{X}\subseteq E$, then $c_{00}\pare{X}\subseteq \Sigma E$.
		\end{enumerate} 
	\end{proposition} 
\begin{proof}
	i. This follows directly from the fact that $T_{E}$ is an isometric isomorphism.
	
	ii. Fix $s=\llave{x_{n}}_{n}\in \Sigma E$, and consider $\sigma_{s}=\llave{S_{n}\pare{s}}_{n}\in   E$. Then, for every $n\in \N$, there exists $C_{n}\in(0,\infty)$ such that 
	\begin{equation}
		\norm{\sigma_{s}\pare{n}}_{X}\leq C_{n}\norm{\sigma_{s}}_{E}.
	\end{equation}\label{zip1}
	Notice that
	\begin{equation}
		\norm{x_{1}}_{X}=\norm{s\pare{1}}_{X}=\norm{\sigma_{s}\pare{1}}_{X}\leq C_{1}\norm{\sigma_{s}}_{E}=C_{1}\norm{s}_{\Sigma E}.
	\end{equation}
	For $n>1$, applying (\ref{zip1}), we obtain
	\begin{equation}
		\norm{x_{n}}_{X}=\norm{S_{n}\pare{s}-S_{n-1}\pare{s}}_{X}\leq \norm{\sigma_{s}\pare{n}}_{X}+\norm{\sigma_{s}\pare{n-1}}_{X} \leq  \pare{C_{n}+C_{n-1}}\norm{s}_{\Sigma E}.
	\end{equation}
	Hence, $\Sigma E$ has continuous evaluations.
	
	iii. Let $n \in \N, x \in X$. Then  $e_{n}x\in E$, it follows that $xe_{n}\in \Sigma E$. Therefore, $c_{00}\pare{X}\subseteq \Sigma E$.
\end{proof}

\begin{example}
	Let $X=\K$ and 
	$
	E=\corchete{\llave{1}_{n}}
	$,  the space generated by $\llave{1}_{n}\in S\pare{\K}$.
	Then $E\subseteq S\pare{\K}$ is a vector space that does not vanish at any point. However,
	$
	\Sigma E=\corchete{e_{1}},
	$
	which vanishes at infinitely many points.
\end{example}
\begin{proof}
	Let $s\in \Sigma E$. Then there exists $a\in \K$ such that $\sigma_{s}=a\llave{1}_{n}$. Therefore,
	$
	a=S_{1}(s)=s\pare{1}.
	$
	Moreover, for each $n\in \N\setminus \llave{1}$,
	$
	a=S_{n}\pare{s}=s\pare{1}+s\pare{n}=a+s\pare{n}.
	$
	Hence, $s\pare{n}=0$ for every $n\geq 2$, and consequently $s=ae_{1}$.
\end{proof}
	
%		\begin{lemma}\label{lola155_c} 
%			Let $J:X\to W$ be an isometric isomorphism, $\widetilde{J}:S\pare{X}\to S\pare{W}$ its associated sequence operator, and $E\subseteq S(X)$, $F\subseteq S\pare{W}$ Banach spaces. If $\widetilde{J}:E\to F$ is an isometric isomorphism, then $\widetilde{J}:\Sigma E \to \Sigma F$ is an isometric isomorphism.
%		\end{lemma}
%		
%		\begin{proof}  
%			Clearly $\widetilde{J}$ is injective. Now observe that
%			\begin{equation}\label{lola172}
%				\sigma_{\widetilde{J}\pare{s}} = \{S_{n}\pare{\widetilde{J}\pare{s}}\}_{n} = \{J\pare{S_{n}\pare{s}}\}_{n} = \widetilde{J}\pare{\sigma_{s}} \in F, \qquad \forall s\in \Sigma E.
%			\end{equation}
%			Similarly, we prove that $\sigma_{\widetilde{J}^{-1}\pare{z}} = \widetilde{J}^{-1}\pare{\sigma_{z}} \in E$ for all $z\in \Sigma F$. Fix $z\in \Sigma F$. Then $s:=\widetilde{J}^{-1}\pare{z}\in \Sigma E$, and therefore $\widetilde{J}\pare{s}=z$. Thus, $\widetilde{J}$ is surjective. Let $s\in \Sigma E$. Using (\ref{lola172}), it follows that $\norm{\widetilde{J}\pare{s}}_{\Sigma F} = \norm{\sigma_{\widetilde{J}\pare{s}}}_{F} = \norm{s}_{\Sigma E}$.
%		\end{proof}

	\subsection{The Spaces $\Sigma \ell_{\infty}(X)$ and $\Sigma c(X)$} 
	
	Taking $E = \ell_{\infty}(X)$ in the  construction of $\Sigma E$, it follows that the vector space
	\begin{equation}\label{d264_a} 
		\Sigma \ell_{\infty}(X) = \{s \in S(X): \sigma_s \in \ell_{\infty}(X)\}
	\end{equation} 
	is a Banach space with the norm 
	\begin{equation}
	\|s\|_{\Sigma \ell_{\infty}(X)} = \sup \left\{ \left\| \sum_{k=1}^n x_k \right\| : n \in \mathbb{N} \right\}, \quad s = \{x_{k}\}_{k}. 
	\end{equation}
    Note the elements of $\Sigma \ell_{\infty}(X)$ are those sequences in $X$ with bounded partial sums.	
    
    We also have   $c_{00}(X) \subseteq \Sigma \ell_{\infty}(X) \subseteq \ell_{\infty}(X)$, and
	\begin{equation}
		\|x_n\| \leq 2 \|s\|_{\Sigma \ell_{\infty}(X)}, \quad \forall n \in \mathbb{N}, \ \forall s = \{x_k\}_k \in \Sigma \ell_{\infty}(X).
	\end{equation}
	 In conclusion,  $\Sigma \ell_{\infty}(X)$ is a $\BKn$-space.

	Next, taking $E = c(X)$, we obtain that the vector space
	\begin{equation}\label{d264_b} 
		\Sigma c(X) = \{s \in S(X): \sigma_s \in c(X)\} \subseteq \Sigma \ell_{\infty}(X)
	\end{equation}
	is a Banach space with the norm  $\|s\|_{\Sigma c(X)} = \|s\|_{\Sigma \ell_{\infty}(X)}.$ Thus, $\Sigma c(X)$ is a closed subspace of $\Sigma \ell_{\infty}(X)$.
	Hence $\Sigma c(X)$ has continuous evaluations. Moreover,   $c_{00}(X) \subseteq \Sigma c(X) \subseteq c_0(X)$, and
	\begin{equation}\label{d245}  
		\norm{\sum_{k=1}^{\infty} x_k} \leq \|s\|_{\Sigma c(X)}, \quad \forall s = \{x_k\}_k \in \Sigma c(X).
	\end{equation}
	Thus, $\Sigma c(X)$ is a $\BKn$-space. Note that $\{x_n\}_n \in \Sigma c(X)$ if, and only if, the series $\sum_n x_n$ converges. Therefore,   the elements of $\Sigma c(X)$ are called \emph{summable sequences}. In this case, we define the function $S: \Sigma c(X) \to X$ by
	\begin{equation}\label{d247}
		S(\{x_n\}_n) := \sum_{n=1}^{\infty} x_n.
	\end{equation}
	Clearly, $S$ is a linear operator, which we will call \emph{the sum operator of $X$}. By (\ref{d245}), this operator is continuous and $\|S\| \leq 1$.

	\begin{lemma}\label{d269}
		Let $X$ be a Banach space. Then $e_n x \in \Sigma c(X)$, $\forall x \in X, \forall n \in \mathbb{N}$. Moreover, for  $P_n: \Sigma \ell_{\infty}(X) \to \Sigma \ell_{\infty}(X)$ we have
		\begin{enumerate}  
			\item[i.]   $\|P_n s\|_{\Sigma \ell_{\infty}(X)} \tien \|s\|_{\Sigma \ell_{\infty}(X)},  \forall s \in \Sigma \ell_{\infty}(X)$. Moreover, if $X \neq \{0\}$, then $\|P_n\| = 1$, $\forall n \in \mathbb{N}$. Therefore, $\nu_{\Sigma \ell_{\infty}(\mathbb{K})} = 1$.
			\item[ii.]  $s \in \Sigma c(X)$ if, and only if, $P_n s \to s$. Therefore, $\Sigma c(X) = (\Sigma \ell_{\infty}(X))_0$.
		\end{enumerate}
	\end{lemma}
	\begin{proof}
		Clearly, $e_n x \in \Sigma c(X)$, for all $x \in X$, and all $n \in \mathbb{N}$.
		
		i. Let $s \in \Sigma \ell_{\infty}(X)$ and $n \in \mathbb{N}$. Then, $\|P_n s\|_{\Sigma \ell_{\infty}(X)} \leq \|P_{n+1} s\|_{\Sigma \ell_{\infty}(X)} \leq \|s\|_{\Sigma \ell_{\infty}(X)}.$  
		We will now prove that $\sup_{n \in \mathbb{N}} \|P_n s\|_{\Sigma \ell_{\infty}(X)} = \|s\|_{\Sigma \ell_{\infty}(X)}$. First, observe that \begin{equation}
			\norm{\sum_{n=1}^{k} s(n)} \leq \|P_k s\|_{\Sigma \ell_{\infty}(X)} \leq \sup_{n \in \mathbb{N}} \|P_n s\|_{\Sigma \ell_{\infty}(X)} \leq \|s\|_{\Sigma \ell_{\infty}(X)}, \quad \forall k \in \mathbb{N}.
		\end{equation}
		Therefore, $\|s\|_{\Sigma \ell_{\infty}(X)} \leq \sup_{n \in \mathbb{N}} \|P_n s\|_{\Sigma \ell_{\infty}(X)} \leq \|s\|_{\Sigma \ell_{\infty}(X)}$. 
		
		Assume that $X \neq \{0\}$, and fix $m \in \mathbb{N}$. Since $\|P_n s\|_{\Sigma \ell_{\infty}(X)} \uparrow \|s\|_{\Sigma \ell_{\infty}(X)}$, it follows that $\|P_m\| \leq 1$. Now, applying example \ref{lola129}, we get that $\|P_n\| = 1$.
		
		ii. Assume $s \in \Sigma c(X)$, and let $n \in \mathbb{N}$. Notice that $$\norm{s - P_n s}_{\Sigma_{c}\pare{X}} = \norm{(0, \dots, 0, x_{n+1}, x_{n+2}, \dots)}_{\Sigma_{c}\pare{X}}=\sup\llave{\norm{\Sum_{k=1}^{m}x_{n+k}}:m\in \N}\tien 0,\text{when } n\tien \infty.$$ Thus $P_n s \to s$. Conversely, suppose $P_n s \to s$. Then the sequence of partial sums $\llave{\sum_{k=1}^{n}x_{k}}_{n}$ is a Cauchy sequence, and therefore the series $\sum_n x_n$ converges, and hence, $s \in \Sigma c(X)$. Thus, $\Sigma c(X) = (\Sigma \ell_{\infty}(X))_0$.
	\end{proof}
	
		\begin{example}\label{d333}
		If $X \neq \{0\}$, then  $\Sigma c\pare{X}$ and $\Sigma \ell_{\infty}\pare{X}$ are not b-ideals. Therefore,  $M(\Sigma c (X))$ and $M(\Sigma \ell_\infty (X))$ are proper subsets of $\ell_{\infty} (\K)$. Moreover, $\Sigma c\pare{X} \subsetneq \Sigma \ell_{\infty}\pare{X}$. 
	\end{example}
	\begin{proof}
		Fix $x \in S_{X}$ and consider $s = \left\{\frac{(-1)^{n+1}x}{n}\right\}_{n} \in \Sigma c\pare{X}$. Taking $t = \{(-1)^{n+1}\}_{n} \in \ell_{\infty}\pare{\K}$, we obtain $t \cdot s = \left\{\frac{x}{n}\right\}_{n} \notin \Sigma \ell_{\infty}\pare{X}$. This shows that neither $\Sigma c\pare{X}$ nor $\Sigma \ell_{\infty}\pare{X}$ are  $b$-ideals. Now observe that $\{(-1)^n x\}_{n} \in \Sigma \ell_{\infty}\pare{X}$ but $\{(-1)^n x\}_{n} \notin \Sigma c\pare{X}$. Therefore, $\Sigma c\pare{X} \subsetneq \Sigma \ell_{\infty}\pare{X}$.
	\end{proof}

	 \subsection{The Space  $\ell_{u}\pare{X}$ } 
Let $X$ be a Banach space. A sequence $\{x_{n}\}_{n}\subseteq X$ is \emph{unconditionally summable} if there exists $x\in X$ such that for every bijection $h:\N\to \N$, the series 
$\sum_{k}x_{h(k)}$ converges to $x$. This will also be indicated by saying that the series $\sum_{n}x_{n}$ \emph{converges unconditionally (or converges unconditionally to $x$)}. 

Let us define the set
\begin{equation}
	\ell_{u}(X):=\left\{\{x_{n}\}_{n}\subseteq X:\text{ the series }\sum_{n}x_{n}\text{ converges unconditionally}\right\}.
\end{equation}
Note that $\ell_{u}(X)$ is a vector subspace of $S\pare{X}$ and that $c_{00}\pare{X}\subseteq \ell_{u}\pare{X}$.

\begin{example} Let $s \in S(X)$. The  Bounded Multiplier Test (\cite[p. 5]{DJT})  states that 
	$s \in \ell_{u}(X)$ if and only if $t \cdot s$ is summable for every $t \in \ell_{\infty}(\K)$.
	This implies that $\ell_{u}(X)$ is  a $b$-ideal. Since  $\ell_{u}\pare{X} \subseteq \Sigma c\pare{X}$, it follows that $\ell_{u}\pare{X} \subseteq \textup{Ip}\hspace{-2pt}\pare{\Sigma c\pare{X}}$. On the other hand,    the Bounded Multiplier Test,   also implies that $\textup{Ip}\hspace{-2pt}\pare{\Sigma c\pare{X}} \subseteq \ell_{u}\pare{X}$. Hence, the Bounded Multiplier Test  is equivalent to the equality
	\begin{equation}\label{lola165} \ell_{u}\pare{X} = \textup{Ip}\hspace{-2pt}\pare{\Sigma c\pare{X}}. \end{equation}
	Thus we can  equip $\ell_{u}\pare{X}$ with the operator norm corresponding to $\textup{Ip}\hspace{-2pt}\pare{\Sigma c\pare{X}}$, and so we define 
\end{example}

\begin{equation}\label{lola109} \norm{s}_{\ell_{{u}\pare{X}}} := \sup\llave{\norm{\sum_{n=1}^{m} t\pare{n} s\pare{n}}  :\ t \in \B_{\ell_{\infty}\pare{\K}},\ m \in \N}, \qquad \forall s \in \ell_{u}\pare{X}. \end{equation} 
In this way, $\ell_{u}\pare{X}$ is a $\BK$-space. Moreover, if $X \neq \llave{0}$, then $c_{00} (X) \subseteq  \ell_u (X)$. \vs

	 It is well-known that   $\ell_{u}\pare{\K}=\ell_{1}\pare{\K}$ as vector spaces. We now show that they have equal norms.
\begin{example}\label{d316_2}
	$\ell_{u}\pare{\K}\equiv \ell_{1}\pare{\K}$. Therefore, the sequence $\llave{e_{n}}_{n}$ is a basis of $\ell_{u}\pare{\K}$.
\end{example}
\begin{proof}
	Let $s\in \ell_{u}\pare{\K}$. Then 
	\begin{equation}
		\norm{s}_{\ell_{u}\pare{\K}} = \sup\llave{\norm{t\cdot s}_{\Sigma c\pare{\K}}: t \in \B_{\ell_{\infty}\pare{\K}}} 
		\leq \sup\llave{\sum_{n=1}^{\infty}\abs{t\pare{n}s\pare{n}}: t \in \B_{\ell_{\infty}\pare{\K}}} = \norm{s}_{\ell_{1}\pare{\K}}.
	\end{equation}
	
%	To verify the reverse inequality, first assume that $\norm{s}_{\ell_{1}\pare{\K}} \leq 1$. 
	
	Define the sequence $z: \N \tien X$ by $z\pare{n} = \frac{\overline{s\pare{n}}}{\abs{s\pare{n}}}$ if $s\pare{n} \neq 0$, and $z\pare{n} = 0$ if $s\pare{n} = 0$. Then $z \in \B_{\ell_{\infty}\pare{\K}}$, and therefore 
	\begin{equation}
		\sum_{n=1}^{m}\abs{s\pare{n}} = \sum_{n=1}^{m}z\pare{n}s\pare{n} \leq \norm{s}_{\ell_{u}\pare{\K}}, \qquad \forall m \in \N.
	\end{equation}
	Thus, $\norm{s}_{\ell_{1}\pare{\K}} \leq \norm{s}_{\ell_{u}\pare{\K}}$. 
	
	Finally, it is well known that $\llave{e_{n}}$ is a basis of $\ell_{1}\pare{\K} = \ell_{u}\pare{\K}$.
\end{proof}

%For the case $\norm{s}_{\ell_{1}\pare{\K}} > 1$, since $\frac{s}{\norm{s}_{\ell_{1}\pare{\K}}} \in \B_{\ell_{\infty}\pare{\K}}$, by the initial case it follows that 
%$$\norm{\frac{s}{\norm{s}_{\ell_{1}\pare{\K}}}}_{\ell_{1}\pare{\K}} \leq \norm{\frac{s}{\norm{s}_{\ell_{1}\pare{\K}}}}_{\ell_{u}\pare{\K}}.$$ 
%Hence, $\norm{s}_{\ell_{1}\pare{\K}} \leq \norm{s}_{\ell_{u}\pare{\K}}$.

		\section{ The Space $M_{\Sigma}\pare{V,X}$}\label{sec10}\label{pop10}
Let $X$ be a Banach space, and fix a $\BKa$-space $V \subseteq S(\K)$. We are interested in describing $\mathcal{L}(V, X)$ as a multiplier space. This leads us to consider the multiplier space $M' = M'(V, \Sigma c(X))$ determined by the pointwise bilinear operator $\BB': M' \times V \to \Sigma c(X)$, and to replace $\BB'$ with another  bilinear operator $B_{V,X}: M' \times V \to X$. In this way the values of the new representation operator lie in $\mathcal{L}(V, X)$. This change in the bilinear operator translates into a change of norm for $M'$, resulting in a new multiplier space which  will denoted by $M_{\Sigma}(V, X)$. Lemma~\ref{naty32} provides conditions under which $M_{\Sigma}(V, X)$ is a Banach space, in which case the norms on $M_{\Sigma}(V, X)$ and $M'$ are equivalent. In Theorem \ref{d272_2}, we show that for the representation operator $R_{V,X}: M_{\Sigma}(V, X) \to \mathcal{L}(V, X)$ to be an isomorphism, it is necessary and sufficient that $\{e_{n}\}_{n}$ be a basis of $V$. We remark that part of this sufficiency is well-known \cite[Thm. I.17.1, p. 211]{S}.

As a conclusion we see that
$M_\Sigma (V, \K)$  extends the notion of associate space, as defined in the context of Banach function spaces. (For example, see  \cite{CDS}  and consider in $S(\K)=F(\N,\K)$ the counting measure.)   This extension   applies to $\BK$-spaces $V \subseteq S(\K)$ satisfying $\{e_n\}_n \subseteq V$.

 Let start by considering   the sum operator $S: \Sigma c(X) \to X$ introduced in (\ref{d247}). Next, we define  the vector space
 \begin{equation}\label{lola208}
 	M_{\Sigma}(V,X) := \Mt(V, \Sigma c(X)) \subseteq S(X)
 \end{equation}
 and the function $B_{V,X}: M_{\Sigma}(V,X) \times V \to X$ by
 \begin{equation}\label{lola126}
 	B_{V,X}(f, v) := S(v \cdot f), \qquad \forall f \in M_{\Sigma}(V,X), \qquad \forall v \in V.
 \end{equation}
 
 Since $B_{V,X}$ is a bilinear operator, then  $M_{\Sigma}(V,X)$ is a multiplier space, with its operator norm given by
 \begin{equation}
 	\norm{f}_{M_{\Sigma}(V,X)} = \sup\left\{\norm{\sum_{n=1}^{\infty} v(n) f(n)} : v \in \B_{V}\right\}.
 \end{equation}
 Using the representation operator $\widetilde{R}: \Mt(V, \Sigma c(X)) \to \mathcal{L}(V, \Sigma c(X))$, we obtain that
 \begin{align}\label{d297}
 	\norm{f}_{M_{\Sigma}(V,X)} &= \sup\left\{\norm{\sum_{n=1}^{\infty} v(n) f(n)} : v \in \B_{V}\right\} \leq \sup\left\{\norm{f \cdot v}_{\Sigma c(X)} : v \in \B_{V}\right\} 
 	\nonumber \\[.3 cm]
 	&= \norm{\widetilde{R}_{f}}_{\mathcal{L}(V, \Sigma c(X))} = \norm{f}_{\Mt(V, \Sigma c(X))}, \qquad \forall f \in M_{\Sigma}(V,X).
 \end{align}
 It follows that all  $B_{V,X}$ multiplication operators are bounded.
 
 Thus, we obtain the following result.
 
 \begin{corollary}
 	Let $X$ be a Banach space, $V \subseteq S(\K)$ a vector space, and $f \in S(X)$ such that $ \sum_{n} v(n) f(n)$ converges for each $v \in V$. If $V$ is a BKN-space, then
 	\begin{equation}
 		\sup\llave{\norm{\sum_{n=1}^{\infty} v(n) f(n)}:v\in \B_{V}}<\infty.
 	\end{equation}
 \end{corollary}
 
 Note  the corresponding representation operator $R_{V,X}: M_{\Sigma}(V,X) \to \mathcal{L}(V,X)$, is given by
 \begin{equation}\label{d311}
 	R_{V,X}f(v) = \sum_{n=1}^{\infty} v(n) f(n), \qquad \forall f \in M_{\Sigma}(V,X), \qquad \forall v \in V.
 \end{equation}

	From the previous equality  we can note that, unlike what happens in $\Mt(V, \Sigma c(X))$, the multipliers in $M_{\Sigma}(V,X)$ do not act pointwise.\vs

	\begin{lemma}\label{lola72}
		Let $V \subseteq S(\K)$ be a $\BK$-space such that $\{e_{n}\}_{n} \subseteq V$. Then $M_{\Sigma}(V,X)$ is normed, has continuous evaluations, and $c_{00}(X) \subseteq M_{\Sigma}(V,X)$. 
	\end{lemma}
	\begin{proof}
		By Lemma \ref{piz7_0} and (\ref{d297}), we have that $\norm{\cdot}_{M_{\Sigma}(V,X)}$ is a seminorm on $M_{\Sigma}(V,X)$. Let $f \in M_{\Sigma}(V,X)$ be such that $\norm{f}_{M_{\Sigma}(V,X)} = 0$. For each $n \in \N$, we have that $e_{n} \in V$, and therefore $f(n) = S(e_{n} \cdot f) = 0$. Thus, $f = 0$. It follows that $\norm{\cdot}_{M_{\Sigma}(V,X)}$ is a norm.
		
		Let $n \in \N$. For each $f \in M_{\Sigma}(V,X)$, we have that 
		\begin{equation}
			\norm{f(n)} = \norm{\sum_{m=1}^{\infty} e_{n}(m) f(m)}_{X} = \norm{B_{V,X}(f, e_{n})}_{X} \leq \norm{e_{n}}_{V} \norm{f}_{M_{\Sigma}(V,X)}.
		\end{equation}
		Thus $M_{\Sigma}(V,X)$ has continuous evaluations. Clearly, $e_{n}x \in M_{\Sigma}(V,X)$, $\forall x \in X, \forall n\in \N$. 
	\end{proof}
	
	\begin{lemma}\label{naty32}
		Let $V \subseteq S(\K)$ be a $\BK$-space such that $\{e_{n}\}_{n} \subseteq V$. The following statements are equivalent: \qs
		\begin{enumerate}
			\item[i.] $M_{\Sigma}(V,X)$ is a Banach space. 
			\item[ii.] There exists $c \in (0, \infty)$ such that for each $f \in M_{\Sigma}(V,X)$ and $v \in \B_{V}$, the following holds:
			\begin{equation}\label{lola191}
				\norm{\sum_{n=1}^{k} v(n) f(n)} \leq c \norm{f}_{M_{\Sigma}(V,X)}, \qquad \forall k \in \N.
			\end{equation} 
		\end{enumerate}
		This is equivalent to  
		\begin{equation}\label{lola196}
			\frac{1}{c} \norm{f}_{\Mt(V, \Sigma c(X))} \leq \norm{f}_{M_{\Sigma}(V,X)} \leq \norm{f}_{\Mt(V, \Sigma c(X))}, \qquad \forall f \in M_{\Sigma}(V,X).
		\end{equation} 
	\end{lemma}
	\begin{proof} 
		$i \Longrightarrow ii.$ Since $M_{\Sigma}(V,X) = \Mt(V, \Sigma c(X))$ and both are Banach spaces, it follows from (\ref{d297}) and the inverse function theorem that there exists a constant $c \in (0, \infty)$ satisfying (\ref{lola191}).
		
		$ii \Longrightarrow i.$ Using (\ref{d297}), it follows that the norms of the spaces $M_{\Sigma}(V,X)$ and $\Mt(V, \Sigma c(X))$ are equivalent. Since $\Mt(V, \Sigma c(X))$ is complete,  we have that $M_{\Sigma}(V,X)$ is complete.
		
		Finally, observe that from (\ref{d297}) and (\ref{lola191}), we directly obtain (\ref{lola196}). 	
	\end{proof}

\begin{theorem}\label{lola52_0} 
	Let $V \subseteq S(\K)$ be a $\BK$-space such that $\{e_{n}\}_{n} \subseteq V$. If $\nu_V < \infty$, then $M_{\Sigma}(V,X)$ is a $\BK$-space.
	Furthermore,  
	\begin{equation}\label{lola193_0} 
		\frac{1}{\nu_V} \norm{f}_{\Mt(V, \Sigma c(X))} \leq \norm{f}_{M_{\Sigma}(V,X)} \leq \norm{f}_{\Mt(V, \Sigma c(X))}, \quad \forall f \in M_{\Sigma}(V,X).
	\end{equation} 
	Therefore, if $\nu_V = 1$, then $M_{\Sigma}(V,X) \equiv \Mt(V, \Sigma c(X))$.
\end{theorem}  
 \begin{proof}
 	We will show that $M_{\Sigma}(V,\K)$ satisfies (\ref{lola191}) with $c = \nu_V$, which will conclude the proof. Fix $f \in M_{\Sigma}(V,X)$ and $v \in \B_V$. Then, $w_k := \frac{\pi_k v}{\nu_V} \in \B_V$, for all $k \in \N$. Thus, 
 	\begin{equation}
 		 \norm{\sum_{n=1}^{k} v(n) f(n)} = \nu_V \norm{\sum_{n=1}^{\infty} w_k(n) f(n)} \leq \nu_V \norm{f}_{M_{\Sigma}(V,X)}, \quad \forall k \in \N.
 	\end{equation}
 	This proves (\ref{lola191}). Finally, (\ref{lola193_0}) follows directly from (\ref{lola196}).
 \end{proof}

The following result is obtained directly from Theorem~\ref{lola52_0} and corresponds to \cite[Thm.~10.5.1]{Wil}, although it is stated and proved differently there.

\begin{proposition}\label{naty27}
	Let $V \subseteq S(\K)$ be a $\BK$-space with $\{e_n\}_{n} \subseteq V$ and $X \neq \{0\}$ a Banach space. Then $\nu_{M_{\Sigma}(V,X)} \leq \nu_V$.  
\end{proposition} 
\begin{proof}
	Fix $X \neq \{0\}$. Let $f \in M_{\Sigma}(V,X)$ and $k \in \N$. From (\ref{lola193_0}), we have that
	\begin{equation}
		\norm{P_k f}_{M_{\Sigma}(V,X)} = \sup \left\{ \norm{\sum_{n=1}^{k} v(n) f(n)} : v \in \B_V \right\} \leq \norm{f}_{M'(V,\Sigma c(X))} \leq \nu_V \norm{f}_{M_{\Sigma}(V,X)}.
	\end{equation}
	From this, we conclude that $\nu_{M_{\Sigma}(V,X)} \leq \nu_V$.
\end{proof}

	The following result is obtained by observing that $\nu_{c(\K)} = \nu_{\Sigma \ell_{\infty}(\K)} = 1$.
	\begin{example} Let $X$ be a Banach space. Then:
		\begin{enumerate}
			\item[i.] The multiplier space $M_{\Sigma}(c(\K), X)$ is a Banach space, with its operator norm   given by 
			\begin{equation}
				\norm{f}_{M_{\Sigma}(c(\K), X)} = \sup \left\{ \norm{\sum_{n=1}^{\infty} v(n) f(n)} : v \in \B_{c(\K)} \right\}.
			\end{equation}
			\item[ii.] The multiplier space $M_{\Sigma}(\Sigma \ell_{\infty}(\K), X)$ is a Banach space. Note that in this case $\Sigma \ell_{\infty}(\K)$ does not have $\{e_{n}\}_{n}$ as a basis.
		\end{enumerate}
	\end{example}

The implication $i \Longrightarrow ii$ in the following result can be proven using \cite[Thm. I.17.1, p. 211]{S}, where $\Mt(V, \Sigma c(X))$ is implicitly used instead of $M_{\Sigma}(V, X)$. Next, we present an alternative proof using our development.

\begin{theorem}\label{d272_2} \label{d250_4} \label{d236_4}\label{efren14_z} \label{efren15_0_z}
	Let $V \subseteq S(\K)$ be a $\BK$-space such that $\{e_{n}\}_{n} \subseteq V$. The following are equivalent:
	\begin{enumerate}
		\item[i.] $\{e_{n}\}_{n}$ is a basis of $V$.
		\item[ii.] The representation operator $R_{V,X}: M_{\Sigma}(V, X) \to \mathcal{L}(V, X)$ is an isometric isomorphism for every $X \neq \{0\}$.
		\item[iii.] The representation operator $R_{V,X}: M_{\Sigma}(V, X) \to \mathcal{L}(V, X)$ is an isometric isomorphism for some $X \neq \{0\}$.
	\end{enumerate}  
\end{theorem} 
\begin{proof} Let $R := R_{V,X}$.
	$i \Longrightarrow ii$. Let $X \neq \{0\}$. By Proposition \ref{d286_0}, $V$ has continuous evaluations. Using Theorem \ref{lola52_0} and Proposition \ref{piz4_2}, we have that $R$ is a linear isometry. Thus, it only remains to verify that $R$ is surjective. Let $T \in \mathcal{L}(V, X)$. Since $\{e_{n}\}_{n}$ is a basis of $V$, we have $v = \sum_{n=1}^{\infty} v(n) e_{n}$ for all $v \in V$. Then, 
	\begin{equation}
			T(v) = \sum_{n=1}^{\infty} v(n) T(e_{n}), \quad \forall v \in V.
	\end{equation}
	Let $f := \{T(e_{n})\}_{n} \in S(X)$. It follows that $f \in \Mt(V, \Sigma c(X))$ and 
	\begin{equation}
		Rf(v) = \sum_{n=1}^{\infty} f(n) v(n) = \sum_{n=1}^{\infty} v(n) T(e_{n}) = T(v).
	\end{equation}
	Hence, $R$ is surjective.
	
	 Clearly, $ii \Longrightarrow iii$ holds.

	$iii \Longrightarrow i$. Now, suppose there exists a Banach space $X \neq \{0\}$ such that $R: M_{\Sigma}(V, X) \to \mathcal{L}(V, X)$ is an isometric isomorphism. 
	
Proceeding by contradiction, we first prove that $V = V_{0}$. Assume that $c_{00}(\K)$ is not dense in $V$. Let $v \in V \setminus \overline{c_{00}(\K)}$ and $x \in X \setminus \{0\}$. Using the Hahn-Banach theorem, choose $v^{*} \in V^{*}$ such that $\langle v, v^{*} \rangle = 1$ and $\langle z, v^{*} \rangle = 0$ for all $z \in \overline{c_{00}(\K)}$. Next, consider the operator $x \oplus v^{*} \in \mathcal{L}(V, X)$, defined by  $x \oplus v^{*}(v) = x$ and $x \oplus v^{*}(z) = 0$ for all $z \in \overline{c_{00}(\K)}$. Since $R$ is an isomorphism, there exists $f \in M_{\Sigma}(V, X)$ such that $Rf = x \oplus v^{*}$. For each $n \in \N$, we have $e_{n} \in V$, and thus 
	$
	f(n) = Rf(e_{n}) = x \oplus v^{*}(e_{n}) = 0.
	$
	Therefore, $f = 0$, which contradicts $Rf(v) = x \neq 0$. Hence, $c_{00}(\K)$ is dense in $V$.
	
	Next, we will show that $\nu_{V} < \infty$, that is, we need to prove that the sequence of canonical projections $\{\pi_{n}\}_{n} \subseteq \mathcal{L}(V, V)$ is bounded. Applying the uniform boundedness principle, it suffices to show that for each $v \in V$, the sequence $\{\pi_{n}v\}_{n}$ is bounded. Fix $v \in V$. By the equivalence between boundedness and weak boundedness (see \cite[Thm. 3.18]{R}), we only need to verify that for each $v^{*} \in V^{*}$, the sequence $\{\langle \pi_{n}v, v^{*} \rangle\}_{n}$ is bounded. Fix $v^{*} \in V^{*}$. Since $X \neq \{0\}$, let us take $x \in X$ with $\norm{x}=1$. Observe that $\{\langle \pi_{n}v, v^{*} \rangle\}_{n}$ is bounded if and only if $\{\langle \pi_{n}v, v^{*}\rangle x\}_{n} = \{x \oplus v^{*} (\pi_{n}v)\}_{n}$ is bounded. Since $x \oplus v^{*} \in \mathcal{L}(V, X)$, by hypothesis there exists $c \in M_{\Sigma}(V, X)$ such that $Rc = x \oplus v^{*}$. Nothing that $c \in M_{\Sigma}(V, X) = M'(V, \Sigma c(X))$, it follows that 
	\begin{equation}
		\norm{x \oplus v^{*}(\pi_{n}v)} = \norm{R_{c}(\pi_{n}v)} = \norm{\sum_{k=1}^{n} c(k) v(k)} \leq \norm{c \cdot v}_{\Sigma c(X)} < \infty, \quad \forall n \in \N.
	\end{equation}
	Thus, $\nu_{V} < \infty$. Finally, from Corollary \ref{lola209_0}, we obtain that $\{e_{n}\}_{n}$ is a basis of $V_{0} = V$.   
\end{proof}

\begin{example}
	Since $\{e_{n}\}_{n}$ is not a basis of $c(\K)$, Theorem \ref{d250_4} allows us to conclude that the representation operator $R_{c(\K), X}: M_{\Sigma}(c(\K), X) \to \mathcal{L}(c(\K), X)$ is not surjective. 
\end{example}

\begin{corollary}\label{naty26}
	Let $V \subseteq S(\K)$ be a Banach space. If $\{e_{n}\}_{n}$ is a basis of $V$, then
	\begin{equation}
		M_{\Sigma}(V, X) = \left\{ \{Te_{n}\}_{n} : T \in \mathcal{L}(V, X) \right\}.
	\end{equation}
\end{corollary}
\begin{proof} Let $M := \left\{ \{Te_{n}\}_{n} : T \in \mathcal{L}(V, X) \right\}$.
	
	$\subseteq)$ Let $f \in M_{\Sigma}(V, X)$. Then, $(R_{V,X})_{f} \in \mathcal{L}(V, X)$ and
\begin{equation}
	(R_{V,X})_{f}(e_{n}) = \sum_{k=1}^{\infty} f(k) e_{n}(k) = f(n), \qquad \forall n \in \N.
\end{equation}
	Thus, $f = \{(R_{V,X})_{f}(e_{n})\}_{n} \in M$.
	
	$\supseteq)$ Let $T \in \mathcal{L}\pare{V,X}$. By Theorem \ref{d250_4}, there exists $f \in M_{\Sigma}(V, X)$ such that $(R_{V,X})_{f} = T$. Then, 
\begin{equation}
		Te_{n} = (R_{V,X})_{f}(e_{n}) = f(n), \qquad \forall n \in \N.
\end{equation}
	Hence, $\{Te_{n}\}_{n} = f \in M_{\Sigma}(V, X)$.
\end{proof}

\begin{example}\label{lola212} 
	Let $V \subseteq S(X)$ be a $\BKa$-space. Since $\Sigma \ell_{\infty}(X)$ is a $\BK$-space, it follows that 
	\begin{equation}
		\text{$M(V,\Sigma \ell_{\infty}(X))$ is a $\BK$-space.}
	\end{equation}
	Moreover, $c_{00}(\K) \subseteq M(V, \Sigma \ell_{\infty}(X))$. Note that a sequence $\{a_n\}_n$ belongs to $M(V,\Sigma \ell_{\infty}(X))$ if and only if the partial sums of $\sum_n a_n v_n$ are bounded in $V$, for each $\{v_n\}_n \in V$. Additionally, $M(V,\Sigma \ell_{\infty}(X))$ is a b-ideal.    Applying Corollary \ref{lola199_0}, we conclude that:
	\begin{equation}
		\{e_n\}_n \text{ is a basic sequence in } M(V,\Sigma \ell_{\infty}(X)).
	\end{equation}
\end{example}

\subsection{The Space $M_{\Sigma}(V)$ as a Banach Algebra}
Let $V \subseteq S(\K)$ be a $\BK$-space such that $c_{00}(\K) \subseteq V$ and $\nu_{V} < \infty$. Then, the space 
\begin{equation}
	M_{\Sigma}(V) := M_{\Sigma}(V, V) \subseteq S(V)
\end{equation}
is a Banach space with its corresponding representation operator $R := R_{V, V}: M_{\Sigma}(V) \to \mathcal{L}(V)$, given by
\begin{equation}\label{naty37}
	Rg(v) = \sum_{n=1}^{\infty} v(n) g(n), \qquad \forall g \in M_{\Sigma}(V), \qquad \forall v \in V.
\end{equation}

Fix $f, g \in M_{\Sigma}(V)$ and $v \in V$. Since $Rg$ is continuous, from (\ref{naty37}) we have that
\begin{equation}\label{naty35}
	Rf \circ Rg(v) = Rf(Rg(v)) = \sum_{n=1}^{\infty} v(n) Rf(g(n)) = \sum_{n=1}^{\infty} v(n) \left( \sum_{k=1}^{\infty} (g(n)(k)) f(k) \right) \in V.
\end{equation}
Define $h \in S(V)$ by
\begin{equation}
	h(n) := Rf(Rg(e_{n})) = \sum_{k=1}^{\infty} (g(n)(k)) f(k) \in V.
\end{equation}
From (\ref{naty35}), it follows that $h \in M_{\Sigma}(V)$ and that $Rf \circ Rg = Rh$. Using Proposition \ref{naty39}, we conclude that $M_{\Sigma}(V)$ equipped with the product $\#$ given by
\begin{equation}
	f \# g (n) := \sum_{k=1}^{\infty} (g(n)(k)) f(k) \in V, \qquad \forall n \in \N,
\end{equation}
is a Banach algebra.

\subsection{The Associate Space $V^{\times}$}\label{subsec}
In \cite[p. 423]{K}, the $\beta$ and $\gamma$ Köthe-Toeplitz duals of a $\BK$-space $V\subseteq S\pare{\K}$   with $\llave{e_{n}}_{n}\subseteq V$ are studied. These dual spaces are defined respectively (as normed spaces) by
\begin{equation}  
	V^{\beta}=M\pare{V,\Sigma c\pare{\K}},\qquad V^{\gamma}=M\pare{V,\Sigma \ell_{\infty}\pare{\K}}. 
\end{equation}

\begin{example}
	Let $V\subseteq S(\K)$ be a $\BK$-space such that $\{e_{n}\}_{n}\subseteq V$ and $V=V_{0}$. From Proposition \ref{chuy18_0_1_0}, it follows that 
	\begin{equation}
		V^{\gamma}=M\pare{V_{0}, \Sigma \ell_{\infty}\pare{\K}_{0}}= V^{\beta},
	\end{equation} 
	as established in \cite[Thm. 7.2.7]{Wil}.
\end{example}

%We recall below the definition of a Banach ideal, as it will be needed to recall   the well-known concept of  associate space.   	Let $D$ be a non-empty set. For each function $f \in F(D, \K)$ we define
%	\[
%	|f|(a) := |f(a)|, \qquad \overline{f}(a) := \overline{f(a)}, \quad a \in D.
%	\]
%	Here, $|f(a)|$ and $\overline{f(a)}$ denote the modulus and complex conjugate of $f(a)$, respectively.

 	The following corollary is obtained directly by using  example \ref{again}  together with Theorem~ \ref{lola52_0} .

	\begin{corollary}\label{efren8}
	If $Y \subseteq S(\K)$ is a Banach ideal that does not vanish at any point, then
	$M_{\Sigma}(Y,X) \equiv \Mt(Y,\Sigma c(X))$. Hence, $M_{\Sigma}(Y,X)$ is a Banach space.
\end{corollary} 

Let $Y \subseteq S(\K)$ be a Banach ideal that does not vanish at some point, which is equivalent to  $\llave{e_n:n\in \N} \subseteq Y$.  
The (normed) space of multipliers $M(Y, \ell_{1}(\K))$ is called the \emph{associate  space of $Y$} (or the \emph{Köthe dual space} of $Y$), and is denoted by $Y^{\times}$ (see \cite[§12]{B}, for example).  
Its norm is the operator norm and is given by
\begin{equation}
	\norm{s}_{Y^{\times}} = \sup\llave{ \sum_{n=1}^{\infty} \abs{s(n)y(n)} : y \in \B_{Y} }, 
	\qquad \forall s \in Y^{\times}.
\end{equation}
In this case, the $M$-inequality takes the following form:
\begin{equation}
	\sum_{n=1}^{\infty} \abs{s(n)y(n)}
	= \norm{s \cdot y}_{\ell_{1}(\K)}
	\leq \norm{s}_{Y^{\times}} \norm{y}_{Y},
	\qquad \forall s \in Y^{\times}, \quad \forall y \in Y.
\end{equation}

Applying Example~\ref{d316_2} and Theorem~\ref{lola166}, we obtain the following.
 \begin{example}\label{lola66}\label{d282_3} 
	Let $Y \subseteq S(\K)$ be a Banach ideal that does not vanish at some point. Then
	\begin{equation}\label{lola97}
		Y^{\times} \equiv M(Y,\ell_{1}(\K)) \equiv M(Y,\ell_{u}(\K)) \equiv M(Y,\Sigma c(\K)).
	\end{equation}  
\end{example}

Now, from (\ref{lola97}) and Corollary \ref{efren8} , we obtain
\begin{equation}\label{naty31}
	Y^{\times}\equiv M\pare{Y,\Sigma c\pare{\K}}\equiv M_{\Sigma}\pare{Y,\K}. 
\end{equation}
This equality and Lemma \ref{lola72} allow us to extend the concept of the associate space to spaces $V\subseteq S\pare{\K}$ that are $\BK$ and satisfy $\llave{e_{n}}_{n}\subseteq V$, but are not necessarily Banach ideals. In this case, we define the \emph{associate space of $V$} as the space  
\begin{equation}\label{lola158}
	V^{\times}:=M_{\Sigma}\pare{V,\K}, 
\end{equation}
equipped with its operator norm
\begin{equation}\label{naty21}
	\norm{s}_{V^{\times}}=\sup\llave{\abs{\Sum_{n=1}^{\infty}v\pare{n}s\pare{n}}:v\in \B_{V}},\qquad \forall s\in V^{\times}.
	\end{equation}
	In this case, the representation operator $R_{V,\K}:V^{\times}\tien V^{*}$ satisfies
	\begin{equation}
\la v,  R_{V,\K}s\ra=\Sum_{n=1}^{\infty}v\pare{n}s\pare{n},\qquad \forall s\in V^{\times},\qquad  \forall v\in V.
\end{equation}

The following result is obtained directly from Theorem~\ref{d272_2} . In \cite[Thm. 10.5.1]{Wil}, this same result was proved by defining another norm for $V^\beta$. In our case, however, we describe this norm naturally as the induced operator norm of the bilinear operator $B_{V,\K}$.

\begin{corollary}\label{lola201_z}
	Let $V \subseteq S(\K)$ be a $\BK$-space such that $\{e_{n}\}_{n} \subseteq V$. Then $\{e_{n}\}_{n} \subseteq V$ is a basis of $V$ if and only if the representation operator $R_{V,\K} : V^{\times} \to V^{*}$ given in \eqref{d311} is an isometric isomorphism.
\end{corollary}

As mentioned at the beginning of this section, the multiplier spaces $V^{\times}$ and $V^{\beta}$ are always equal as vector spaces, except that the involved bilinear operators are different. Therefore, the use of $V^{\times}$ can be directly addressed through $V^{\beta}$ and the representation operator $R_{V,\K}$. 

The above approach is the one employed in \cite[Ch. 4-10]{Wil}. It follows from (\ref{lola196}) that the inclusion of $V^\beta \hookrightarrow V^\times$ is always continuous. Therefore, when $V^{\times}$ is complete, the corresponding norms are equivalent, and no significant difficulties arise. On the other hand, in \cite[Ex. 10.1.3]{Wil}, although without using our terminology, an example is presented of a $\BKn$-space $V\subseteq S\pare{\K}$  such that the range of the representation operator $R_{V,\K}:V^{\times}\tien V^{*}$ is not closed. Thus, in this case, the associate space $V^{\times}$ is not a Banach space, although the space $V^{\beta}$ is.

			\section{The Weak Vectorialization  $U_{w}\pare{D,X}$}\label{Holamundo2}\label{pop11}
	
	In this section we show that a family of well-known Banach spaces, which will be introduced in Example \ref{holamundo3b}, can be constructed as $B$-multiplier spaces through an appropriate choice of a bilinear operator $B$.\vs 
	
	Let $D$ be a non-empty set and $X$ a Banach space. The \emph{scalarization} of $f \in F\pare{D,X}$ with respect to a functional $x^{*} \in X^{*}$ is the function $\langle f, x^{*} \rangle : D \to \K$ defined by
	\begin{equation}
		\langle f, x^{*} \rangle \pare{a} := \langle f(a), x^{*} \rangle,
	\end{equation}
	where $\langle \cdot, \cdot \rangle : X \times X^{*} \to \K$ is the dual pairing. Defining now $B_{*} : F\pare{D,X} \times X^{*} \to F\pare{D,\K}$   by
	\begin{equation}
		B_{*}\pare{f, x^{*}} := \langle f, x^{*} \rangle,
	\end{equation}
	we obtain a bilinear operator.\vs
	
	Let $U \subseteq F\pare{D,\K}$ be a vector space and consider the $B_{*}$-multiplier space
	\begin{equation}\label{chuy19}
		U_{w}\pare{D,X} := M_{B_{*}}\pare{X^{*}, U} = \left\{ f \in F\pare{D,X} : \langle f, x^{*} \rangle \in U, \forall x^{*} \in X^{*} \right\}.
	\end{equation}
Define  $\norm{\cdot}_{w}:U_{w}\pare{D,X}\to [0,\infty]$ by
\begin{equation}\label{lola202_x}
	\norm{f}_{w}:=\norm{f}_{M_{B_{*}}\pare{X^{*},U}}=\sup\left\{\norm{\langle f,x^{*}\rangle}_{U}:x^{*}\in \B_{X^{*}}\right\}.
\end{equation}

		\begin{theorem}\label{d190_0} 
			If $U \subseteq F\pare{D,\K}$ is a $\BK$-space, then $U_{w}\pare{D,X}$ with   the norm given by~\eqref{lola202_x} (its operator norm)  is a $\BK$-space. Moreover:
			\begin{enumerate}
				\item[i.] If $U$ is an order ideal, then $U_{w}\pare{D,X}$ is a  $b$-ideal.
				\item[ii.] If $U$ is a Banach ideal, then $U_{w}\pare{D,X}$ is a $\BK$ $b$-ideal.
				\item[iii.] If $U$ does not vanish at some point, then $U_{w}\pare{D,X}$ does not vanish at some point.
				\item[iv.] If $\llave{e_{a}}_{a\in D}\subseteq U$, then $c_{00}\pare{X}\subseteq U_{w}\pare{D,X}$.
				\item[v.] If $U \subseteq B\pare{D,\K}$, then
				$U_{w}\pare{D,X} \subseteq B\pare{D,X}.$
			\end{enumerate}   
		\end{theorem}
		
		\begin{proof}
			We will first  verify  \ref{lola76_0} and \ref{lola86_0}.
			
			I. Let $f \in U_{w}\pare{D,X}$ be such  that $B_{*}\pare{f,x^{*}} = 0$ for all $x^{*} \in X^{*}$. Then for each $a \in D$, we have $\langle f(a), x^{*} \rangle = 0$ for all $x^{*} \in X^{*}$, and consequently $f(a) = 0$. Thus, $f = 0$.
			
			II. Let $\{x_{n}^{*}\}_{n} \subseteq X^{*}$ and $x^{*} \in X^{*}$ be such that $x_{n}^{*} \to x^{*}$, and let $f \in U_{w}\pare{D,X}$. Then 
			\begin{equation}
				B_{*}\pare{f,x_{n}^{*}}(a) = \langle f(a), x_{n}^{*} \rangle \to \langle f(a), x^{*} \rangle = B_{*}\pare{f,x^{*}}(a), \quad \forall a \in D.
			\end{equation}
			In view of the above, Corollary \ref{piz11_b_0} implies that $U_{w}\pare{D,X}$ is a normed space.
			
			Now observe that $\norm{f(a)}_{X} = \sup \left\{ \abs{\langle f,x^{*} \rangle(a)} : x^{*} \in \B_{X^{*}} \right\},  \forall a \in D.$
			Thus $B_{*}$ satisfies condition (\ref{piz17}), and so by Lemma \ref{d339_0} it follows that $U_{w}\pare{D,X}$ has continuous evaluations.
			
			To conclude that $U_{w}\pare{D,X}$ is a Banach space, according to Corollary \ref{d189_4} we only need to verify that $B_{*}$ satisfies condition (\ref{lola9}). Let $\{f_{k}\}_{n} \subseteq U_{w}\pare{D,X}$  and $f \in F\pare{D,X}$ be such that $f_{n} \to f$ pointwise. Then for each $x^{*} \in X^{*}$ and $a \in D$, we have
			$ B_{*}\pare{f_{n},x^{*}}(a) = \langle f_{n}(a), x^{*} \rangle \to \langle f(a), x^{*} \rangle = B_{*}\pare{f,x^{*}}(a). $
			Therefore, it follows that $U_{w}\pare{D,X}$ is a Banach space.
			
			i. Assume now that $U$ is an order ideal. Let $g\in B\pare{D,\K}$, $f\in U_{w}\pare{D,X}$, and fix $x^{*}\in X^{*}$. Since
			$\abs{\langle g\cdot f,x^{*}\rangle}
			=
			\abs{g}\cdot \abs{\langle f,x^{*}\rangle}$,
			and $\langle f,x^{*}\rangle\in U$, it follows from the ideal property of $U$ that $\langle g\cdot f,x^{*}\rangle\in U$. Hence, $g\cdot f\in U_{w}\pare{D,X}$. 
			
			ii. Assuming in addition that $U$ is a Banach $b$-ideal, it follows that \begin{equation}
				\norm{\langle g \cdot f, x^{*} \rangle}_{U} = \norm{\abs{g} \cdot \abs{\langle f, x^{*} \rangle}}_{U}\leq \|g\|_{B\pare{D,\K}} \norm{ \langle f, x^{*} \rangle}_{U}\leq \|g\|_{B\pare{D,\K}} \norm{f}_{U_{w}(D,X)}
			\end{equation}  Taking the supremum over $x^{*}\in \B_{X^{*}}$ yields
			\begin{equation}
				\|g\cdot f\|_{U_{w}(D,X)}
				\leq
				\|g\|_{B\pare{D,\K}}
				\norm{f}_{U_{w}(D,X)}. 
			\end{equation} 
			
			Statements iii and iv follow directly.
			
			v.  Let $f\in U_{w}\pare{D,X}$ and $x^{*}\in X^{*}$. Then $\la f, x^{*}\ra=\llave{\la f\pare{a},x^{*}\ra}_{a\in D}\in U\subseteq B\pare{D,\K}$. Hence, the set $\llave{f\pare{a}}_{a\in D}$ is weakly bounded, which implies that $f\in B\pare{D,\K}$.
		\end{proof}

		From (ii) above, we obtain the following result.
		
	\begin{corollary}
		Let $X$ be a Banach space, $U \subseteq S(\K)$ a vector space, and $f \in S\pare{X}$. If $U$ is a BK-space and $\langle f, x^{*} \rangle \in U$ for every $x^{*} \in X^{*}$, then
		\begin{equation}
			\sup \llave{\norm{\langle f, x^{*} \rangle}_{U} : x^{*} \in \B_{X^{*}}}<\infty.
		\end{equation}
	\end{corollary}

		\subsection{The Weak Vectorialization  $Y_{w}\pare{X}$}

		Let $X$ be a Banach space and $U=Y\subseteq S\pare{\K}$  a Banach ideal. Then, instead of $Y_{w}\pare{\N,X}=M_{B_{*}}\pare{X^{*},Y}$ we will simply write $Y_{w}\pare{X}$. Since $Y$ has continuous evaluations, Theorem~\ref{d190_0} implies that the space $Y_{w}\pare{X}$ with its operator norm $\norm{\cdot}_{w}:Y_{w}\pare{X}\to [0,\infty)$, given by
		\begin{equation}
			\norm{s}_{w} := \sup\left\{\norm{\la s,x^{*}\ra}_{Y} : x^{*} \in \B_{X^{*}}\right\},
		\end{equation}
		is a $\BK$-space. Henceforth, we will call $Y_{w}\pare{X}$ the \emph{weak vectorialization of $Y$ (or weak $X$-vectorialization)} (see \cite[p. 32]{DJT}).  It follows directly from Theorem~\ref{d190_0} that
			\begin{equation}
				\|b \cdot s\|_{Y_{w}(X)} \leq \|b\|_{\ell_{\infty}\pare{\K}} \norm{s}_{Y_{w}(X)}, \qquad \forall s\in Y_{w}(X),\qquad \forall b\in \ell_{\infty}\pare{\K}.
			\end{equation} 
			This means that $Y_{w}(X)$ is a   $BK$ $b$-space, as is $Y(X)$.  Furthermore,
			\begin{equation}
				Y(X) \subseteq Y_w(X) \quad \text{and}\quad 	\|s\|_{Y_w(X)} \leq \|s\|_{Y(X)}, \qquad s\in Y(X). 
			\end{equation} 
		Note that if $Y$ does not vanish at any point and $X\neq \llave{0}$, then neither does $Y_{w}\pare{X}$.

		\begin{example} The spaces $\ell_{p,w}\pare{X}$ and $c_{0,w}\pare{X}$ \label{holamundo3b}

			Let $X$ be a Banach space, $1 \leq p \leq \infty$, and $Y = \ell_{p}\pare{\K}$. In this case, instead of $Y_{w}\pare{X}$, we use the notation $\ell_{p,w}\pare{X}$(See \cite[p. 32]{DJT}). Thus,
			\begin{equation}
				\ell_{p,w}\pare{X} := \left\{s \in S\pare{X} : \langle s, x^{*}\rangle \in \ell_{p}\pare{\K}, \forall x^{*} \in X^{*}\right\}
			\end{equation}
			is a $BK$-space with its corresponding operator norm 
			\begin{align}
				\norm{s}_{p,w}: 
				& = \left \{ \begin{matrix}  \sup\llave{\pare{\Sum_{n=1}^{\infty}\abs{\la s\pare{n},x^{*}\ra}^{p}}^{1/p}:x^{*}\in \B_{X^{*}}}, \quad   &\mbox{if }1\leq p<\infty,
					\vspace{.5 cm}\\ 
					\sup\llave{\abs{\la s\pare{n},x^{*}\ra}:n\in \N,\  x^{*}\in \B_{X^{*}}},\quad   \quad \quad&\mbox{if } p=\infty\end{matrix}\right..
			\end{align}

			Clearly, $c_{00}\pare{X} \subseteq \ell_{p,w}\pare{X}$ for $1 \leq p \leq \infty$. The elements of $\ell_{1,w}\pare{X}$ are called \emph{weakly absolutely summable sequences}.

			Similarly, when $Y = c_{0}\pare{\K}$, instead of $Y_{w}\pare{X}$ we will use $c_{0,w}\pare{X}$. Thus,
			\begin{equation}
				c_{0,w}\pare{X} = \left\{s \in S\pare{X} : \langle s, x^{*}\rangle \in c_{0}\pare{\K}, \forall x^{*} \in X^{*}\right\}
			\end{equation}
			is a $BK$-space with its corresponding operator norm
			\begin{equation}
				\norm{s}_{w} := \sup\left\{\norm{\langle s, x^{*}\rangle}_{\ell_{\infty}\pare{\K}} : x^{*} \in \B_{X^{*}}\right\}.
			\end{equation}
			Clearly, $c_{00}\pare{X} \subseteq c_{0,w}\pare{X}$. 
		\end{example}  \vs

		 Since $Y(X) \subseteq Y_w(X)$, the question arises as to whether $Y(X)$ and $Y_w(X)$ can be equal. For example, we know that in general $\ell_{\infty,w}\pare{X}=\ell_{\infty}\pare{X}$ (see \cite[p.~33]{DJT}). On the other hand, from the work of Dvoretzky-Rogers and A. Pietsch we have the classical result that, for $1 \leq p < \infty$, $\ell_{p,w}(X) = \ell_p(X)$ if and only if $X$ has finite dimension (see \cite[p.~50]{DJT}). Next, we show that when $X$ has finite dimension, we always have $Y(X) = Y_w(X)$.
		 
%		 	The following classical result is well known. The case $p=1$ is due to Dvoretzky--Rogers, while the general case follows from the work of Pietsch (see \cite[p.~50]{DJT}).
		 
%		 \begin{theorem}\label{final4}
%		 	Let $1\leq p<\infty$. Then $\ell_{p,w}\pare{X}= \ell_{p}\pare{X}$ if and only if $X$ is finite-dimensional.
%		 \end{theorem}
		 
		 \begin{proposition}\label{efren4}
		 	Let $Y\subseteq S\pare{\K}$ be a Banach ideal. If $X$ is a finite-dimensional Banach space, then
		 	\begin{equation}
		 		Y_{w}\pare{X}=Y\pare{X}.
		 	\end{equation}
		 	In this case,
		 	\begin{equation}\label{final5}
		 		\norm{s}_{Y_{w}\pare{X}}\leq \norm{s}_{Y\pare{X}}\leq \pare{\dim X}\norm{s}_{Y_{w}\pare{X}},\qquad \forall s\in Y_{w}\pare{X}.
		 	\end{equation}
		 \end{proposition}
		 
		 \begin{proof}
		 	Let $n:=\dim X$. If $X=\llave{0}$, then clearly $Y_{w}\pare{X}=\llave{0}=Y\pare{X}$. Assume therefore that $X\neq \llave{0}$. We have already seen that $Y\pare{X}\subseteq Y_{w}\pare{X}$ and the first inequality in (\ref{final5}) is valid.
		 	
		 	By a well known theorem of H.~Auerbach (see, for example, \cite[Thm.~II.2.2]{S}), there is a basis $\llave{x_{1},\dots,x_{n}}$ of $X$ and $\llave{x_{1}^{*},\dots,x_{n}^{*}}$ a basis of $X^{*}$ satisfying $\la x_{i},x_{j}^{*}\ra=1$ if $i=j$ and $\la x_{i},x_{j}^{*}\ra=0$ if $i\neq j$. Moreover, $\norm{x_{i}}=\norm{x_{i}^{*}}=1$ for all $i\in\llave{1,\dots,n}$.
		 	
		 	Fix $s\in Y_{w}\pare{X}$ and $k\in\N$. Choose $x^{*}\in\B_{X^{*}}$ such that $\norm{s\pare{k}}=\abs{\la s\pare{k},x^{*}\ra}$. Write $x^{*}=a_{1}x_{1}^{*}+\cdots+a_{n}x_{n}^{*}$ with $a_{1},\dots,a_{n}\in\K$. Then $\abs{a_{j}}=\abs{\la x_{j},x^{*}\ra}\leq \norm{x_{j}}\norm{x^{*}}\leq 1$ for all $j\in\llave{1,\dots,n}$. Hence
		 	\begin{equation}\label{d213}
		 		\norm{s\pare{k}}=\abs{\sum_{j=1}^{n}a_{j}\la s\pare{k},x_{j}^{*}\ra}\leq \sum_{j=1}^{n}\abs{\la s\pare{k},x_{j}^{*}\ra}.
		 	\end{equation}
		 	
		 	Fix $m\in\llave{1,\dots,n}$. Since $Y$ is an ideal, $\abs{\llave{\la s\pare{k},x_{m}^{*}\ra}_{k}}=\llave{\abs{\la s\pare{k},x_{m}^{*}\ra}}_{k}\in Y$. Therefore,
		 	$
		 	z:=\llave{\abs{\la s\pare{k},x_{1}^{*}\ra}+\cdots+\abs{\la s\pare{k},x_{n}^{*}\ra}}_{k}\in Y.
		 	$
		 	Using (\ref{d213}), we obtain $\llave{\norm{s\pare{k}}}_{k}\leq z$, and hence $\llave{\norm{s\pare{k}}}_{k}\in Y$. Therefore, $s\in Y\pare{X}$, which proves that $Y_{w}\pare{X}=Y\pare{X}$.
		 	
		 	Applying (\ref{d213}) once again, for every $s\in Y\pare{X}$ we obtain
		 	\begin{equation}
		 		\norm{s}_{Y\pare{X}}
		 		 \leq \norm{\abs{\la s,x_{1}^{*}\ra}+\cdots+\abs{\la s,x_{n}^{*}\ra}}_{Y} \leq \norm{x_{1}^{*}}\norm{s}_{Y_{w}\pare{X}}+\cdots+\norm{x_{n}^{*}}\norm{s}_{Y_{w}\pare{X}} = n\norm{s}_{Y_{w}\pare{X}}. 
		 	\end{equation}
		 \end{proof}
		
	 \begin{example}
	 	Let $Y\subseteq S\pare{\K}$ be a Banach ideal. Then $M(Y(X),Y_w (X))=\ell_{\infty}\pare{\K}$.
	 \end{example}

 \subsection{$(Z,Y)$-Summing Operator Spaces}
 
 Let $Y,Z\subseteq S(\K)$ be Banach ideals such that $Y$ does not vanish at any point, and let $V$ and $W$ be Banach spaces. Since $Y_w(V)$ and $Z(W)$ are $\BK$-spaces and $c_{00}(V)\subseteq Y_w(V)$, Theorem~\ref{efren10} yields that
 \begin{equation}\label{naty19}
 	M_{\mathcal L}(Y_w(V),Z(W))
 	=\llave{T\in\mathcal L(V,W):\widetilde{T}v\in Z(W),\ \forall v\in Y_w(V)},
 \end{equation}
 endowed with the (operator) norm
 \begin{equation}
 	\norm{T}_{M_{\mathcal L}(Y_w(V),Z(W))}
 	=\sup\llave{\norm{\llave{\norm{Tv(n)}_W}_n}_{Z}:v\in\B_{Y_w(V)}},
 \end{equation}
 is a $\BK$-space.
 
 The elements of $M_{\mathcal L}(Y_w(V),Z(W))$ will be called \emph{$(Z,Y)$-summing operators}. When $Y=Z$, they are simply called \emph{$Y$-summing operators} (see \cite[Sec.~6]{HG}).

 \begin{example} \textbf{$(q,p)$-Summing Operators}
 	
 	Let $1\le p,q<\infty$ and let $X,W$ be Banach spaces. Taking $Z=\ell_{p,w}(X)$, $Y=\ell_q(W)$, and since $c_{00}(X)\subseteq \ell_{p,w}(X)$, by Theorem~\ref{efren10} we have that
 	\begin{equation}
 		M_{\mathcal L}(\ell_{p,w}(X),\ell_q(W))
 		=\llave{T\in\mathcal L(X,W):\widetilde{T}s\in\ell_q(W),\ \forall s\in\ell_{p,w}(X)},
 	\end{equation}
 	with norm
 	\begin{equation}
 		\norm{T}_{M_{\mathcal L}(\ell_{p,w}(X),\ell_q(W))}
 		=\sup\llave{\left(\sum_{n=1}^{\infty}\norm{T(s(n))}^{q}\right)^{1/q}:s\in\B_{\ell_{p,w}(X)}},
 	\end{equation}
 	is a $\BK$-space. The space $M_{\mathcal L}(\ell_{p,w}(X),\ell_q(W))$ is usually denoted by $\Pi_{q,p}(X,W)$, and its elements are known as \emph{$(q,p)$-summing linear operators}  (see, e.g., \cite{P}, \cite[Ch.~10]{DJT}).
 \end{example}

\subsection{b-ideal  Part of $\Sigma \ell_{\infty}(X)$} 

In \eqref{d264_a} and Example~\ref{d333} we saw that $\Sigma \ell_{\infty}(X)$ is a $\BK$-space and that it is not a  $b$-ideal. Then, by Theorem~\ref{d334} , 
\begin{equation}
	\textup{Ip}\big(\Sigma \ell_{\infty}(X)\big)
	=
	\Big\{ s \in \Sigma \ell_{\infty}(X) : t \cdot s \in \Sigma \ell_{\infty}(X),\ \forall t \in \ell_{\infty}(\K) \Big\}
	\subsetneq \Sigma \ell_{\infty}(X)
\end{equation}
is a $\BKa$-space, endowed with the (operator) norm  given by
\begin{equation}
	\|s\|_{\textup{Ip}(\Sigma \ell_{\infty}(X))}
	=
	\sup\llave{ \norm{ \sum_{n=1}^{m} t(n)s(n) } : t \in \B_{\ell_{\infty}(\K)},\ m \in \N }.
\end{equation}

The following result was proved by Bessaga and Pełczyński \cite[Lemma~2]{BP}.

\begin{theorem}
	Let $s \in S(X)$. Then $s \in \ell_{1,w}(X)$ if and only if there exists a constant $c \in (0,\infty)$ such that
\begin{equation}
	\sup\llave{ \norm{ \sum_{n=1}^{m} t(n)s(n) } : m \in \N }
	\le c \|t\|_{\infty},
	\qquad \forall t \in \ell_{\infty}(\K).
\end{equation}
\end{theorem}

In terms of multipliers,  this is equivalent to
\begin{equation}\label{final10}
	M'\big(\ell_{\infty}(\K), \Sigma \ell_{\infty}(X)\big)
	=
	\ell_{1,w}(X),
\end{equation}
and hence
\begin{equation}
	\textup{Ip}\big(\Sigma \ell_{\infty}(X)\big)
	=
	\ell_{1,w}(X).
\end{equation}
Next, we show that these spaces actually have the same norm.

\begin{proposition}\label{naty3}\label{lola215_0}
	\[
	\textup{Ip}\big(\Sigma \ell_{\infty}(X)\big) \equiv \ell_{1,w}(X).
	\]
\end{proposition}

\begin{proof}
	Let $s \in \textup{Ip}\big(\Sigma \ell_{\infty}(X)\big) = \ell_{1,w}(X)$. Fix $x^{*} \in \B_{X^{*}}$, and define $t \in S(\K)$ by
	\begin{equation}
	t(n) :=
	\begin{cases}
		\dfrac{|\langle s(n), x^{*} \rangle|}{\langle s(n), x^{*} \rangle},
		& \text{if } \langle s(n), x^{*} \rangle \neq 0, \\[.2cm]
		0, & \text{if } \langle s(n), x^{*} \rangle = 0,
	\end{cases}
	\qquad \forall n \in \N. 
\end{equation}
	Note that $t \in \B_{\ell_{\infty}(\K)}$. For $m \in \N$, we have
	\begin{equation}
		\sum_{n=1}^{m} |\langle s(n), x^{*} \rangle|
	 =
		\sum_{n=1}^{m} t(n)\langle s(n), x^{*} \rangle \le
		\norm{\sum_{n=1}^{m} t(n)s(n) }
		\le
		\|s\|_{\textup{Ip}(\Sigma \ell_{\infty}(X))}.
	\end{equation}
	Letting $m \to \infty$ and then taking the supremum over $x^{*} \in \B_{X^{*}}$, we obtain
\begin{equation}
	\|s\|_{\ell_{1,w}(X)}
	\le
	\|s\|_{\textup{Ip}(\Sigma \ell_{\infty}(X))}.
\end{equation}
	
	Conversely, for $t \in \B_{\ell_{\infty}(\K)}$ and $m \in \N$, by the Hahn--Banach theorem there exists $x^{*} \in \B_{X^{*}}$ such that
	\begin{equation}
		\norm{ \sum_{n=1}^{m} t(n)s(n)}
		 =
		  \abs{ \ral{  \sum_{n=1}^{m} t(n)s(n), x^{*} } } \le
		\sum_{n=1}^{\infty} |\langle s(n), x^{*} \rangle|
		\le
		\|s\|_{\ell_{1,w}(X)}.
	\end{equation}
	Taking the supremum over $m \in \N$ and $t \in \B_{\ell_{\infty}(\K)}$, we conclude that
	\[
	\|s\|_{\textup{Ip}(\Sigma \ell_{\infty}(X))}
	\le
	\|s\|_{\ell_{1,w}(X)}.\qedhere
	\]
\end{proof}

Observe that $\ell_{u}(X)\subseteq \ell_{1,w}(X)$ and $\ell_{u}\pare{X}$ originally carries the ``$b$-ideal  part'' norm defined in \eqref{lola109}. In view of the previous proposition, this norm coincides with the subspace norm inherited from $\ell_{1,w}(X)$.  A well-known and important result states that $\ell_{1,w} (X) = \ell_u (X)$ if  and only if $X$ does not contain a copy of $c_0$  (\cite[Thm. 5]{BP}).
Next, we establish a general relationship between $\ell_{1,w} (X)$ and $\ell_u (X)$.

\begin{lemma}\label{final14_0}
Let $\llave{x_{n}}_{n}\subseteq X$. Then $\llave{x_{n}}_{n}\in \ell_{u}\pare{X}$ if and only if, given $\epsilon>0$, there exists a finite set  $A\subseteq \N$ such that
	\begin{equation}\label{final13}
		\norm{\Sum_{n\in N}b_{n}x_{n}}\leq \epsilon,\qquad \forall \llave{b_{n}}_{n}\in \B_{\ell_{\infty}\pare{\K}},\qquad  \text{for each finite set } N\subseteq \N\setminus A.
	\end{equation}
\end{lemma}
 \begin{proof} Suppose first that $\llave{x_{n}}_{n}\in \ell_{u}\pare{X}$.	Let $\epsilon>0$. From \cite[p. 959]{Hi}, there exists a finite set $A\subseteq \N$ such that
 	\begin{equation}\label{final12}
 		\norm{\Sum_{n\in N}x_{n}}<\frac{\epsilon}{4},\qquad \text{for each finite set } N\subseteq \N\setminus A.
 	\end{equation}
 	Fix $b\in \B_{\ell_{\infty}\pare{\K}}$, a finite set $N\subseteq \N\setminus A$ and choose $x^{*}\in \B_{X^{*}}$ such that
 	\begin{equation}
 		\norm{\Sum_{n\in N} b_{n}x_{n}}=\la \Sum_{n\in N} b_{n}x_{n}, x^{*}\ra =\re{\Sum_{n\in N} b_{n}\la x_{n},x^{*}\ra}.
 	\end{equation}
 	Define the following sets
 	\begin{align*}
 		P_{r}:=\llave{n\in N:\re{\la x_{n},x^{*}\ra}\geq 0}&,\qquad P_{i}:=\llave{n\in N:\im{\la x_{n},x^{*}\ra}\geq 0},\\
 		N_{r}:=\llave{n\in N:\re{\la x_{n},x^{*}\ra}< 0}&,\qquad N_{i}:=\llave{n\in N:\im{\la x_{n},x^{*}\ra}< 0}.
 	\end{align*}  Then
 	\begin{align*}
 		\re{\Sum_{n\in N} b_{n}\la x_{n},x^{*}\ra}&=\Sum_{n\in N} \re{b_{n}\la x_{n},x^{*}\ra}=\Sum_{n\in N}\re{b_{n}}\re{\la x_{n},x^{*}\ra}-\Sum_{n\in N}\im{b_{n}}\im{\la x_{n},x^{*}\ra}\\
 		&\leq \Sum_{n\in N} \abs{\re{\la x_{n},x^{*}\ra}}+\Sum_{n\in N} \abs{\im{\la x_{n},x^{*}\ra}}\\
 		&= \Sum_{n\in P_{r}}\re{\la x_{n},x^{*}\ra} - \Sum_{n\in N_{r}}\re{\la x_{n},x^{*}\ra} + \Sum_{n\in P_{i}}\im{\la x_{n},x^{*}\ra} - \Sum_{n\in N_{i}}\im{\la x_{n},x^{*}\ra}\\
 		&\leq \norm{\Sum_{n\in P_{r}}x_{n}} + \norm{\Sum_{n\in N_{r}}x_{n}}+\norm{\Sum_{n\in P_{i}}x_{n}} + \norm{\Sum_{n\in N_{i}}x_{n}}.
 	\end{align*}
 	Applying \eqref{final12}, we obtain  that $\abs{\re{\Sum_{n\in N} b_{n}\la x_{n},x^{*}\ra}}\leq \epsilon$.  
 	
 	Furthermore, let $\{x_{n_k}\}_{k}$ be a subsequence of $\llave{x_{n}}_{n}$, and fix $\epsilon>0$. By hypothesis, there exists a finite set $A\subseteq \N$ satisfying \eqref{final13}. Taking $b=\llave{1}_{n}$ and $K=\max\llave{n_{k}:k\in A}$, we obtain
 	\begin{equation}
 		\norm{\Sum_{k= \ell}^{m}x_{n_{k}}}\leq \epsilon,\qquad \forall \ell \geq m>K.
 	\end{equation}
 	Letting $m\tien \infty$, it follows that $\norm{\Sum_{k= \ell}^{\infty}x_{n_{k}}}\leq \epsilon, \forall \ell \geq m.$
 	
 	Thus, each subsequence $\{x_{n_k}\}_{k}\subseteq \llave{x_{n}}_{n}$ is summable. By \cite[p. 959]{Hi}, it follows that $\llave{x_{n}}_{n}$ is unconditionally summable.
 \end{proof}

\begin{proposition}
	$\ell_{1,w}\pare{X}_{0}= \ell_{u}\pare{X}$.
\end{proposition}
\begin{proof}
	By Propositions \ref{naty3} and \ref{chuy18_0_1_t_0}, it follows that
	\begin{equation}
		\ell_{1,w}\pare{X}_{0}= M'\pare{\ell_{\infty}\pare{\K},\Sigma \ell_{\infty}\pare{X}}_{0}\subseteq M'\pare{\ell_{\infty}\pare{\K},\Sigma c\pare{X}}=\ell_{u}\pare{X}.
	\end{equation}
	
	Let $s=\llave{x_{n}}_{n}\in \ell_{u}\pare{X}$ and $\epsilon>0$. Using Lemma \ref{final14_0} , there exists $A\subseteq \N$ such that
	\begin{equation}
		\norm{\Sum_{n\in N}b_{n}x_{n}}\leq\epsilon,\qquad \forall N\subseteq \N\setminus A,\qquad \forall \llave{b_{n}}_{n}\in \B_{\ell_{\infty}\pare{\K}}.
	\end{equation}
	
	Now take $x^{*}\in \B_{X^{*}}$ and define $b_{n}:=0$ if $\la x_{n},x^{*}\ra=0$ and $b_{n}:=\frac{\overline{\la x_{n},x^{*}\ra}}{\abs{\la x_{n},x^{*}\ra}}$ if $\la x_{n},x^{*}\ra\neq 0$, $\forall n\in \N$. Clearly $b:=\llave{b_{n}}_{n}\in \B_{\ell_{\infty}\pare{\K}}$, and taking $M:=\max\llave{n:n\in A}$, it follows that
	\begin{equation}
		\Sum_{n=k+1}^{m}\abs{\la x_{n},x^{*}\ra}=\Sum_{n=k+1}^{m}b_{n}\la x_{n},x^{*}\ra=\la\Sum_{n=k+1}^{m} b_{n}x_{n},x^{*}\ra \leq \norm{\Sum_{n=k+1}^{m} b_{n}x_{n}}\leq \epsilon,\qquad \forall m> k> M.
	\end{equation}
	Letting $m\to \infty$ and taking the supremum over $x^{*}\in \B_{X^{*}}$, we obtain
	\begin{equation}
		\norm{P_{k}\pare{s}-s}_{\ell_{1,w}\pare{X}} \leq \epsilon,\qquad \forall k>M.
	\end{equation}
\end{proof}
 
 \begin{corollary}\label{yaya4}
 	$\ell_{1,w}\pare{c_{0}\pare{\K}}$ is a  $b$-ideal that is not a  v-ideal.  
 \end{corollary}
 \begin{proof}
 	For simplicity, we use the notation $c_{0}= c_{0}\pare{\K}$ and $\ell_{\infty}=\ell_{\infty}\pare{\K}$. Theorem~\ref{naty3} shows that $\ell_{1,w}\pare{c_{0}}$ is a $b$-ideal. Consider the sequences  $s:=\llave{\frac{1}{n}e_{1}}_{n}$ and $z:=\llave{e_{n}}_{n}$ which are elements in $S\pare{c_{0}}$. 
 	Now observe that
 	$
 	\norm{\sum_{n=1}^{m}b\pare{n}e_{n}}_{c_{0}}\leq \norm{b}_{c_{0}},
 	$
 	for every $m\in \N$ and $b\in \ell_{\infty}$. Since
 	$
 	\text{Ip}\pare{\Sigma\ell_{\infty}\pare{c_{0}}}=\ell_{1,w}\pare{c_{0}}  ,
 	$
 	it follows that $z\in \ell_{1,w}\pare{c_{0}}$. 
 	Next, observe that
 	$
 	\J s\pare{n}=\norm{\frac{1}{n}e_{1}}_{c_{0}}=\frac{1}{n}\leq 1=\norm{e_{n}}_{c_{0}}=\J z\pare{n}.
 	$
 	Hence, $\J s \leq \J z$. However,
 	$
 	\norm{\sum_{n=1}^{m} s\pare{n}}_{c_{0}}
 	=
 	\norm{\sum_{n=1}^{m} \frac{1}{n}e_{1}}_{c_{0}}
 	=
 	\sum_{n=1}^{m}\frac{1}{n}\tien \infty.
 	$
 	This implies that $s\notin \ell_{1,w}\pare{c_{0}}$, and therefore $\ell_{1,w}\pare{c_{0}}$ is not a v-ideal. 
 \end{proof}

 \section{Spaces of Sequences of Bounded Variation}
 
Next, we study the space $bv\pare{X}$ consisting of vector sequences of bounded variation and present the weak vectorialization of $bv\pare{\K}$. We will see that these spaces describe familiar multiplier spaces.\vs
 
Let $s=\llave{x_{n}}_{n}\subseteq X$. The \emph{variation sequence of $s$} is $v_s\in S\pare{X}$ defined by
	\begin{equation}
	v_{s}(n) :=
	\begin{cases}
		x_{1},
		& \text{if } n=1, \\[.2cm]
		x_{n}-x_{n-1}, & \text{if } n>1.
	\end{cases} 
\end{equation} 

 \begin{definition}
 	Let $s=\llave{x_{n}}_{n}\subseteq X$. We say that $s$ has \emph{bounded variation} if $v_{s} \in \ell_1(X)$. That is, if
 	\begin{equation}
 			\sum_{n=1}^{\infty}\norm{x_{n+1}-x_{n}}<\infty.
 	\end{equation} 
 \end{definition}
 	The set of sequences of bounded variation will be denoted by $bv(X)$.  Clearly, $bv\pare{X}$ is a vector space.
 
 \begin{lemma}\label{final5_0}
 	$bv\pare{X}\subseteq c\pare{X}\subseteq  \ell_{\infty}(X)$.  
 \end{lemma}
 \begin{proof}
 	Let $s=\llave{x_{n}}_{n}\in bv\pare{X}$. Then $v_{s}\in \ell_{1}\pare{X}\subseteq \Sigma c\pare{X}$ and $y:=\lim_{n\tien \infty}S_{n}\pare{v_{s}}=\lim_{n\tien \infty}x_{n}$.  
 \end{proof}

	\begin{example}
		If $X\neq \llave{0}$ is a Banach space, then $bv (X)$ is not a b-ideal.
	\end{example}
	\begin{proof}
Take $b:=\llave{\sum_{k=1}^{n}\frac{(-1)^{k}}{k}}_{n}\in \ell_{\infty}\pare{\K}$ and $s:=\llave{x}_{n}\in bv\pare{X}$, where $\norm{x}=1$. Consider $z:=b\cdot s=\llave{\sum_{k=1}^{n}\frac{(-1)^{k}}{k}x}_{n}$. Then
\begin{equation}
	\norm{z\pare{n+1}-z\pare{n}} = \abs{\frac{(-1)^{n+1}}{n+1}} = \frac{1}{n+1}.
\end{equation}
Hence $b\cdot s\notin bv\pare{X}$.
	\end{proof}
	
	 Define next the operator
	$J:bv\pare{X}\to \ell_{1}\pare{X}$ by
	\begin{equation}
		J(s):=v_{s}.
	\end{equation}
	Proceeding directly, we conclude that $J$ is an injective operator.
	
	Take $z\in \ell_{1}\pare{X}$, and define $s:=\sigma_{{z}}$. Observe that
	\begin{equation}
		\Sum_{n=1}^{\infty}\norm{s\pare{n+1}-s\pare{n}}=\Sum_{n=1}^{\infty}\norm{S_{n+1}\pare{s}-S_{n}\pare{s}}= \Sum_{n=1}^{\infty}\norm{z\pare{n+1}}<\infty.
	\end{equation}
	Hence, $s\in bv\pare{X}$. Moreover, $J\pare{s}=v_{s}=v\pare{\sigma_{z}}=z$. Thus, $J$ is surjective. 
 
Since $J$ is an isomorphism, we now consider on $bv\pare{X}$ the norm $\| \cdot \|_{bv\pare{X}}$ induced by $J$, that is 
\begin{equation}
	\|s\|_{bv\pare{X}} := \|J s\|_{\ell_{1}\pare{X}}=\norm{x_{1}}+\Sum_{n=1}^{\infty}\norm{x_{n+1}-x_{n}},\qquad    \forall s=\llave{x_{n}}_{n}\in bv\pare{X}.
\end{equation}
In this way, $J: bv\pare{X} \to \ell_{1}\pare{X}$ is an isometric isomorphism and thus $bv (X)$ is a Banach space. Let $s=\llave{x_{n}}_{n}\in bv\pare{X}$. Then $\norm{x_{1}}\leq \norm{s}_{bv\pare{X}}$, and
\begin{equation}\label{fin2}
	\norm{x_{n}}=\norm{x_{1}+\Sum_{k=1}^{n-1}\pare{x_{k+1}-x_{k}}}\leq \norm{x_{1}}+\Sum_{k=1}^{n-1}\norm{x_{k+1}-x_{k}}\leq \norm{s}_{bv\pare{X}},\qquad \forall n\in \N\setminus \llave{1}.
\end{equation}
Hence, $bv\pare{X}$ is a BK-space.  

 Note that   $e_{n}x\in bv\pare{X}, \forall n\in \N, \forall x\in X$. Therefore,  if $X\neq \llave{0}$, then $bv\pare{X}$ is a BKN-space.

\vs

  Taking $U = bv(\K)$, we now define 
	\begin{equation}
		bv_{w}\pare{X}:=U_{w}\pare{\N,X}.
	\end{equation} 
	In this case, its operator norm is
	\begin{equation}
		\norm{s}_{bv_{w}\pare{X}}=\sup\llave{\abs{\la s\pare{1},x^{*}\ra}+ \Sum_{n=1}^{\infty}\abs{\la s\pare{n}-s\pare{n+1},x^{*}\ra}  :x^{*}\in \B_{X^{*}}},\qquad \forall s \in bv_{w}\pare{X}.
	\end{equation}
	 Applying Theorem \ref{d190_0}, we obtain that $bv_{w}\pare{X}$ is a BK-space. We also have $c_{00} (X) \subseteq  bv_w (X)$. Thus, if $X\neq \llave{0}$, then $bv_{w}\pare{X}$ is a BKN-space. Clearly $bv (X)\subseteq bv_w (X)$.

\vs

Proposition \ref{efren4} does not apply to $bv(\K)$ since this BK-space is not an ideal.
However, the next result   holds.

\begin{proposition}
Let $X$ be a Banach space. Then $bv_w\pare{X}=bv\pare{X}$ if and only if $X$ is finite-dimensional.
\end{proposition}
\begin{proof}
	First suppose that $bv_w\pare{X}=bv\pare{X}$. To prove that $X$ is finite-dimensional,  it is enough to show that $\ell_{1,w}\pare{X}\subseteq \ell_{1}\pare{X}$. Let $s =\llave{x_{n}}_{n}\in \ell_{1,w}\pare{X}$.  
	Define $u_{1}:=0$ and $u_{n+1}:=u_{n}-x_{n}$, $\forall n\in \N$. Then
	\begin{equation}
		\sum_{n=1}^{\infty} |\langle u_{n}-u_{n+1},x^{*}\rangle|
		=
		\sum_{n=1}^{\infty} |\langle x_{n},x^{*}\rangle|,
		\qquad
		\forall x^{*}\in \B_{X^{*}}.
	\end{equation}
	Hence, $u\in bv_w\pare{X}=bv\pare{X}$. Therefore,
	\begin{equation}
		\sum_{n=1}^{\infty}\norm{x_{n}}
		=
		\sum_{n=1}^{\infty}\norm{u_{n}-u_{n+1}}
		<
		\infty.
	\end{equation}
	Thus, $s\in \ell_{1}\pare{X}$.
	
	The converse implication is proved similarly.
\end{proof}
 
\begin{example}
	$bv_w(\K^{n})=bv(\K^{n})$.
\end{example}

For scalar sequences $\{x_k\}_{k}\subseteq \K$ and $\{y_k\}_{k}\subseteq \K$, Abel's identity (\cite[(1.2.9)]{Wil}) states that
\begin{equation}\label{final3}
	y_{1}x_{1}+\Sum_{k=2}^{n}\pare{y_{k}-y_{k-1}}x_{k}
	=
	\Sum_{k=1}^{n-1}y_{k}\pare{x_{k}-x_{k+1}} + y_{n}x_{n}.
\end{equation}
The same operations used in its proof shows that (\ref{final3}) remains valid when $\{x_k\}_{k}\subseteq X$. In this context, we shall employ this identity repeatedly throughout the sequel.

\begin{lemma} 
Let $X$ be a Banach space. Then: \begin{enumerate} \item[i.] $bv_w\pare{X}\subseteq M'\pare{\Sigma \ell_{\infty}\pare{\K},\Sigma \ell_{\infty}\pare{X}}\subseteq M'\pare{\Sigma c\pare{\K},\Sigma \ell_{\infty}\pare{X}}$. \item[ii.] $bv\pare{\K}\subseteq M\pare{\Sigma \ell_{\infty}\pare{X},\Sigma \ell_{\infty}\pare{X}}\subseteq M\pare{\Sigma c\pare{X},\Sigma \ell_{\infty}\pare{X}}$. \end{enumerate} \end{lemma}

\begin{proof}
	Observe that the second inclusion  in both parts follow directly from Lemma~\ref{q10_0_b} .
	
	i. Let $s=\llave{x_{k}}_{k}\in bv_w\pare{X}$ and $t=\llave{t_{k}}_{k}\in \Sigma \ell_{\infty}\pare{\K}$. Consider $y=\sigma_{t}\in\ell_{\infty}\pare{\K}$. Then $t_{1}=y\pare{1}$ and $t_{k}=y\pare{k}-y\pare{k-1}$, $\forall k\in \N$. Take $n \in \N$. Using (\ref{final3}) and the Hahn--Banach theorem, there exists $x^{*}\in \B_{X^{*}}$ such that
	\begin{align*}
		\norm{\Sum_{k=1}^{n}t_{k}x_{k}}
		&=
		\norm{y\pare{1}x_{1}+\Sum_{k=2}^{n}\pare{y\pare{k}-y\pare{k-1}}x_{k}}
		=
		\la
		\Sum_{k=1}^{n-1} y\pare{k}\pare{x_{k}-x_{k+1}}
		+y\pare{n}x_{n},
		x^{*}
		\ra
		\nonumber\\
		&\leq
		\Sum_{k=1}^{n-1}
		\abs{\la x_{k}-x_{k+1},x^{*}\ra}
		\abs{\sigma_{t}\pare{k}}
		+
		\norm{x_{n}}
		\abs{\sigma_{t}\pare{n}}
		\leq
		\norm{s}_{bv\pare{X}}\norm{\sigma_{t}}_{\ell_{\infty}\pare{\K}}
		+
		\norm{s}_{\ell_{\infty}\pare{X}}
		\norm{\sigma_{t}}_{\ell_{\infty}\pare{\K}}.
	\end{align*}
	Thus, $t\cdot s\in\Sigma \ell_{\infty}\pare{X}$, and therefore $s\in M'\pare{\Sigma \ell_{\infty}\pare{\K},\Sigma \ell_{\infty}\pare{X}}$.
	
	ii. The proof is analogous to that of~i.
\end{proof}

	\begin{lemma}
		 Let $X$ be a Banach space. Then:
		\begin{enumerate}
			\item[i.] $M'\pare{\Sigma c\pare{\K},\Sigma \ell_{\infty}\pare{X}}
			\subseteq 
			bv_w(X).$
			\item[ii.] $M\pare{\Sigma c\pare{X},\Sigma \ell_{\infty}\pare{X}}
			\subseteq 
			bv\pare{\K}$.
		\end{enumerate} 
	\end{lemma} 
	
	\begin{proof}
		i. Let $s=\llave{x_{k}}_{k}\in M'\pare{\Sigma c\pare{\K},\Sigma \ell_{\infty}\pare{X}}$. Since $\pare{\Sigma \ell_{\infty}\pare{X}}_{0}=\Sigma c\pare{\K}$ and $\frac{\norm{x}_{X}\norm{e_{n}}_{\Sigma c\pare{\K}}}{\norm{e_{n}x}_{\Sigma \ell_{\infty}\pare{X}}}=1$, $\forall n\in \N$ and $\forall x\in X\setminus\llave{0}$, Propositions~\ref{chuy18_0_1_t_0} and \ref{final2} imply that $s\in \ell_{\infty}\pare{X}$. Given $y=\llave{y_{n}}_{n}\in c_{0}\pare{\K}$, define $z\pare{k}:=y_{k}-y_{k+1}$, $\forall k\in \N$. Clearly, $z\in \Sigma c\pare{\K}$, and therefore $z\cdot s\in \Sigma \ell_{\infty}\pare{X}$. Now let $u\in S\pare{X}$ be defined by $u\pare{k}:=x_{k}-x_{k+1}$, $\forall k\in \N$.  Using Abel's inequality, for $n\geq 2$, we obtain
		\begin{align}
			\norm{\sum_{k=1}^{n-1}y_{k}u\pare{k}}
			&=
			\norm{\sum_{k=1}^{n-1}y_{k}\pare{x_{k}-x_{k+1}}}
			=
			\norm{y_{1}x_{1}+\sum_{k=2}^{n}\pare{y_{k}-y_{k+1}}x_{k}-y_{n}x_{n}}
			\nonumber\\
			&\leq
			\norm{y_{1}x_{1}}
			+
			\norm{\sum_{k=2}^{n}z\pare{k}x_{k}}
			+
			\norm{y_{n}x_{n}}
			\leq
			\norm{z\cdot s}_{\Sigma \ell_{\infty}\pare{X}}
			+
			2\norm{y}_{\ell_{\infty}\pare{\K}}\norm{s}_{\ell_{\infty}\pare{X}}.
		\end{align}
		Then $y\cdot u\in \Sigma \ell_{\infty}\pare{X}$ for every $y\in c_{0}\pare{\K}$. Since $c_{0}\pare{\K}$ is a Banach ideal, Corollary~\ref{d332_3} implies now that
		\begin{equation}
		u\in M'\pare{c_{0}\pare{\K},\Sigma \ell_{\infty}\pare{X}}
			= M'\pare{c_{0}\pare{\K},\text{Ip}\pare{\Sigma \ell_{\infty}\pare{X}}}
			= M'\pare{c_{0}\pare{\K},\ell_{1,w}\pare{X}}.
		\end{equation}
		Therefore, for $m\in \N$, we have 
		\begin{equation}
			\Sum_{k=1}^{m}\abs{\la x_{k}-x_{k+1},x^{*}\ra}
			\leq
			\sup\llave{\Sum_{k=1}^{\infty}\abs{\la t(k)u\pare{k},z^{*}\ra}: t\in \B_{c_{0}\pare{\K}},\, z^{*}\in \B_{X^{*}}}
			<
			\infty,
			\qquad \forall x^{*}\in \B_{X^{*}}.
		\end{equation}
		Hence, $s\in bv_w\pare{X}$.
		
		ii. The proof is analogous to that of~i, using that
		$
		\ell_{1,w}\pare{\K}=\ell_{1}\pare{\K}.
		$
	\end{proof}

	The following result is obtained directly from the two previous lemmas. The case $X = \K$ is already known (\cite[Thm. 7.3.5]{Wil}).
	\begin{proposition}\label{final9}
		Let $X$ be a Banach space. Then:
		\begin{enumerate}
			\item[i.] $M'\pare{\Sigma \ell_{\infty}\pare{\K},\Sigma \ell_{\infty}\pare{X}}
			=
			bv_w(X)=M'\pare{\Sigma c\pare{\K},\Sigma c\pare{X}}.$
			\item[ii.] $M\pare{\Sigma \ell_{\infty}\pare{X},\Sigma \ell_{\infty}\pare{X}}
			=
			bv\pare{\K}=M\pare{\Sigma c\pare{X},\Sigma c\pare{X}}$.  
		\end{enumerate} 
	\end{proposition}

		\begin{corollary}
		If $X$ is finite-dimensional, then $M'\pare{\Sigma \ell_{\infty}\pare{\K},\Sigma \ell_{\infty}\pare{X}}
		=
		bv(X)=M'\pare{\Sigma c\pare{\K},\Sigma c\pare{X}}.$
	\end{corollary}

	\begin{proposition} 
		Let $X$ be a Banach space. Then:
		\begin{enumerate}
			\item[i.] $\textnormal{Ip}\hspace{-2pt}\pare{bv\pare{X}}=\ell_{1}\pare{X}$.
			\item[ii.] $\textnormal{Ip}\hspace{-2pt}\pare{bv_{w}\pare{X}}=\ell_{1,w}\pare{X}$.  
		\end{enumerate} 
	\end{proposition}
	\begin{proof} Since the proofs of i and ii are analogous, we will only prove ii.
		
  Let $s = \llave{x_{n}}_{n}\in \ell_{1,w}\pare{X}$ and $b=\llave{b_{n}}_{n}\in \B_{\ell_{\infty}\pare{\K}}$. Clearly,
\begin{equation}
	\Sum_{n=1}^{\infty}\abs{\la b_{n}x_{n}-b_{n+1}x_{n+1},x^{*}\ra}\leq \Sum_{n=1}^{\infty}\abs{\la x_{n},x^{*}\ra}+\Sum_{n=1}^{\infty}\abs{\la x_{n+1},x^{*}\ra}\leq 2 \norm{s}_{\ell_{1,w}\pare{X}},\qquad \forall x^{*}\in \B_{X^{*}}.
\end{equation}
Hence, $b\cdot s\in bv_{w}\pare{X}$ and
\begin{equation}
	\|s\|_{\textnormal{Ip}  \pare{bv_w (X)}} \leq 2 \|s\|_{\ell_{1,w} (X)}.
\end{equation}

Now let $s=\llave{x_{n}}_{n}\in \textnormal{Ip}\hspace{-2pt}\pare{bv_w\pare{X}}$. Consider the sequence $b_{1}\in \ell_{\infty}\pare{\K}$ defined by $b_{1}(n)=1$ if $n = 2(k-1)+1$ and $b_{1}(n)=0$ if $n = 2k$, for each $k\in \N$. Also, let $b_{-1}\in \ell_{\infty}\pare{\K}$ be the sequence defined by $b_{-1}(n)=1$ if $n = 2k$ and $b_{-1}(n)=0$ if $n = 2(k-1)+1$, for each $k\in \N$. Clearly, $b_{1},b_{-1}\in \B_{\ell_{\infty}\pare{\K}}$, and therefore $b_{1}\cdot s, b_{-1}\cdot s \in bv_w\pare{X}$. Consequently, for every $x^{*}\in \B_{X^{*}}$, we have
\begin{align*}
	2\Sum_{n=1}^{\infty}\abs{\la   x_{2\pare{n-1}+1},x^{*}\ra}&=\abs{\la b_{1}\pare{1}x_{1},x^{*}\ra}+ \Sum_{n=1}^{\infty}\abs{\la b_{1}\pare{n}x_{n}-b_{1}\pare{n+1}x_{n+1},x^{*}\ra}\leq \norm{s}_{\textnormal{Ip} \pare{bv_{w}\pare{X}}} \qquad \text{ and }\\
	2\Sum_{n=1}^{\infty}\abs{\la   x_{2n},x^{*}\ra} &=\abs{\la b_{-1}\pare{1}x_{1},x^{*}\ra}+ \Sum_{n=1}^{\infty}\abs{\la b_{-1}\pare{n}x_{n}-b_{-1}\pare{n+1}x_{n+1},x^{*}\ra}\leq \norm{s}_{\textnormal{Ip} \pare{bv_{w}\pare{X}}}.
\end{align*}
It follows that
\begin{equation}
	 \Sum_{n=1}^{\infty}\abs{\la   x_{n},x^{*}\ra}  \leq  \norm{s}_{\textnormal{Ip} \pare{bv_{w}\pare{X}}},\qquad \forall x^{*}\in \B_{X^{*}}.
\end{equation}
Thus, $s\in \ell_{1,w}\pare{X}$ and
\begin{equation}
	\|s\|_{\ell_{1,w} (X)}\leq	\|s\|_{\textnormal{Ip}  \pare{bv_w (X)}}.
\end{equation} 
	\end{proof}

Moreover, by Proposition \ref{final9} and Theorem \ref{sam1} together with \eqref{lola165} and \eqref{final10}, we obtain the following.
\begin{corollary}
	Let $X$ be a Banach space. Then
	\begin{enumerate}
		\item[i.] $\ell_{1,w}\pare{X}=M'\pare{\Sigma \ell_{\infty}\pare{\K},\ell_{1,w}\pare{X}}=M'\pare{\Sigma c\pare{\K},\ell_{u}\pare{X}}$.
		\item[ii.] $\ell_{1}\pare{\K}=M'\pare{\Sigma \ell_{\infty}\pare{X},\ell_{1,w}\pare{X}}=M'\pare{\Sigma c\pare{X},\ell_{u}\pare{X}}$.
	\end{enumerate}
\end{corollary}

\end{document}